\documentclass[11pt,letterpaper]{amsart}
\usepackage{format}

\newcommand \bfG{{\mathbf G}}
\newcommand \bfT{{\mathbf T}}

\newcommand \bfI{{\mathbf I}}
\newcommand \bfP{{\mathbf P}}
\newcommand \bfU{{\mathbf U}}
\newcommand \bfL{{\mathbf L}}

\newcommand \sfk{{\mathsf k}}
\newcommand \bark{{\bar{\mathsf k}}}

\newcommand \CX{{\mathcal X}}

\title[Wavefront Sets of Unipotent Representations, I]{Wavefront Sets of Unipotent Representations of Reductive $p$-adic Groups I}
\begin{document}

\maketitle

\begin{abstract}
   The wavefront set is a fundamental invariant arising from the Harish-Chandra-Howe local character expansion of an admissible representation.  We prove a precise formula for the wavefront set of an irreducible Iwahori-spherical representation with `real infinitesimal character' and determine a lower bound for this invariant in terms of the Deligne-Langlands-Lusztig parameters. 
In particular, for the Iwahori-spherical representations with real infinitesimal character, we deduce that the algebraic wavefront set is a singleton, as conjectured by M\oe glin and Waldspurger.

As a corollary, we obtain an explicit description of the wavefront set of an irreducible spherical representation with real Satake parameter.
\end{abstract}

\tableofcontents


\section{Introduction}


\subsection{The local Langlands classification}\label{sec:Arthur}
We begin with a brief overview of the local Langlands classification. 
Let $\mathsf k$ be a nonarchimedean local field of characteristic $0$ with ring of integers $\mathfrak o$, finite residue field $\mathbb F_q$ of cardinality $q$ and valuation $\mathsf{val}_{\mathsf k}$. Fix an algebraic closure $\bar{\mathsf k}$ of $\mathsf k$ and let $K\subset \bar{\mathsf k}$ be the maximal unramified extension of $\mathsf k$ in $\bar{\mathsf k}$. 

Let $W_{\mathsf k}$ be the Weil group of $\mathsf k$ \cite[(1.4)]{Tate1979}. Let $\mathbf{G}(\mathsf k)$ be the group of $\mathsf k$-rational points of a connected reductive algebraic group $\mathbf{G}$ defined and split over $\mathsf k$. Let $G^{\vee}$ denote the complex Langlands dual group associated to $\mathbf{G}$, see \cite[\S2.1]{Borel1979}. (Since we assume throughout that $\mathbf G$ is $\mathsf k$-split, we can work with $G^\vee$ in place of the L-group.) Let $W_{\mathsf k}'=W_{\mathsf k}\ltimes \mathbb C$ denote the Weil-Deligne group associated to $\mathsf k$ \cite[\S 8.3.6]{Deligne1972}.
The semidirect product is defined via the action
\[w x w^{-1}=\|w\| x,\qquad x\in \mathbb C,\ w\in W_k,
\]
where $\|w\|$ is the norm of $w\in W_k$.

\begin{definition}\label{def:Langlandsparams}
A \emph{Langlands parameter} is a continuous homomorphism
\begin{equation}
    \phi: W_{\mathsf k}'\rightarrow G^\vee
\end{equation}
which respects the Jordan decompositions in $W_{\mathsf k}'$ and $G^\vee$ (\cite[\S8.1]{Borel1979}).
\end{definition}
For any Langlands parameter $\phi$, let $\Pi^{\mathsf{Lan}}_{\phi}(\mathbf{G}(\mathsf k))$ denote the associated $L$-packet of irreducible admissible $\mathbf{G}(\mathsf k)$-representations \cite[\S10]{Borel1979}. 
The $L$-packets have not yet been defined in general. A discussion of this problem is beyond the scope of this paper---we refer the reader to \cite{Vogan1993}, \cite{Arthur2013}, or \cite{Kaletha2022}. 

Of particular interest for us is the following special case. 
We say that a Langlands parameter is \emph{unramified} if it is trivial when restricted to the inertia subgroup $I_{\mathsf k}$ of $W_{\mathsf k}$. If we fix a generator $\mathsf{Fr}$ of the infinite cyclic group $W_{\mathsf k}/I_{\mathsf k}$, we get a bijection between the set of unramified Langlands parameters and pairs of the form
\begin{equation}\label{e:unramified-Langlands}
   (s,u) \in G^{\vee} \times G^{\vee},\qquad s\text{ semisimple},\ u \text{ unipotent}, \ sus^{-1}=u^q.
\end{equation}
This bijection is defined by $\phi \mapsto (s,u) = (\phi(\mathsf{Fr}),\phi(1))$, where $1\in\mathbb{C}\subset W'_{\mathsf k}$. Replacing $u$ with $n \in \fg^{\vee}$ such that $q^n=u$, we get a further bijection onto pairs of the form
\begin{equation}\label{e:unramified-Langlands2}(s, n) \in G^{\vee} \times \fg^{\vee}, \qquad s\text{ semisimple},\ n \text{ nilpotent}, \ \Ad(s)n=qn.\end{equation}
If $\phi$ is an unramified Langlands parameter which corresponds to the pair $(s,n)$, the $L$-packet $\Pi_{\phi}^{\mathsf{Lan}}(\mathbf{G}({\mathsf k}))$ is in bijection with the set of irreducible representations $\rho$ of the component group $A(s,n)$ of the mutual centralizer in $G^\vee$ of $s$ and $n$, such that the center $Z(G^\vee)$ acts trivially on $\rho$. This is known as the Deligne-Langlands-Lusztig correspondence, see Theorem \ref{thm:Langlands} below.

Of fundamental importance are the \emph{tempered} Langlands parameters. 
For unramified parameters, `tempered' means that the semisimple element $s$ in (\ref{e:unramified-Langlands}) is `relatively compact', see Theorem \ref{thm:Langlands}(1) for a more precise condition.

\subsection{Unipotent representations: wavefront sets}\label{sec:results}

The main result of the paper is a formula for the wavefront set in terms of the Langlands parameters for the class of representations with \emph{unipotent cuspidal support} introduced by Lusztig \cite{Lu-unip1} (see Section \ref{s:unip-cusp} for the precise definition). In terms of the Langlands correspondence, see Section \ref{sec:Arthur}, these are exactly the representations which correspond to unramified Langlands parameters. 
Hence, via (\ref{e:unramified-Langlands2}), the set of irreducible representations with unipotent cuspidal support is in bijection with $G^{\vee}$-conjugacy classes of triples $(s,n,\rho)$. Write $X(s,n,\rho)$ for the irreducible representation corresponding to $(s,n,\rho)$. Since triples are considered up to conjugation by $G^{\vee}$, we may assume without loss of generality that $s$ belongs to a fixed maximal torus $T^\vee$ of $G^\vee$. There is a polar decomposition 
$T^\vee=T_c^\vee T_{\mathbb R}^\vee$,
where $T^{\vee}_c$ is compact and $T^{\vee}_{\RR}$ is a vector group, see (\ref{eq:real}). We say that $s$ is \emph{real} if $s\in T^\vee_{\mathbb R}$, and in this case, we say that $X(s,n,\rho)$ has \emph{real infinitesimal character}. 

As is well-known, the category of smooth $\bfG(\mathsf k)$-representation decomposes as a product of full subcategories, called \emph{Bernstein blocks}, see \cite[(2.10)]{Bernstein1984}. The category of representations with unipotent cuspidal support is a finite product of Bernstein blocks. The block containing the trivial representation (known as the \emph{principal block}) is precisely the category of \emph{Iwahori-spherical representations} (i.e. smooth representations generated by their Iwahori-fixed vectors). Under the Deligne-Langlands-Lusztig correspondence, the irreducible Iwahori-spherical representations correspond to the triples $(s,n,\rho)$ for which $\rho$ is of \emph{Springer type}. This means that $\rho$ occurs in the permutation representation of $A(s,n)$ in the top cohomology group of the variety of Borel subalgebras of $\mathfrak g^\vee$ that contain $n$ and are invariant under $\mathrm{Ad}(s)$. Notably, the set of irreducible representations with unipotent cuspidal support is the smallest class of representations which is a union of $L$-packets and contains all Iwahori-spherical representations, see Theorem \ref{thm:Langlands}.

An important role will be played by the \emph{Aubert-Zelevinsky} involution \cite{Au}, denoted $X \mapsto \mathrm{AZ}(X)$, see Section \ref{sec:AZduality}. This is an involution on the Grothendieck group of finite-length smooth $\mathbf{G}(\sfk)$-representations which preserves  the Grothendieck group of each Bernstein block and carries irreducibles to irreducibles (up to a sign). In the case of representations with unipotent cuspidal support, it preserves the semisimple parameter $s$. For example, $\mathrm{AZ}$ takes the trivial representation to the Steinberg discrete series representation, and more generally, a spherical representation to the unique generic representation (in the sense of admitting Whittaker models) with the same semisimple parameter. 

A fundamental invariant attached to an admissible representation $X$ is its \emph{wavefront set}. In its classical form, the wavefront set $\WF(X)$ of $X$ is a collection of nilpotent $\bfG(\sfk)$-orbits in the Lie algebra $\mathfrak g(\sfk)$: these are the maximal orbits for which the Fourier transforms of the corresponding orbital integrals contribute to the Harish-Chandra-Howe local character expansion of the distribution character of $X$, \cite[Theorem 16.2]{HarishChandra1999}. In this paper, we consider two coarser invariants (see Section \ref{s:wave} for the precise definitions). The first of these invariants is the \emph{algebraic wavefront set}, denoted $\hphantom{ }^{\bar{\sfk}}\WF(X)$. This is a collection of nilpotent orbits in $\mathfrak{g}(\bark)$, see for example \cite[p. 1108]{Wald18} (where it is simply referred to as the `wavefront set' of $X$). The second invariant is a natural refinement $^K\WF(X)$ of $\hphantom{ }^{\bar{\sfk}}\WF(X)$ called the \emph{canonical unramified wavefront set}, defined recently in \cite{okada2021wavefront}. This is a collection of nilpotent orbits $\mathfrak g(K)$ (modulo a certain equivalence relation $\sim_A$). The relationship between these three invariants is as follows: the algebraic wavefront set $\hphantom{ }^{\bar{\sfk}}\WF(X)$ is deducible from the usual wavefront set $\WF(X)$ as well as the canonical unramified wavefront set $^K\WF(X)$. It is not known whether the canonical unramified wavefront set is deducible from the usual wavefront set, but we expect this to be the case, see \cite[Section 5.1]{okada2021wavefront} for a careful discussion of this topic. 

The final ingredient that we need in order to state our main result is the duality map $d$ defined by Spaltenstein in \cite[Proposition 10.3]{Spaltenstein} (see also \cite[\S13.3]{Lusztig1984} and \cite[Appendix A]{BarbaschVogan1985}), and its refinement $d_A$ due to Achar \cite{Acharduality}. Suppose $G$ is the complex reductive group with dual $G^\vee$. Let $\mathcal N_o$ be the set of nilpotent orbits in the Lie algebra $\mathfrak g$ and let $\mathcal N_{o,\bar c}$ be the set of pairs $(\mathbb O,\bar C)$ consisting of a nilpotent orbit $\mathbb O\in \mathcal N_o$ and a conjugacy class $\bar C$ in $\bar A(\mathbb O)$, Lusztig's canonical quotient of the $G$-equivariant fundamental group $A(\OO)$ of $\OO$. Let $\mathcal N^\vee_o$ and $\mathcal N^\vee_{o,\bar{c}}$ be the corresponding sets for $G^\vee$. The duality $d$ is a map $d: \cN^{\vee}_o \to \cN_o$ whose image consists of Lusztig's special orbits.  The duality $d_A$ is a map
\[d_A: \mathcal N^\vee_{o,\bar c}\rightarrow \mathcal N_{o,\bar c}
\]
satisfying certain properties, see section \ref{subsec:nilpotent}. One of these properties is:
\[d_A(\OO^\vee,1)=(d(\OO^\vee),\bar C'),
\]
where $\bar{C}'$ is a conjugacy class in $\bar{A}(d(\OO^{\vee}))$ (which is trivial in the case when $\OO^\vee$ is special in the sense of Lusztig).

In \cite[Section 5.1]{okada2021wavefront}, it is shown that there is bijection between $\mathcal N_{o,\bar c}$ and the set of $\sim_A$-classes of unramified nilpotent orbits $\mathcal N_o(K)$ of $\bfG(K)$. This is recalled in section \ref{subsec:nilpotent}. This way, we can think of $d_A(\OO^\vee, \bar{C})$ as a class in $\mathcal N_o(K)/\sim_A$.


\begin{theorem}[See Theorems \ref{thm:realwf}, \ref{cor:wfbound} below]\label{t:main} Let $X=X(s,n,\rho)$ be an irreducible smooth Iwahori-spherical $\bfG(\mathsf k)$-representation with real infinitesimal character and let $\AZ(X) = X(s,n',\rho')$. 
\begin{enumerate}
    \item The canonical unramified wavefront set $^K\WF(X)$ is a singleton, and
\[^K\WF(X) = d_A(\OO^{\vee}_{\AZ(X)},1),\]
where $\OO^{\vee}_{\AZ(X)}$ is the $G^\vee$-orbit of $n'$ and $d_A$ is the duality map defined by Achar. In particular,
\[\hphantom{ }^{\bar{\sfk}}\WF(X) = d(\OO^\vee_{\AZ(X)}).\]
\item Suppose $X=X(q^{\frac 12 h^\vee},n,\rho)$ where $h^\vee$ is the neutral element of a Lie triple attached to a nilpotent orbit $\OO^\vee\subset \mathfrak g^\vee$. Then
\[d_A(\OO^{\vee}, 1) \leq_A \hphantom{ } ^K\WF(X),
\]
where $\leq_A$ is the partial order defined by Achar. In particular,
\[d(\OO^{\vee}) \le \hphantom{ }^{\bar{\sfk}}\WF(X).\]
\end{enumerate}
\end{theorem}

The proof relies on the local character expansion ideas of Barbasch and Moy \cite{barmoy} and the subsequent refinements of the third-named author \cite{okada2021wavefront}, together with a detailed analysis of the Springer correspondence, nilpotent orbits, and the combinatorics of the Achar duality.

We believe that Theorem \ref{t:main} should remain true for all irreducible representations with unipotent cuspidal support with real infinitesimal character. {However, the combinatorics in the more general case are more involved as they rely on the generalized Springer correspondence; moreover, in order to obtain the necessary branching results to parahoric subgroups, one needs to use the correspondences with Lusztig's graded affine Hecke algebras, and this adds another layer of complexity.}  
{Additionally, there are examples that show that the statement of Theorem \ref{t:main}(1) does not hold if $s$ is not real.} We intend to address these generalizations in future work. 

There is a vast literature on the computation of the (mainly algebraic) wavefront set of representations of $p$-adic groups. We mention only two examples related to the unipotent representations. In \cite{MW87}, M\oe glin and Waldspurger introduced the notion of generalized Whittaker models, which they used to compute the wavefront sets for representations of $\mathrm{GL}(n)$, small-rank (in the sense of Howe) representations of $\mathrm{Sp}(2n)$, and, in the notation of Theorem \ref{t:main}, the case $\frac 12 h^\vee=\rho$, the infinitesimal character where the trivial $\mathbf G(\mathsf k)$-representation occurs. For $\bfG=\mathrm{SO}(2n+1)$, Waldspurger  has computed the algebraic wavefront sets of  all irreducible tempered representations  with unipotent cuspidal support \cite[Th\'eor\`eme 2]{Wald20} and of their $\AZ$-duals \cite[Th\'eor\`eme, p.1108]{Wald18}. 

\smallskip

To use Theorem \ref{t:main} as a tool for computing wavefront sets, we need an algorithm for computing the $\AZ$-dual of an irreducible representation. 
For example, if $X$ is an irreducible spherical representation (in the sense of having nonzero fixed vectors under the action of $\mathbf G(\mathfrak o)$) with Satake parameter $s$), then $\mathsf{AZ}(X)=X(s,n',\rho')$ admits nonzero Whittaker models, and it is known that this is equivalent to the condition that $G^\vee(s)n'$ is the unique open orbit in $\mathfrak g^\vee_q$ (section \ref{s:unip-cusp}). Denote $\OO^\vee_s=G^\vee n'$, i.e., the $G^\vee$-saturation of the open $G^\vee(s)$-orbit in $\mathfrak g^\vee_q$. Then an immediate consequence of Theorem \ref{t:main} is
the following.

\begin{cor}
The algebraic wavefront set of the irreducible spherical $\mathbf G(\mathsf k)$-representation $X(s)$ with real Satake parameter $s$ is
\[\hphantom{ }^{\bar{\sfk}}\WF(X(s)) = d(\OO^\vee_s).\]
\end{cor}
While this result was expected for a long time (by analogy with real groups for example), as far as we know, this is the first proof for all split $p$-adic groups.

\medskip

In general, it is a difficult problem to compute the $\mathsf{AZ}$-dual. Evens and Mirkovi\'c \cite[Theorem 0.1]{EM} proved that for Iwahori-spherical representations, $\AZ$ admits a geometric description in terms of the Fourier-Deligne transform on irreducible perverse sheaves which is computable in principle via the algorithms in \cite[\S2]{Lusztig2010}. For $\bfG=\mathrm{GL}(n)$, an algorithm involving Zelevinsky multisegments was given in \cite{MW-duality}.
For split symplectic and odd orthogonal groups, an explicit combinatorial algorithm for computing $\AZ$ was recently announced in \cite{AtobeMinguez}.  We also mention the recent progress by Waldspurger \cite{Wald19} on the generalized Springer correspondence for $\mathrm{Sp}(2n)$ which provides another perspective on $\AZ$-duality (more in line with our methods in this paper).

In a sequel to this paper, we will use the wavefront set results in order to give a new characterization of the anti-tempered unipotent Arthur packets and a $p$-adic analogue of the weak Arthur packets defined for real groups in \cite{AdamsBarbaschVogan}.

\subsection{Acknowledgments}

The authors would like to thank Kevin McGerty and David Vogan for many helpful conversations. The authors would also like to thank Anne-Marie Aubert, Colette M\oe glin, David Renard, and Maarten Solleveld for their helpful comments and corrections on an earlier draft of this paper. The first and second authors were partially supported by the Engineering and Physical Sciences Research Council under grant EP/V046713/1.
The third author was supported by Aker Scholarship.

\section{Preliminaries}\label{sec:preliminaries}


Let $\bfG$ be a connected reductive algebraic group defined over $\mathbb{Z}$, and let $\bfT \subset \mathbf{G}$ be a maximal torus. For any field $F$, we write $\mathbf{G}(F)$, $\mathbf{T}(F)$, etc. for the groups of $F$-rational points. The $\CC$-points are denoted by $G$, $T$, etc. 

Write $X^*(\mathbf{T},\bark)$ (resp. $X_*(\mathbf{T},\bark)$) for the lattice of algebraic characters (resp. co-characters) of $\mathbf{T}(\bark)$, and write $\Phi(\mathbf{T},\bark)$ (resp. $\Phi^{\vee}(\mathbf{T},\bark)$) for the set of roots (resp. co-roots). Let
$$\mathcal R=(X^*(\mathbf{T},\bark), \ \Phi(\mathbf{T},\bark),X_*(\mathbf{T},\bark), \ \Phi^\vee(\mathbf{T},\bark), \ \langle \ , \ \rangle)$$
be the root datum corresponding to $\mathbf{G}$, and let $W$ the associated (finite) Weyl group. Let $\mathbf{G}^\vee$ be the Langlands dual group of $\bfG$, i.e. the connected reductive algebraic group corresponding to the root datum 
$$\mathcal R^\vee=(X_*(\mathbf{T},\bark), \ \Phi^{\vee}(\mathbf{T},\bark),  X^*(\mathbf{T},\bark), \ \Phi(\mathbf{T},\bark), \ \langle \ , \ \rangle).$$
Set $T^\vee=X^*(\bfT,\bark)\otimes_\ZZ \CC^\times$, regarded as a maximal torus in $G^\vee$ with Lie algebra $\mathbf{\mathfrak t}^\vee=X^*(\bfT,\bark)\otimes_{\mathbb Z} \mathbb C$, a Cartan subalgebra of the Lie algebra $\mathbf{\mathfrak g}^\vee$ of $\bfG^\vee$. Define
\begin{align}\label{eq:real}
\begin{split}
    T^\vee_{\mathbb R} &=X^*(\bfT,\bark)\otimes_{\mathbb Z} {\mathbb R}_{>0}\\
    \mathbf{\mathfrak t}_{\mathbb R}^\vee &= X^*(\bfT,\bark)\otimes_{\mathbb Z} \mathbb R\\
    T^\vee_c &=X^*(\bfT,\bark)\otimes_{\mathbb Z} S^1
\end{split}
\end{align}
There is a polar decomposition $T^\vee=T^\vee_c T ^\vee_{\mathbb R}$.

If $H$ is a complex group and $x$ is an element of $H$ or $\fh$, we write $H(x)$ for the centralizer of $x$ in $H$, and $A_H(x)$ for the group of connected components of $H(x)$. If $S$ is a subset of $H$ or $\fh$ (or indeed, of $H \cup \fh$), we can similarly define $H(S)$ and $A_H(S)$. We will sometimes write $A(x)$, $A(S)$ when the group $H$ is implicit.

Write $\mathcal B^\vee$ for the flag variety of $G^\vee$, i.e. the variety of Borel subgroups $B^{\vee} \subset G^{\vee}$. Note that $\mathcal{B}^{\vee}$ has a natural left $G^{\vee}$-action. 
For $g\in G^\vee$, write
$$\mathcal B^\vee_g = \{B^\vee\in \mathcal B^\vee \mid g\in B^\vee \}.$$
(this coincides with the subvariety of Borels fixed by $g$). Similarly, for $x\in \mathfrak g^\vee$, write
$$\mathcal B^\vee_x = \{B^\vee\in \mathcal B^\vee \mid x\in \mathfrak b^\vee \}.$$
If $S$ is a subset of $G^{\vee}$ or $\fg^{\vee}$ (or indeed of $G^{\vee} \cup \fg^{\vee}$), write
$$\mathcal B^\vee_S = \bigcap_{x\in S} \mathcal{B}^{\vee}_x.$$
The singular cohomology group $H^i(\mathcal B^\vee_S,\CC) = H^i(\mathcal{B}^{\vee}_S)$ carries an action of $A(S)=A_{G^\vee}(S)$. For an irreducible representation $\rho\in\mathrm{Irr}(A(S)))$, let $H^i(\mathcal B^\vee_S)^\rho := \Hom_{A(S)}(\rho,H^i(\mathcal{B}^{\vee}_S))$. Write $H^{\mathrm{top}}(\mathcal{B}^{\vee}_S)$ for the top-degree nonzero cohomology group and $H^{\bullet}(\mathcal{B}_S^{\vee})$ for the alternating sum of all cohomology groups. We will often consider the subset
\begin{equation}\label{eq:defofIrr0}
    \mathrm{Irr}(A(S))_0 := \{\rho \in \mathrm{Irr}(A(S)) \mid H^{\mathrm{top}}(\mathcal{B}_S^{\vee})^{\rho} \neq 0\}.
\end{equation}

\medskip

Let $\mathcal C(\bfG(\mathsf k))$ be the category of smooth complex $\bfG(\mathsf k)$-representations and let $\Pi(\mathbf{G}(\mathsf k)) \subset \mathcal C(\bfG(\mathsf k))$ be the set of irreducible objects. Let $R(\bfG(\mathsf k))$ denote the Grothendieck group of $\mathcal C(\bfG(\mathsf k))$.

\subsection{Nilpotent orbits}\label{subsec:nilpotent}

Let $\mathcal N$ be the functor which takes a field $F$ to the set of nilpotent elements of $\mf g(F)$.
By `nilpotent' in this context we mean the unstable points (in the sense of GIT) with respect to the adjoint action of $\bfG(F)$, see \cite[Section 2]{debacker}.
For $F$ algebraically closed this coincides with all the usual notions of nilpotence.
Let $\mathcal N_o$ be the functor which takes $F$ to the set of orbits in $\mathcal N(F)$ under the adjoint action of $\bfG(F)$.
When $F$ is $\sfk$ or $K$, we view $\mathcal N_o(F)$ as a partially ordered set with respect to the closure ordering in the topology induced by the topology on $F$.
When $F$ is algebraically closed, we view $\mathcal N_o(F)$ as a partially ordered set with respect to the closure ordering in the Zariski topology.
For brevity we will write $\mathcal N(F'/F)$ (resp. $\mathcal N_o(F'/F)$) for $\mathcal N(F\to F')$ (resp. $\mathcal N_o(F\to F')$) where $F\to F'$ is a morphism of fields.
For $(F,F')=(\sfk,K)$ (resp. $(\sfk,\bark)$, $(K,\bark)$), the map $\mathcal N_o(F'/F)$ is strictly increasing (resp. strictly increasing, non-decreasing).
We will simply write $\mathcal N$ for $\mathcal N(\CC)$ and $\mathcal N_o$ for $\mathcal N_o(\CC)$.
In this case we also define $\mathcal N_{o,c}$ (resp. $\mathcal N_{o,\bar c}$) to be the set of all pairs $(\OO,C)$ such that $\OO\in \mathcal N_o$ and $C$ is a conjugacy class in the fundamental group $A(\OO)$ of $\OO$ (resp. Lusztig's canonical quotient $\bar A(\OO)$ of $A(\OO)$, see \cite[Section 5]{Sommers2001}). There is a natural map 
\begin{equation}
    \mf Q:\mathcal N_{o,c}\to\mathcal N_{o,\bar c}, \qquad (\OO,C)\mapsto (\OO,\bar C)
\end{equation}
where $\bar C$ is the image of $C$ in $\bar A(\OO)$ under the natural homomorphism $A(\OO)\twoheadrightarrow \bar A(\OO)$. There are also projection maps $\pr_1: \cN_{o,c} \to \cN_o$, $\pr_1: \cN_{o,\bar c} \to \cN_o$. We will typically write $\mathcal N^\vee$, $\mathcal N^\vee_o, \cN^{\vee}_{o,c}$, and $\cN^{\vee}_{o,\bar c}$ for the sets $\mathcal N$, $\mathcal N_o, \cN_{o,c}$, and $\cN_{o,\bar c}$ associated to the Langlands dual group $G^\vee$. When we wish to emphasise the group we are working with we include it as a superscript e.g. $\mathcal N_o^\bfG(k)$.

Recall the following classical result.

\begin{lemma}[Corollary 3.5, \cite{Pommerening} and Theorem 1.5, \cite{Pommerening2}]\label{lem:Noalgclosed}
    Let $F$ be algebraically closed with good characteristic for $\bfG$.
    Then there is canonical isomorphism of partially ordered sets $\Theta_F:\mathcal N_o(F)\xrightarrow{\sim}\mathcal N_o$.
\end{lemma}

Write
\begin{equation}\label{eq:dBV}
d: \cN_0 \to \cN_0^{\vee}, \qquad d: \cN_0^{\vee} \to \cN_0.
\end{equation}
for the \emph{Barbasch-Lusztig-Spaltenstein-Vogan duality maps} (see \cite[Appendix A]{BarbaschVogan1985}). If $F$ is algebraically closed, we will also write
\begin{equation}
d: \cN_o(F) \to \cN_o^{\vee}(F), \qquad d: \cN_o^{\vee}(F) \to \cN_o(F)
\end{equation}
for the maps obtained by composing the maps (\ref{eq:dBV}) with the natural identifications $\mathcal N_o(F)\simeq \mathcal N_o$, $\mathcal N_o^\vee(F)\simeq \mathcal N_o^\vee$ of Lemma \ref{lem:Noalgclosed}. Write 
\begin{equation}
    d_S: \cN_{o,c} \twoheadrightarrow \cN^{\vee}_o, \qquad d_S: \cN^{\vee}_{o,c} \twoheadrightarrow \cN_o
\end{equation}
for the duality maps defined by Sommers in \cite[Section 6]{Sommers2001} and 
\begin{equation}
    d_A: \cN_{o,\bar c} \to \cN^{\vee}_{o,\bar c}, \qquad d_A: \cN^{\vee}_{o,\bar c} \to \cN_{o,\bar c}
\end{equation}
for the duality maps defined by Achar in (\cite[Section 1]{Acharduality}). We have the following compatibilities, which are clear from the definitions:
\begin{itemize}
    \item $d(\OO) = d_S(\OO,1)$ .
    \item $d_S(\OO,C) = \pr_1\circ d_A(\OO,C)$.
\end{itemize}
There is a natural pre-order $\leq_A$ on $\mathcal N_{o,c}$ defined by 
$$(\OO,C)\le_A(\OO',C') \iff \OO\le \OO' \text{ and } d_S(\OO,C)\ge d_S(\OO',C').$$
Write $\sim_A$ for the equivalence relation on $\cN_{o,c}$ induced by this pre-order, i.e. 
$$(\OO_1,C_1) \sim_A (\OO_2,C_2) \iff (\OO_1,C_1) \leq_A (\OO_2,C_2) \text{ and } (\OO_2,C_2) \leq_A (\OO_1,C_1)$$
Write $[(\OO,C)]$ for the equivalence class of $(\OO,C) \in \cN_{o,c}$. The $\sim_A$-equivalence classes in $\cN_{o,c}$ coincide with the fibres of the projection map $\mf Q:\mathcal N_{o,c}\to\mathcal N_{o,\bar c}$ \cite[Theorem 1]{Acharduality}. So $\le_A$ descends to a partial order on $\mathcal N_{o,\bar c}$, also denoted by $\le_A$. The maps $d,d_S,d_A$ are all order reversing with respect to the relevant pre/partial orders. We also have the following easy result.

\begin{lemma}
    \label{lem:injectiveachar}
    Let $\OO^\vee_1,\OO^\vee_2\in\mathcal N_o^\vee$.
    Then the following are true:
    \begin{enumerate}
        \item $\OO^\vee_1\le\OO^\vee_2$ if and only if $d_A(\OO^\vee_1,1)\ge_Ad_A(\OO^\vee_2,1)$,
        \item $\OO^\vee_1=\OO^\vee_2$ if and only if $d_A(\OO^\vee_1,1)=d_A(\OO^\vee_2,1)$.
    \end{enumerate}
\end{lemma}
\begin{proof}
\begin{enumerate}
    \item $(\Rightarrow)$ Suppose $\OO^\vee_1\le\OO^\vee_2$.
    Then $d(\OO^\vee_1)\ge d(\OO^\vee_2)$.
    So we have 
    \begin{align*}\pr_1(d_A(\OO^\vee_1,1)) &= d(\OO^\vee_1)\ge d(\OO^\vee_2) = \pr_1(d_A(\OO^\vee_2,1))\\
    d_S(d_A(\OO^\vee_1,1)) &= \OO^\vee_1 \le \OO^\vee_2 = d_S(d_A(\OO^\vee_2,1)).\end{align*}
    Thus indeed $d_A(\OO^\vee_1,1)\ge_Ad_A(\OO^\vee_2,1)$.
    
    $(\Leftarrow)$ If $d_A(\OO^\vee_1,1)\ge_Ad_A(\OO^\vee_2,1)$ then $\OO^\vee_1 = d_S(d_A(\OO^\vee_1,1)) \le d_S(d_A(\OO^\vee_2,1)) = \OO^\vee_2$.
    \item This follows from part (i).
\end{enumerate}
\end{proof}

\subsection{The Bruhat-Tits building}
Let $\mathcal B(\bfG)$ denote the (enlarged) Bruhat-Tits building for $\bfG(\sfk)$. We use the notation  $c\subseteq \mathcal B(\bfG)$ to indicate that $c$ is a face of $\mathcal B$.
Given a maximal torus $\bfT$ defined and split over $\sfk$, write $\mathcal A(\bfT,\sfk)$ for the corresponding apartment in $\mathcal B(\bfG)$.
For an apartment $\mathcal A$ of $\mathcal B(\bfG)$ and $\Omega\subseteq \mathcal A$ we write $\mathcal A(\Omega,\mathcal A)$ for the smallest affine subspace of $\mathcal A$ containing $\Omega$.
We write $\Phi(\bfT,\sfk)$ (resp. $\Psi(\bfT,\sfk)$) for the set of roots of $\bfG(\sfk)$ (resp. affine roots) of $\bfT(\sfk)$ on $\bfG(\sfk)$. 
For $\psi\in \Psi(\bfT,\sfk)$ write $\dot\psi\in \Phi(\bfT,\sfk)$ for the gradient of $\psi$, and
$W=W(\bfT,\sfk)$ for the Weyl group of $\bfG(\sfk)$ with respect to $\bfT(\sfk)$.
Let $\widetilde W=W\ltimes X_*(\mathbf{T},\sfk)$ be the (extended) affine Weyl group. 
Write 
\begin{equation}
    \widetilde W\to W, \qquad w\mapsto \dot w
\end{equation}
for the natural projection map.
Fix a special point $x_0$ of $\mathcal A(\bfT,k)$.
The choice of special point $x_0$ of $\mathcal B(\bfG,K)$ fixes an inclusion $\Phi(\bfT_1,K)\to \Psi(\bfT_1,K)$ and an isomorphism between $\widetilde W$ and $N_{\bfG(K)}(\bfT_1(K))/\bfT_1(\mf O^\times)$.
For a face $c\subseteq \mathcal A$ let $W_c$ be the subgroup of $\widetilde{W}$ generated by reflections in the hyperplanes through $c$ (equivalently $\mathcal A(c,\mathcal A)$).
For a face $c\subseteq \mathcal B(\bfG)$ there is a subgroup $\bfP_c^+$ of $\bfG$ defined over $\mf o$ such that $\bfP_c^+(\mf o)$ is the stabiliser of $c$ in $\bfG(k)$. There is an exact sequence
\begin{equation}\label{eq:parahoricses}
    1 \to \bfU_c(\mf o) \to  \bfP_c^+(\mf o) \to  \bfL_c^+(\mathbb F_q) \to 1,
\end{equation}
where $\bfU_c(\mf o)$ is the pro-unipotent radical of $\bfP_c^+(\mf o)$ and $\bfL_c^+$ is the reductive quotient of the special fibre of $\bfP_c^+$ .
Let $\bfL_c$ denote the identity component of $\bfL_c^+$, and let $\bfP_c$ be the subgroup of $\bf P_c^+$ defined over $\mf o$ such that $\bfP_c(\mf o)$ is the inverse image of $\bfL_c(\mathbb F_q)$ in $\bfP_c^+(\mf o)$.
The torus $\bfT$ is in fact defined over $\mf o$, is a subgroup of $\bfP_c$ and the special fibre of $\bfT$, denoted $\bar\bfT$, is a $\mathbb F_q$-split torus of $\bfL_c$.
Write $\Phi_c(\bar\bfT,\mathbb F_q)$ for the root system of $\bfL_c$ with respect to $\bar\bfT$.
Then $\Phi_c(\bar\bfT,\mathbb F_q)$ naturally identifies with the set of $\psi\in\Psi(\bfT,k)$ that vanish on $\mathcal A(c,\mathcal A)$, and the Weyl group of $\bar\bfT$ in $\bfL_c$ is naturally isomorphic to $W_c$. The groups $\bfP_c$ obtained in this manner are called ($\sfk$-)\emph{parahoric subgroups} of $\bfG$.
When $c$ is a chamber in the building, then we call $\bfP_c$ an \emph{Iwahori subgroup} of $G$. 

Now suppose $\mathbf{T}$ is defined over $\ZZ$ and let $\mathcal A = \mathcal A(\bfT,\sfk)$.
Fix a special point $x_0$ in $\mathcal A$.
Choose a set of simple roots $\Delta \subset \Phi(\bfT,\sfk)$, and let $\tilde\Delta\subseteq \Psi(\bfT,\sfk)$ be the corresponding set of extended simple roots (this depends on the choice of $x_0$).
Let $c_0$ denote the chamber cut out by $\tilde\Delta$.
Each subset $J\subseteq \tilde\Delta$ cuts out a face of $c_0$ which we denote by $c(J)$.
We will need the following result from \cite{okada2021wavefront}.
Recall that a \emph{pseudo-Levi} subgroup $L$ of $G$ is a subgroup arising as the centraliser of a semisimple group element. A pseudo-levi subgroup is \emph{standard} if contains our fixed maximal torus $T=\mathbf{T}(\CC)$. We will write $Z_L$ for the center of $L$.
\begin{lemma}
    \cite[Proposition 4.17]{okada2021wavefront}
    There is a natural $W$-equivariant map
    \begin{equation}
        \Xi:\{\text{faces of } \mathcal A\} \rightarrow\{(L,tZ_L^\circ) \mid L\text{ a standard pseudo-Levi}, \  Z_{G}^\circ(tZ_L^\circ)=L\}
    \end{equation}
    where $c_1,c_2$ lie in the same fibre iff $\mathcal A(c_1,\mathcal A)+X_*(\bfT, \sfk)=\mathcal A(c_2,\mathcal A)+X_*(\bfT,\sfk)$.
    Moreover, if $\Xi(c) =(L,tZ_L^\circ)$ then $L$ is the complex reductive group with the same root datum as $\bfL_c(\mathbb F_q)$ (with respect to the tori $T$ and $\bar \bfT$ respectively).
\end{lemma}
\begin{rmk}\label{rmk:rootdata}
    In the statement of the lemma, the weight lattices naturally identify since $T=\bfT(\CC)$ and $\bar\bfT(\mathbb F_q)=\bfT(\mf o)/\bfT(1+\mf p)$.
    The root data identify via the map $$\Phi_c(\bar\bfT,\mathbb F_q)\xrightarrow{\sim}\{\dot\psi \mid \psi\in\Psi(\bfT,\sfk), \ \psi(c)=0\}\subseteq\Phi(\bfT,k)\xrightarrow{\sim}\Phi(\bfT,\CC).$$
\end{rmk}
For a face $c\subseteq\mathcal B(\bfG)$ and $\Xi(c) = (L,tZ_L^\circ)$ there is a natural bijection of partially ordered sets 
$$\Theta_c:\mathcal N_o^{\bfL_c}(\overline{\mathbb F}_q)\xrightarrow{\sim} \mathcal N_o^{L}$$
induced by the isomorphism of root data in Remark \ref{rmk:rootdata}.

\subsection{Nilpotent orbits over a maximal unramified extension}

Let $\bfT$ be a maximal $\sfk$-split torus of $\bfG$ and $x_0$ be a special point in $\mathcal A(\bfT,K)$.
In \cite[Section 2.1.5]{okada2021wavefront} the third-named author constructs a bijection
$$\theta_{x_0,\bfT}:\mathcal N_o^{\bfG}(K)\xrightarrow{\sim}\mathcal N_{o,c}.$$
\begin{theorem}
    \label{lem:paramNoK}
    \cite[Theorem 2.20, Theorem 2.27, Proposition 2.29]{okada2021wavefront}
    The bijection 
    $$\theta_{x_0,\bfT}:\mathcal N_o^{\bfG}(K)\xrightarrow{\sim}\mathcal N_{o,c}$$
    is natural in $\bfT$, equivariant in $x_0$, and makes the following diagram commute:
    \begin{equation}
        \begin{tikzcd}[column sep = large]
            \mathcal N_o^{\bfG}(K) \arrow[r,"\theta_{x_0,\bfT}"] \arrow[d,"\mathcal N_o(\bar k/K)",swap] & \mathcal N_{o,c} \arrow[d,"\pr_1"] \\
            \mathcal N_o(\bar k) \arrow[r,"\Lambda^{\bar k}"] & \mathcal N_o.
        \end{tikzcd}
    \end{equation}
\end{theorem}

The composition 
$$d_{S,\bfT}:= d_S\circ \theta_{x_0,\bfT}:\mathcal N_o(K) \to \mathcal N^\vee_o$$
is independent of the choice of $x_0$ and natural in $\bfT$ \cite[Proposition 2.32]{okada2021wavefront}.
For $\OO_1,\OO_2\in \mathcal N_o(K)$ define $\OO_1\le_A\OO_2$ by
$$\OO_1\le_A \OO_2 \iff \mathcal N_o(\bar k/K)(\OO_1) \le \mathcal N_o(\bar k/K)(\OO_2),\text{ and } d_{S,\bfT}(\OO_1)\ge d_{S,\bfT_2}(\OO_2)$$
and let $\sim_A$ denote the equivalence classes of this pre-order.
This pre-order is independent of the choice of $\bfT$ and by Theorem \ref{lem:paramNoK}, the map 
$$\theta_{x_0,\bfT}:(\mathcal N_o(K),\le_A) \to (\mathcal N_{o,c},\le_A)$$
is an isomorphism of pre-orders.
\begin{theorem}
    \label{thm:unramclasses}
    The composition $\mf Q\circ \theta_{x_0,\bfT}:\mathcal N_o(K)\to \mathcal N_{o,\bar c}$ descends to a (natural in $\bfT$) bijection 
    $$\bar\theta_{\bfT}:\mathcal N_o(K)/\sim_A\to \mathcal N_{o,\bar c}$$
    which does not depend on $x_0$.
\end{theorem}

By the construction of these maps, we have the following commutative diagram
\begin{equation}
    \label{eq:square}
    \begin{tikzcd}
        \mathcal N_o(K) \arrow[r,"{[\bullet]}"] \arrow[d,"\theta_{x_0,\bfT}",swap] & \mathcal N_0(K)/\sim_A \arrow[d,"\bar\theta_{\bfT}"] \\
        \mathcal N_{o,c} \arrow[r,"\mf Q"] & \mathcal N_{o,\bar c}.
    \end{tikzcd}
\end{equation}

Define 
\begin{equation}
    \mathcal I_o = \{(c,\OO) \mid c\subset \mathcal B(\bfG),\OO\in\cN_o^{\bfL_c}(\overline{\mathbb F}_q)\}.
\end{equation}
There is a partial order on $\mathcal{I}_o$, defined by
$$(c_1,\OO_1)\le(c_2,\OO_2) \iff c_1=c_2 \text{ and } \OO_1\le\OO_2$$
In \cite[Section 4]{okada2021wavefront} the third author defines a strictly increasing surjective map 
$$\mathscr L:(\mathcal I_o,\le)\to(\mathcal N_o(K),\le).$$
The composition $[\bullet]\circ \mathscr L:(\mathcal I_o,\le)\to (\mathcal N_o(K)/\sim_A,\le_A)$ is also strictly increasing \cite[Corollary 4.7,Lemma 5.3]{okada2021wavefront}.

Write $L_c$ for $\pr_1\circ\Xi(c)$.
Define 
\begin{align}
    \mathcal{I}_{\tilde\Delta}&=\{(J,\OO) \mid J\subsetneq\tilde\Delta, \ \OO\in\mathcal{N}_o^{\bfL_{c(J)}}(\overline{\mathbb F}_q)\}, \\
    \mathcal{K}_{\tilde\Delta}&=\{(J,\OO) \mid J\subsetneq\tilde\Delta, \ \OO\in\mathcal{N}_o^{L_{c(J)}}(\CC)\}.
\end{align}
Then $\mathcal I_{\tilde\Delta}\xrightarrow{\sim}\mathcal K_{\tilde\Delta}$ via $(J,\OO)\mapsto(J,\Theta_{c(J)}(\OO))$.
Define
\begin{equation}
    \mathbb L:\mathcal K_{\tilde\Delta} \to \mathcal N_{c,o}
\end{equation}
to be the map that sends $(J,\OO)$ to $(Gx,tZ_{G}^\circ(x))$ where $x\in\OO$ and $(L,tZ_{L}^\circ) = \Xi(c(J))$.
Define $\overline{\mathbb L}=\mf Q\circ \mathbb L$.
\begin{prop}
    \label{prop:square}
    The diagram 
    \begin{equation}
        \begin{tikzcd}
            \mathcal I_{\tilde\Delta} \arrow[r,"\sim"] \arrow[d,"\mathscr L"] & \mathcal K_{\tilde\Delta}  \arrow[d,"\mathbb L"] \\
            \mathcal N_o(K) \arrow[r,"\theta_{x_0,\bfT}"] & \mathcal N_{o,c}
        \end{tikzcd}
    \end{equation}
    commutes.
\end{prop}
\begin{proof}
    It is a straightforward observation from the definition of $\Xi$ in \cite[Theorem 4.16]{okada2021wavefront} that if $c_1\subseteq c_2$ are faces of $\mathcal A$, and $(L_i,t_iZ_{L_i}^\circ) = \Xi(c_i)$ $i=1,2$, then under the natural map $Z_{L_1}/Z_{L_1}^\circ\to Z_{L_2}/Z_{L_2}^\circ$, we have that $t_1Z_{L_1}^\circ\mapsto t_2Z_{L_2}^\circ$.
    Now suppose $(J_1,\OO_1),(J_2,\OO_2)\in \mathcal I_{\tilde\Delta}$ are such that $J_1\subseteq J_2$ and the saturation of $\OO_1$ in $\bfL_{c(J_2)}$ (which contains $\bfL_{c(J_2)}$ as a Levi) is $\OO_2$.
    Let $(L_i,t_iZ_{L_i}^\circ) = \Xi(c(J_i))$ and $\OO'_i = \Theta_{c(J_i)}(\OO_i)$ for $i=1,2$.
    Then the saturation of $\OO_1'$ in $L_2$ is $\OO_2'$.
    Let $x\in \OO_1'\subseteq \OO_2'$.
    We have that $\mathbb L(J_1,\OO_1') = (Gx,t_1Z_{G}^\circ(x))$ and $\mathbb L(J_2,\OO_2') = (Gx,t_2Z_G^\circ(x))$.
    But since $Z_{L_2}^\circ\subseteq Z_{L_1}^\circ \subseteq Z_G^\circ(x)$, we have that $t_1Z_G^\circ(x) = t_2Z_G^\circ(x)$.
    Therefore $\mathbb L(J_1,\OO_1') = \mathbb L(J_2,\OO_2')$.
    Since also $\mathscr L(c(J_1),\OO_1) = \mathscr L(c(J_2),\OO_2)$ we can reduce to checking that the diagram commutes for pairs $(J,\OO)$ where $\OO$ is distinguished.
    But the distinguished case follows by construction of the map $\theta_{x_0,\bfT}$ in \cite[Section 4]{okada2021wavefront}.
\end{proof}

\subsection{$W$-representations}\label{subsec:Wreps}
Let $\sim$ denote the equivalence relation on $\mathrm{Irr}(W)$ defined by Lusztig in \cite[Section 4.2]{Lusztig1984}.
The relation $\sim$ partitions $\mathrm{Irr}(W)$ into subsets called \emph{families}.
Each family contains a unique special representation which we denote by $E(\phi)$.
For $E\in\mathrm{Irr}(W)$ we write $\phi(E)$ for the family containing $E\otimes\mathrm{sgn}$.
Let
\begin{equation}\label{eq:Springer}
    \mathrm{Springer}:\mathrm{Irr}(W)\hookrightarrow \{(\OO,\rho) \mid \OO\in\mathcal N_o,\ \rho\in\mathrm{Irr}(A(\OO))\}.
\end{equation}
be the \emph{Springer correspondence}. Define the subset
\begin{equation}\label{eq:IrrA0}\mathrm{Irr}(A(\OO))_0 \subset \mathrm{Irr}(A(\OO))\end{equation}
as in (\ref{eq:defofIrr0}) (taking $S=\{n\}$, where $n \in \OO$). Then the image of the map (\ref{eq:Springer}) is the subset
$$\{(\OO,\rho) \mid \OO \in \cN_o, \ \rho \in \mathrm{Irr}(A(\OO))_0\}.$$
If $\mathrm{Springer}(E) = (\OO,\rho)$, we write $E(\OO,\rho):=E$ and $\OO(E,\CC) :=\OO$. We call $\OO(E,\CC)$ the \emph{Springer support} of $E$ (with respect to $G$). 
We write $\OO^\vee(E,\CC)$ for the Springer support of $E$ with respect to $G^\vee$, where we view $E$ as a representation of $W^\vee$ via the canonical isomorphism $W\xrightarrow{\sim} W^\vee$.
It will always be clear from context which reductive group (up to root datum) we are working with, but occasionally we will have two reductive groups with isomorphic root data, but defined over different fields (e.g. $\bfL_c$ and $L_c$) and so we include the base field in the second argument to distinguish between the two cases.
Thus if $E\in\mathrm{Irr}(W_c)$ then we write $\OO(E,\CC)$ for the Springer support with respect to $L_c$ and $\OO(E,\overline{\mathbb F}_q)$ for the Springer support with respect to $\bfL_c(\overline{\mathbb F}_q)$.

For each $J \subsetneq \widetilde{\Delta}$, write $W_J \subset W$ for the subgroup generated by the simple reflections $\dot J$.
Note that $W_J = \dot W_{c(J)}$, the Weyl group of $L_{c(J)}$ and $\bfL_{c(J)}$. Define the set
\begin{equation}
    \mathcal F_{\tilde\Delta}=\{(J,\phi) \mid J\subsetneq\tilde\Delta, \ \phi \text{ a family of $W_J$}\}.
\end{equation}
For $J\subsetneq\tilde\Delta$ and a family $\phi$ of $W_J$ write $\OO(\phi,\bullet)$ for $\OO(E(\phi),\bullet)$, where $\bullet\in\{\overline{\mathbb F}_q,\CC\}$.
Clearly $\Theta_{c(J)}(\OO(\phi,\overline{\mathbb F}_q)) = \OO(\phi,\CC)$.

\subsection{Representations with unipotent cuspidal support}\label{s:unip-cusp}

For every face $c \subset \mathcal{B}(\mathbf{G})$, we get
a parahoric subgroup $\mathbf{P}_c(\mf o) \subset \mathbf{G}(\mf o)$ with pro-unipotent radical $\bfU_c(\mf o)$ and reductive quotient $\bfL_c(\mathbb F_q)$ (see the paragraph preceding (\ref{eq:parahoricses}). If $X$ is a smooth admissible $\mathbf{G}(\sfk)$-representation, the space of invariants $X^{\bfU_c(\mf o)}$ is a finite-dimensional $\bfL_c(\mathbb F_q)$-representation.

\begin{definition} Let $X$ be an irreducible $\mathbf{G}(\mathsf{k})$-representation. We say that $X$ has \emph{unipotent cuspidal support} if there is a parahoric subgroup $\bfP_c \subset \bfG$ such that $X^{\mathbf U_c(\mf o)}$ contains an irreducible Deligne-Lusztig cuspidal unipotent representation of ${\mathbf L}_c(\mathbb F_q)$. Write $\Pi^{\mathsf{Lus}}(\bfG(\mathsf k))$ for the subset of $\Pi(\bfG(\mathsf k))$ consisting of all such representations.
\end{definition}

Recall that an irreducible $\mathbf{G}(\mathsf{k})$-representation $V$ is \emph{Iwahori-spherical} if $V^{\mathbf{I}(\mathfrak {o})} \neq 0$ for an Iwahori subgroup $\bfI$ of $\bfG$. We note that all such representations have unipotent cuspidal support, corresponding to the case $\mathbf P_c=\mathbf I$ and the trivial representation of $\mathbf T(\mathbb F_q)$.

We will now recall the classification of irreducible representations of unipotent cuspidal support. Write $\Phi(\bfG(\mathsf k))$ for the set of $G^\vee$-orbits (under conjugation) of triples $(s,n,\rho)$ such that
\begin{itemize}
    \item $s\in G^\vee$ is semisimple,
    \item $n\in \mathfrak g^\vee$ such that $\operatorname{Ad}(s) n=q n$,
    \item $\rho\in \mathrm{Irr}(A_{G^{\vee}}(s,n))$ such that $\rho|_{Z(G^\vee)}$ is a multiple of the identity.
\end{itemize}
Without loss of generality, we may assume that $s\in T^\vee$. Note that $n\in\mathfrak g^\vee$ is necessarily nilpotent. The group $G^\vee(s)$ acts with finitely many orbits on the $q$-eigenspace of $\Ad(s)$
$$\mathfrak g_q^\vee=\{x\in\mathfrak g^\vee\mid \operatorname{Ad}(s) x=qx\}$$
In particular, there is a unique open $G^\vee(s)$-orbit in $\mathfrak g_q^\vee$.

Fix an $\mathfrak{sl}(2)$-triple $\{n^-,h,n\} \subset \fg^{\vee}$ with $h\in \mathfrak t^\vee_{\mathbb R}$ and set
$$s_0:=sq^{-\frac{h}{2}}.$$
Then $\operatorname{Ad}(s_0)n=n$.

The following theorem is a combination of several results: \cite[Theorems 7.12, 8.2, 8.3]{KL} for $\bfG$ adjoint and Iwahori-spherical representations, \cite[Corollary 6.5]{Lu-unip1} and \cite[Theorem 10.5]{Lu-unip2} for $\bfG$ adjoint and representations with unipotent cuspidal support, \cite[Theorem 3.5.4]{Re-isogeny} for $\bfG$ of arbitrary isogeny and Iwahori-spherical representations, and \cite{Sol-LLC} for $\bfG$ of arbitrary isogeny and representations with unipotent cuspidal support. See \cite[\S2.3]{AMSol} for a discussion of the compatibility between these  classifications. Define the subset $\mathrm{Irr}(A(s,n))_0 \subset \mathrm{Irr}(A(s,n))$ as in (\ref{eq:defofIrr0}) (taking $S=\{s,n\}$).

\begin{theorem}[{Deligne-Langlands-Lusztig correspondence}]\label{thm:Langlands} Suppose that $\bfG$ is  $\mathsf k$-split. There is a bijection
$$\Phi(\bfG(\mathsf k))\xrightarrow{\sim} \Pi^{\mathsf{Lus}}(\bfG(\mathsf k)),\qquad (s,n,\rho)\mapsto X(s,n,\rho),$$
such that
\begin{enumerate}
    \item $X(s,n,\rho)$ is tempered if and only if $s_0\in T_c^\vee$ and $\overline {G^\vee(s)n}=\mathfrak g_q^\vee$,
    \item $X(s,n,\rho)$ is square integrable (modulo the center) if and only if it is tempered and $Z_{G^{\vee}}(s,n)$ contains no nontrivial torus.
\item $X(s,n,\rho)^{\mathbf I(\mf o)}\neq 0$ if and only if $\rho\in \mathrm{Irr}(A(s,n))_0$.
\end{enumerate}
\end{theorem}
 Denote by $\Phi(\bfG(\mathsf k))_0$ the subset of $ \Phi(\bfG(\mathsf k))$ for which $\rho\in \mathrm{Irr}(A(s,n))_0$.
\begin{rmk}
We note that the condition $\overline {G^\vee(s)n}=\mathfrak g_q^\vee$ in (1) is superfluous---it is a consequence of the condition $s_0 \in T_c^{\vee}$, see Lemma \ref{lem:orbitclosure}. It is included in Theorem \ref{thm:Langlands} for expository purposes.
\end{rmk}

\begin{rmk}\label{r:real}
The semisimple parameter $s$ in Theorem \ref{thm:Langlands} plays a similar role in the representation theory of $\mathbf{G}(\sfk)$ as the \emph{infinitesimal character} $\lambda \in \mathfrak{t}^*/W$ of an irreducible $\mathbf{G}(\RR)$-representation. If $X = X(s,n,\rho)$, it is thus customary to call $s$ the \emph{infinitesimal character} of $X$. Pursuing this analogy, we say, following \cite{BM1}, that $s$ is \emph{real} (or that $X$ has \emph{real infinitesimal character}) if $s \in T_{\RR}^{\vee}$, see (\ref{eq:real}). 

\end{rmk}

For $s\in W\backslash T^\vee$, write
\begin{equation}\label{e:inf-char-packet}
\begin{aligned}
    \Pi^{\mathsf{Lus}}_{s}(\bfG(\mathsf k)) &:= \{X(s,n,\rho) \mid (s,n,\rho) \in \Phi(\mathbf{G}(\mathsf{k})) \} \subset \Pi^{\mathsf{Lus}}(\mathbf{G}(\mathsf{k})),\\
     \Pi^{\mathsf{Lus}}_{s}(\bfG(\mathsf k))_0 &:= \{X(s,n,\rho)\in  \Pi^{\mathsf{Lus}}_{s}(\bfG(\mathsf k)) \mid \rho\in \mathrm{Irr}(A(s,n))_0\}.
    \end{aligned}
\end{equation}
By Theorem \ref{thm:Langlands}, there is a bijection between $\Pi^{\mathsf{Lus}}_{s}(\bfG(\mathsf k))$ and pairs $(\mathcal C,\mathcal E)$, where $\mathcal C$ ranges over the finite set of $G^\vee(s)$-orbits on $\mathfrak g^\vee_q$, and $\mathcal E$ is an irreducible $G^\vee(s)$-equivariant local system on $\mathcal C$ with trivial $Z(G^\vee)$-action.

Recall the \emph{Iwahori-Hecke algebra} associated to $\mathbf{G}(\sfk)$
$$\mathcal H_{\mathbf{I}}=\{f\in C^\infty_c(\bfG(\mathsf k))\mid f(i_1gi_2)=f(g),\ i_1,i_2\in \mathbf I(\mf o)\}.$$
Multiplication in $\mathcal H_{\mathbf{I}}$ is given by convolution with respect to a fixed Haar measure of $\bfG(\mathsf k)$. Let $\mathcal C_{\mathbf{I}}(\bfG(\mathsf k))$ denote the Iwahori category, i.e. the full subcategory of $\mathcal C(\bfG(\mathsf k))$ consisting of representations $X$ such that $X$ is generated by $X^{\mathbf I(\mf o)}$. The simple objects in this category are the (irreducible) Iwahori-spherical representations. By the Borel-Casselman Theorem \cite[Corollary 4.11]{Bo}, there is an exact equivalence of categories
\begin{equation}\label{eq:mI}
    m_{\mathbf{I}}: \mathcal C_{\mathbf{I}}(\bfG(\mathsf k))\to \mathrm{Mod}(H_{\mathbf{I}}), \qquad m_{\mathbf{I}}(V) = V^{\mathbf I(\mf o)}.
\end{equation}
This equivalence induces a group isomorphism
\begin{equation}\label{eq:mIhom}
m_{\mathbf{I}}: R_{\mathbf{I}}(\mathbf{G}(\sfk)) \xrightarrow{\sim}  R(\mathcal{H}_{\mathbf{I}}),
\end{equation}
where $R_{\mathbf{I}}(\mathbf{G}(\sfk))$ (resp. $R(\mathcal{H}_{\mathbf{I}})$) is the Grothendieck group of $\mathcal{C}_{\mathbf{I}}(\mathbf{G}(\sfk))$ (resp. $\mathrm{Mod}(\mathcal{H}_{\mathbf{I}})$). The irreducible representations occurring in the sets  $\Pi^{\mathsf{Lus}}_{s}(\bfG(\mathsf k))_0$ are precisely the irreducible objects in $ \mathcal C_{\mathbf{I}}(\bfG(\mathsf k))$.

We note that the classification of $\bfG(\mathsf k)$-representations with real infinitesimal character is invariant under isogenies. More precisely, let \[f:\bfG'\to\bfG\] be an isogeny. Then $f$ induces an isogeny $T^\vee\to T'^\vee$ and hence a map $W\backslash T^\vee\to W\backslash T'^\vee$, see \cite[\S 5.3]{Re-euler}. Suppose $\tilde s\in W\backslash T'^\vee_{\mathbb R}$ and $s\in  W\backslash T^\vee_{\mathbb R}$ correspond. We may identify $\mathfrak g^\vee= \mathfrak g'^\vee$ and $ \mathfrak g^\vee_q= \mathfrak g'^\vee_q$. Let $n\in \mathfrak g^\vee_q$ and $x\in \mathfrak t^\vee_{\mathbb R}$, $[x,n]=n$, such that $s$ (resp. $\tilde s$) equals the $q$-exponential of $x$ in $G^\vee$ (resp. $G'^\vee$). Then $\mathcal B^\vee_{s,n}=\mathcal B'^\vee_{\tilde s,n}=\mathcal B'^\vee_{x,n}.$ Since the center of $G^\vee$ (resp. $G'^\vee$) acts trivially on this variety, it follows that $\mathrm{Irr}(A(s,n))_0=\mathrm{Irr}(A(\tilde s,n))_0$. Hence there is a bijection
\begin{equation}\label{e:real-identify}
\Pi^{\mathsf{Lus}}_{\tilde s}(\bfG'(\mathsf k))_0\cong  \Pi^{\mathsf{Lus}}_{s}(\bfG(\mathsf k))_0, \quad X'(\tilde s,n,\rho)\leftrightarrow X(s,n,\rho).
\end{equation}
More precisely, in terms of the underlying representation theory, let $H_{\mathbf{I}}$ and $H_{\mathbf{I}}'$ be the Iwahori-Hecke algebras of $\bfG$ and $\bfG'$, respectively. The isogeny $f$ induces an embedding of Iwahori-Hecke algebras \cite[\S1.4,1.5]{Re-isogeny}
\[f_*: H_{\mathbf{I}}'\hookrightarrow H_{\mathbf{I}}.
\]

\begin{prop}[{\cite[Lemma 5.3.1]{Re-euler}}]\label{p:real-identify} With the notation above, the inclusion $f_*$ induces a bijection $\Pi^{\mathsf{Lus}}_{s}(\bfG(\mathsf k))_0\cong  \Pi^{\mathsf{Lus}}_{\tilde s}(\bfG'(\mathsf k))_0$: the restriction of the irreducible $H_{\mathbf{I}}$-module $X(s,n,\rho)^{\mathbf{I}(\mathfrak o)}$ to $H_{\mathbf{I}}'$ is irreducible and it equals $X'(\tilde s,n,\rho)$.

\end{prop}

For each parameter $(s,n,\rho) \in \Phi(\mathbf{G}(\mathsf{k}))_0$, there is an associated \emph{standard representation} $Y(s,n,\rho) \in \mathcal{C}(\mathbf{G}(\mathsf{k}))$. 
If $\bfG$ is adjoint, the relevant results are \cite[Theorems 7.12, 8.2, 8.3]{KL}, see also \cite[\S8.1]{Chriss-Ginzburg}.
There is an identity in the Grothendieck group $R(\mathbf{G}(\mathsf{k}))$ (see \cite[Theorem 8.6.15]{Chriss-Ginzburg} or \cite[Proposition 10.5]{Lu-gradedII}):
\begin{equation}\label{e:multi}
    Y(s,n,\rho)=X(s,n,\rho)+\sum_{(n',\rho')} m_{(n,\rho),(n',\rho')} X(s,n',\rho'),
\end{equation}
where $m_{(n,\rho),(n',\rho')}\in \mathbb Z_{\ge 0}$ and $n'$ ranges over the set of representatives of $G^\vee(s)$-orbits in $\mathfrak g^\vee_q$ such that $n\in \partial(G^\vee(s)n')$, and $\rho'\in \mathrm{Irr}(A(s,n'))_0.$

If $s$ is real, in light of Proposition \ref{p:real-identify}, we can extend (\ref{e:multi}) to $\bfG'$ by defining the standard $\bfG'(\sfk)$-module \[Y'(\tilde s,n,\rho):=m_{\bfI'}^{-1}( Y(s,n,\rho)^{\bfI}|_{\mathcal H_{\bfI}'}).\]

\subsection{The Iwahori-Hecke algebra}\label{s:Hecke}

Let $c_0$ be the chamber of $\mathcal A(\bfT,k)$ cut out by $\tilde \Delta$ and let $\bfI$ be the Iwahori subgroup corresponding to $c_0$.
Suppose $\mathbf{P}_c$ is a parahoric subgroup containing $\mathbf{I}$ with pro-unipotent radical $\mathbf{U}_c$ and reductive quotient $\mathbf{L}_c$.
The finite Hecke algebra $\mathcal H_{c}$ of $\bfL_c(\mathbb F_q)$ embeds as a subalgebra of $\mathcal H_{\bfI}$.
For $X\in \mathcal C_{\bfI}(\bfG(\sfk))$, the Moy-Prasad theory of unrefined minimal $K$-types \cite{moyprasad} implies that the finite dimensional $\bfL_c(\mathbb F_q)$-representation $X^{\bfU_{c}(\mf o)}$ is a sum of principal series unipotent representations and so corresponds to an $\mathcal H_{c}$-module with underlying vector space 
$$(X^{\bfU_c(\mf o)})^{\bfI(\mf o)/\bfU_c(\mf o)} = X^{\bfI(\mf o)}.$$
The $\mathcal H_c$-module structure obtained in this manner coincides naturally with that of
$$\Res_{\mathcal H_c}^{\mathcal H_\bfI}m_\bfI(X).$$
Let 
$$\mathfrak J_c:\mathcal H_c\to \CC[W_c]$$
be the isomorphism introduced by Lusztig in \cite{lusztigdeformation}.
Given any $\mathcal H_c$-module $M$ we can use the isomorphism $\mathfrak J_c$ to obtain a $W_c$-representation which we denote by $M_{q\to1}$.
Define
\begin{equation}
    \label{eq:Wcqto1}
    X|_{W_c}:=(\Res_{\mathcal H_c}^{\mathcal H_\bfI}m_\bfI(X))_{q\to 1}.
\end{equation}

We will need to recall some structural facts about the Iwahori-Hecke algebra. Let $\CX:=X_*(\mathbf T,\bar{\mathsf k})=X^*(\mathbf{T}^\vee,\bar{\sfk})$ and consider the (extended) affine Weyl group $\widetilde{W} := W \ltimes \CX$. Let 
$$S := \{s_{\alpha} \mid \alpha \in \Delta\} \subset W$$
denote the set of simple reflections in $W$. For each $x \in \CX$, write $t_x \in \widetilde{W}$ for the corresponding translation. If $W$ is irreducible, let $\alpha_0$ be the highest root and set $s_0=s_{\alpha_0} t_{-\alpha_0^\vee}$, $S^a=S\cup\{s_0\}$. If $W$ is a product, define $S^a$ by adjoining to $S$ the reflections $s_0$, one for each irreducible factor of $W$. Consider the length function $\ell: \widetilde{W} \to \ZZ_{\geq 0}$ extending the usual length function on the affine Weyl group $W^a=W\ltimes \mathbb Z \Phi^\vee$
\[\ell(w t_x)=\sum_{\substack{\alpha\in \Phi^+\\w(\alpha)\in \Phi^-}} |\langle x,\alpha\rangle+1|+\sum_{\substack{\alpha\in \Phi^+\\w(\alpha)\in \Phi^+}} |\langle x,\alpha\rangle|.
\]
For each $w \in \widetilde{W}$, choose a representative $\dot w$ in the normalizer $N_{\bfG(\mathsf k)}(\mathbf I(\mf o))$. Recall the Bruhat decomposition
\[\bfG(\mathsf k)=\bigsqcup_{w\in \widetilde W} \mathbf I(\mf o) \dot w \mathbf I(\mf o), 
\]
For each $w \in \widetilde{W}$, write $T_w \in \mathcal{H}_{\mathbf{I}}$ for the characteristic function of $\mathbf I(\mf o) \dot w \mathbf I(\mf o) \subset \mathbf{G}(\sfk)$. Then $\{T_w\mid w\in \widetilde W\}$ forms a $\mathbb C$-basis for $\mathcal H_{\mathbf{I}}.$

The relations on the basis elements $\{T_w \mid w \in \widetilde{W}\}$ were computed in \cite[Section 3]{IM}:
\begin{equation}\label{eq:relations}
    \begin{aligned}
    &T_w\cdot T_{w'}=T_{ww'}, \qquad \text{if }\ell(ww')=\ell(w)+\ell(w'),\\
    &T_s^2=(q-1) T_s+q,\qquad s\in S^a.
    \end{aligned}
\end{equation}

Let $R$ be the ring $\CC[v,v^{-1}]$ and for $a\in \CC^*$ let $\CC_a$ be the $R$-module $R/(v-a)$.
Let $\mathcal H_{\bfI,v}$ denote the Hecke algebra with base ring $R$ instead of $\CC$ and where $q$ is replaced with $v^2$ in the relations (\ref{eq:relations}).
By specializing $v$ to $\sqrt{q}$, $1$, we obtain
\begin{equation}
    \mathcal H_{\bfI,v}\otimes_R\CC_{\sqrt q} \cong \mathcal H_{\bfI}, \qquad \mathcal H_{\bfI,v}\otimes_R\CC_1 \cong \CC \widetilde W.
\end{equation}
Suppose $Y=Y(s,n,\rho)$ is a standard Iwahori-spherical representation, see Section \ref{s:unip-cusp}, and let $M=m_\bfI(Y)$. 
By \cite[Section 5.12]{KL} there is a $\mathcal H_{\bfI,v}$ module $M_v$, free over $R$, such that 
$$M_v\otimes_R\CC_{\sqrt q} \cong M$$
as $\mathcal H_\bfI$-modules.
We can thus construct the $\widetilde W$-representation
$$Y_{q\to 1}:= M_v\otimes_R \CC_{1}.$$
Let $R(\widetilde{W})$ be the Grothendieck group of $\mathrm{Rep}(\widetilde W)$. Since the standard modules form a $\ZZ$-basis for $R_\bfI(\bfG(\sfk))$, the Grothendieck group of $\mathcal C_\bfI(\bfG(\sfk))$, there is a unique homomorphism
\begin{equation}\label{eq:qto1hom}(\bullet)_{q\to 1}: R_\bfI(\bfG(\sfk)) \to R(\widetilde W)\end{equation}
extending $Y \mapsto Y_{q \to 1}$. Moreover, since 
$$\Res_{W_c}^{\widetilde W} Y_{q\to 1} = Y|_{W_c}$$
for the Iwahori-spherical standard modules we have that
\begin{equation}
    \label{eq:heckerestriction}
    \Res_{W_c}^{\widetilde W}X_{q\to 1} = X|_{W_c}
\end{equation}
for all $X\in R_\bfI(\bfG(\sfk))$.
Finally, let $c$ be the face of $c_0$ corresponding to $\Delta\subseteq\tilde\Delta$.
Then $W_c = W$ and equation \ref{eq:Wcqto1} in this special case gives rise to a map
\begin{equation}
    \label{eq:restrictiontoW}
    |_W:R_\bfI(\bfG(\sfk))\to R(W)
\end{equation}
where $R(W)$ is the Grothendieck group of $\mathrm{Rep}(W)$.

There is a second description of $\mathcal{H}_{\mathbf{I}}$ in terms of generators and relations, called \emph{the Bernstein presentation} \cite{Lu-graded}. Set
\[\CX_+=\{x\in \CX\mid \langle\alpha,x\rangle\ge 0,\text{ for all }\alpha\in \Phi^+\}.
\]
For $x\in \CX_+$, set $\theta_x=q^{-\ell(t_x)/2} T_{t_x}$. If $x\in \CX$, write $x=x_1-x_2$ for $x_1,x_2\in \CX_+$ and define
\[\theta_x=\theta_{x_1}\theta_{x_2}^{-1}.
\]
Then $\theta_x\theta_{x'}=\theta_{x+x'}$ for all $x,x'\in \CX$ and $\theta_0=1$. Let $\mathcal A$ denote the abelian subalgebra of $\mathcal H_{\mathbf{I}}$ generated by the elements $\theta_x$. The Bernstein basis of $\mathcal H_{\mathbf{I}}$ is given by $T_w \theta_x$, for $w\in W$ and $x\in \CX$. The cross-relations are \cite{Lu-graded}
\begin{equation}
    T_{s_\alpha}\theta_x-\theta_{ s_\alpha(x)}T_{ s_\alpha}=(q-1)\frac{\theta_x-\theta_{s_\alpha(x)}}{1-\theta_{-\alpha^\vee}}, \qquad s_{\alpha} \in S
\end{equation}

\subsection{Aubert-Zelevinsky duality}\label{sec:AZduality}

There are three interrelated involutions that appear in the study of smooth $\mathbf{G}(\sfk)$-representations. The first is an algebra involution $\tau$ of the Hecke algebra $\mathcal{H}_{\mathbf{I}}$. It is defined on generators by
\[\tau(T_w)=(-1)^{\ell(w_1)}q^{\ell(w)} T_{w^{-1}}^{-1}, \qquad w=w_1 t_x,\ w_1\in W,\ x\in \CX.
\]
Let $w_0$ denote the longest element of $W$. Since 
$$\ell(w_0 t_x)=\ell(w_0)+\ell(t_x)=\ell(t_{w_0(x)})+\ell(w_0), \text{ for every } x\in \CX_+$$
we have that
\[T_{w_0(x)}=T_{w_0} T_{t_x} T_{w_0}^{-1},\qquad x\in X_+.
\]
Then if $x\in \CX_+$
\[\tau(T_{t_{w_0(x)}})=q^{\ell(t_{w_0(x)})} T_{t_{-w_0(x)}}^{-1}=q^{\ell(t_x)/2} \theta_{w_0(x)}.
\]
Hence
\[\tau(\theta_x)= T_{w_0} \theta_{w_0(x)} T_{w_0}^{-1}.
\]
So on the Bernstein generators, $\tau$ is given by
\begin{equation}\label{e:tau-Bernstein}
    \tau(T_w)=(-q)^{\ell(w)} T_{w^{-1}}^{-1},\ w\in W,\qquad \tau(\theta_x)=T_{w_0} \theta_{w_0(x)} T_{w_0}^{-1},\ x\in \CX.
\end{equation}
If $M$ is a simple $\mathcal H_{\mathbf{I}}$-module $M$, write $M^\tau$ for the module which is equal to $M$ as a vector space but with $\mathcal{H}_{\mathbf{I}}$-action twisted by $\tau$.

The center of $\mathcal H_{\mathbf{I}}$ is (\cite[Proposition 3.11]{Lu-graded})
\[Z(\mathcal H_{\mathbf{I}})=\mathcal A^W,
\]
and therefore by (\ref{e:tau-Bernstein})
\begin{equation}
    \tau(z)=z, \text{ for all }z\in Z(\mathcal H_{\mathbf{I}}).
\end{equation}
Every simple $\mathcal H_{\mathbf{I}}$-module $M$ is necessarily finite dimensional. So by Schur's Lemma, $Z(\mathcal H_{\mathbf{I}})$ acts on $M$ by a central character $\chi_M\in W\backslash T^\vee$. Then 
\begin{equation}\label{e:tau-cc}
\chi_{M^\tau}=\chi_{M}.
    \end{equation}

\medskip

The second involution we will consider is an involution $D$ of the Grothendieck group $R(\mathcal H_{\mathbf{I}})$ of $\mathcal H_{\mathbf{I}}$-modules. This involution can be defined in the following manner. For every $J\subset S$, let $\mathcal H_{\mathbf{I}}(J)$ denote the (parabolic) subalgebra of $\mathcal H_{\mathbf{I}}$ generated by $\mathcal A$ and $\{T_{s} \mid s\in J\}$. If $M$ is an $\mathcal H_{\mathbf{I}}$-module, let $\mathsf{Res}_J(M)$ denote the restriction of $M$ to $\mathcal H_{\mathbf{I}}(J)$. If $N$ is an $\mathcal H_{\mathbf{I}}(J)$-module, let $\mathsf{Ind}_J(N)=\mathcal H_{\mathbf{I}}\otimes_{\mathcal H_{\mathbf{I}}(J)} N$ denote the induced $\mathcal{H}_{\mathbf{I}}$-module. Then $D$ is the involution
\begin{equation}
    D: R(\mathcal{H}_{\mathbf{I}}) \to R(\mathcal{H}_{\mathbf{I}}), \qquad  D(M)=\sum_{J\subset S} (-1)^{|J|} ~\mathsf{Ind}_J(\mathsf{Res}_J(M)).
\end{equation}
By \cite[Theorem 2]{Kato}
\begin{equation}
    D(M)=M^\tau, \text{for all }\mathcal H_{\mathbf{I}}\text{-modules }M.
\end{equation}

The involution $D$ has nice behavior with respect to unitarity. Let $*:\mathcal H_{\mathbf{I}}\to \mathcal H_{\mathbf{I}}$ be the conjugate-linear anti-involution defined on generators by
\begin{equation}
    T_w^*=T_{w^{-1}},\qquad w\in \widetilde W.
\end{equation}
A non-degenerate sesquilinear form $(~,~)_M$ on $M$ is \emph{Hermitian} if
\[(h\cdot m,m')_M=(m,h^*\cdot m'),\qquad m,m'\in M,\ h\in \mathcal H_{\mathbf{I}}.
\]

\medskip

The final involution we will consider is an involution $\AZ$ on the Grothendieck group $R(\mathbf{G}(\sfk))$ of smooth $\mathbf{G}(\sfk)$-representation, called \emph{Aubert-Zelevinsky duality} \cite[\S1]{Au}. This involution can defined in the following manner. Let $\mathcal Q$ denote the set of parabolic subgroups of $\bfG$ defined over $\mathsf k$ and containing $\mathbf B$. For every $\mathbf Q\in\mathcal Q$, let $i_{\mathbf Q(\mathsf k)}^{\mathbf{G}(\mathsf k)}$ and $\mathsf{r}_{\mathbf Q(\mathsf k)}^{\mathbf{G}(\mathsf k)}$ denote the normalized parabolic induction functor and the normalized Jacquet functor, respectively. Then $\AZ$ is defined by
\begin{equation}
    \mathsf{AZ}: R(\mathbf{G}(\mathsf{k})) \to R(\mathbf{G}(\mathsf{k})), \qquad \mathsf{AZ}(X)=\sum_{\mathbf Q\in \mathcal Q} (-1)^{r_{\mathbf Q}} ~i_{\mathbf Q(\mathsf k)}^{\bfG(\mathsf k)}(\mathsf{r}_{\mathbf Q(\mathsf k)}^{\bfG(\mathsf k)}(X)),
\end{equation}
where $r_{\mathbf Q}$ is the semisimple rank of the reductive quotient of $\mathbf Q$. 
If a class $X \in R(\mathbf{G}(\sfk))$ is irreducible, irreducible with unipotent cuspidal support, or Iwahori-spherical, so is $\AZ(X)$ (up to multiplication by a sign). More generally, $\AZ$ is an involution on the set of irreducible representations (up to sign) in any Bernstein component, see for example \cite[\S3.2]{BBK}. For the basic properties of parabolic induction and Jacquet functors relative to Berstein components, see \cite[\S1]{Au} or \cite{Roc-parabolic}.

Under the Borel-Casselman equivalence (\ref{eq:mI}) the induction/restriction functors correspond
\[i_{\mathbf Q(\mathsf k)}^{\mathbf{G}(\mathsf k)}\leftrightarrow \mathsf{Ind}_J,\qquad \mathsf{r}_{\mathbf Q(\mathsf k)}^{\mathbf{G}(\mathsf k)}\leftrightarrow \mathsf{Res}_J,
\]
Here, $J$ is the subset of $S$ corresponding to the parabolic $\mathbf Q$. In particular,
\begin{equation}
    m_{\mathbf{I}}(\mathsf{AZ}(X))=D(m_{\mathbf{I}}(X))=(m_{\mathbf{I}}(X))^\tau,
\end{equation}
for all irreducible objects $X\in \mathcal C_{\mathbf{I}}(\bfG(\mathsf k))$.


\subsection{Wavefront sets}\label{s:wave}
Let $X$ be an admissible smooth representation of $\bfG(\sfk)$.
Recall that for each nilpotent orbit $\OO\in \mathcal N_o(\sfk)$ there is an associated distribution $\mu_\OO$ on $C_c^\infty(\mf g(\sfk))$ called the \emph{nilpotent orbital integral} of $\OO$ \cite{rangarao}.
Write $\hat\mu_\OO$ for the Fourier transform of this distribution.
By the Harish-Chandra-Howe local character expansion \cite[Theorem 16.2]{HarishChandra1999} there is an open neighbourhood $\mathcal{V}$ of the identity $1 \in \mathbf{G}(\sfk)$ and (unique) complex numbers $c_{\OO}(X)\in \CC$ for each $\OO\in \mathcal N_o(\sfk)$ such that
\begin{equation}
    \Theta_{X}(f) = \sum_{\OO \in \mathcal N_o(\sfk)}c_{\OO}(X) \hat \mu_{\OO}(f\circ \exp)
\end{equation}
for all $f\in C_c^\infty(\mathcal V)$.
The \textit{($p$-adic) wavefront set} of $X$ is
$$\WF(X) := \max\{\OO \mid  c_{\OO}(X)\ne 0\} \subseteq \mathcal N_o(\sfk).$$
The \emph{algebraic wavefront set} of $X$ is
$$^{\bar k}\WF(X) := \max \{\mathcal N_o(\bark/\sfk)(\OO) \mid c_{\OO}(X)\ne 0\} \subseteq \mathcal N_o(\bark),$$
see \cite[p. 1108]{Wald18} (note: in \cite{Wald18}, the invariant $^{\bar k}\WF(X)$ is called simply the `wavefront set' of $X$). 

By analogy with the case of real reductive groups (\cite[Theoerem 3.10]{Joseph1985}) and finite groups of Lie type (\cite[Theorem 11.2]{lusztigunip}) it is expected that when $X$ is irreducible, $^{\bark}\WF(X)$ is a singleton $\{\OO\}$ and 
$$\mathcal N_o(\bark/\sfk)(\OO') = \OO \text{ for all } \OO'\in \WF(X).$$
In \cite[Section 5.1]{okada2021wavefront} the third author has introduced a third type of wavefront set for depth $0$ representations, called
the \emph{canonical unramified wavefront set}. This invariant is a natural refinement of $\hphantom{ }^{\bar{\sfk}}\WF(X)$. We will now define $^K\WF(X)$ and explain how to compute it. Fix the notation of Sections \ref{subsec:nilpotent} and \ref{subsec:Wreps}, e.g. $\mathcal{L}$, $\mathbb{L}$, $\sim_A$, $\OO(\bullet,\bullet)$, $\phi(\bullet)$, and so on. 

For every face $c \subset \mathcal{B}(\mathbf{G})$, the space of invariants $X^{\bfU_c(\mf o)}$ is a (finite-dimensional) $\bfL_c(\mathbb F_q)$-representation. Let $\WF(X^{\bfU_c(\mf o)}) \subseteq \cN^{\bfL_c}_o(\overline{\mathbb F}_q)$ denote the Kawanaka wavefront set \cite{kawanaka} and let 
\begin{equation}
    ^K\WF_c(X) := \{\mathscr L (c,\OO) \mid \OO\in \WF(X^{\bfU_{c}(\mf o)})\} \subseteq \cN_o(K).
\end{equation}


\begin{definition}\label{def:CUWF}
    The \emph{canonical unramified wavefront set} of $X$ is
    \begin{equation}
        ^K\WF(X) := \max \{[\hphantom{ }^K\WF_c(X)] \mid  c\subseteq \mathcal B(\bfG)\} \subseteq \cN_o(K)/\sim_A.
    \end{equation}
\end{definition}
\begin{rmk}
    Definition \ref{def:CUWF} yields an invariant which is a bit coarser than the one given in \cite{okada2021wavefront}, but is arguably more natural, and the two are straightfowardly related by an application of the map $[\bullet]:\mathcal N_o(K)\to \mathcal N_o(K)/\sim_A$.
\end{rmk}
\begin{theorem}
    \cite[Theorem 5.4]{okada2021wavefront}\label{thm:OkadaWFset} Let $c_0$ be a chamber of $\mathcal B(\bfG)$.
    Then
    \begin{equation}
        ^K\WF(X) = \max \{\hphantom{ }^K\WF_c(X) \mid c\subseteq c_0\}
    \end{equation}
    where $c\subseteq c_0$ means that $c$ is a face of $c_0$.
\end{theorem}

We will often want to view $^K\WF(X)$ and $[^K\WF_c(X)]$ as subsets of $\mathcal N_{o,\bar c}$ using the identification $$\bar\theta_{\bfT}:\mathcal N_o(K)/\sim_A\to \mathcal N_{o,\bar c}.$$ 
We will write $^{K}\WF(X,\CC)$ (resp. $^K\WF_c(X,\CC)$) for $\bar \theta_{\bfT}(\hphantom{ }^K\WF(X))$ (resp. $\bar\theta_{\bfT}([\hphantom{ }^K\WF_c(X)])$).
We will also want to view $\hphantom{ }^{\bar{\sfk}}\WF(X)$, which is naturally contained in $\mathcal N_o(\bar k)$, as a subset of $\mathcal{N}_o$ and so will write $\hphantom{ }^{\bar{\sfk}}\WF(X,\CC)$ for $\Theta_{\bar k}(\hphantom{ }^{\bar{\sfk}}\WF(X))$.
The canonical unramified wavefront set is a refinement of the usual wavefront set in the following sense.
\begin{theorem}
    \cite[Theorem 5.2]{okada2021wavefront} 
    Let $X$ be a depth $0$ representation.
    Then
    \begin{equation}
        \hphantom{ }^{\bar{\sfk}}\WF(X,\CC) = \max (\pr_1(\hphantom{ }^{K}\WF(X,\CC))).
    \end{equation}
\end{theorem}
The proof uses the test functions from \cite[Theorem 4.5]{barmoy}.
If $^{K}\WF(X)$ is a singleton, $\hphantom{ }^{\bark}\WF(X)$ is also a singleton and 
\begin{equation}
    \label{eq:wfcompatibility}
    \hphantom{ }^{K}\WF(X,\CC) = (\hphantom{ }^{\bark}\WF(X,\CC),\bar C).
\end{equation}
for some conjugacy class $\bar C$ in $\bar A(\hphantom{ }^\bark\WF(X,\CC))$.

We will use the following theorem to compute the local wavefront sets.
\begin{theorem}
    \cite[Theorem 3.9]{okada2021wavefront}
    \label{thm:locwf}
    Suppose $X\in\mathcal C_\bfI(\bfG(\sfk))$ and $m_\bfI(X)$ is deformable.
    Let $J\subsetneq \tilde\Delta$ and $c = c(J)$. 
    Then 
    \begin{equation}
        \label{eq:locwf}
        [^K\WF_{c}(X)]= \max\{[\mathscr L(c,\OO_{\phi(E)}(\overline{\mathbb F}_q)] \mid E\in\mathrm{Irr}(W_{c}),\Hom(E,\Res_{W_{c}}^{\widetilde W}(X_{q\to1}))\ne 0\}.
    \end{equation}
\end{theorem}
The notion of `deformable' in the statement of the theorem is a technical condition on $m_\bfI(X)$ (see \cite[Section 3.3]{okada2021wavefront}).
In light of the extension of the $q\to 1$ operation (\ref{eq:qto1hom}) to the full Grothedieck group $R(\mathcal H_\bfI)$ and Equation \ref{eq:heckerestriction}, we may drop the requirement that $m_\bfI(X)$ is deformable from Theorem \ref{thm:locwf}.

In this paper we will primarily be concerned with $^K\WF_{c}(X,\CC)$. By applying $\bar\theta_{\bfT}$ to Equation (\ref{eq:locwf}) we get
\begin{equation}
    ^K\WF_c(X,\CC)= \max\{\overline{\mathbb L}(J,\OO(\phi(E),\CC)) \mid E\in\mathrm{Irr}(W_c), \ \Hom(E,\Res_{W_c}^{\widetilde W}(X_{q\to1}))\ne 0\}
\end{equation}
since $\bar\theta_{\bfT}$ is an isomorphism of partial orders and 
\begin{align}
    \bar\theta_{\bfT}([\mathscr L(c,\OO(\phi(E'),\overline{\mathbb F}_q))]) &= \mf Q(\theta(\mathscr L(c,\OO(\phi(E'),\overline{\mathbb F}_q)))) &&\text{(Equation (\ref{eq:square}))} \nonumber \\
    &= \mf Q(\mathbb L(J,\Theta_{c(J)}(\OO(\phi(E'),\overline{\mathbb F}_q))))  &&\text{(Proposition \ref{prop:square})} \nonumber \\
    &= \overline{\mathbb L}(J,\OO(\phi(E'),\CC)).
\end{align}

\section{Computation of wavefront sets}\label{sec:main-result}

Our main result is a general formula for the (canonical unramified) wavefront set of an irreducible Iwahori-spherical representation with real infinitesimal character. To prove this result, it will be convenient to define the \emph{wavefront set} of a Weyl group representation (analogous to the invariant $^K\WF(X)$ defined in Section \ref{s:wave}). Fix the notation of Sections \ref{subsec:nilpotent} and \ref{subsec:Wreps}, e.g. $\overline{\mathbb{L}}$, $\phi(\bullet)$, $\OO(\bullet,\CC)$, and so on. For $E\in\mathrm{Irr}(W)$ and $J\subsetneq\tilde\Delta$, consider the subset
    $$\WF_J(E) = \max\{\overline{\mathbb L}(J,\OO(\phi(E'),\CC)) \mid E'\in\mathrm{Irr}(W_J), \ \Hom(E',E|_{W_J})\ne 0\} \subset \cN_{o,\bar c}.$$

\begin{definition}
\label{def:wWF}
The \emph{wavefront set} of $E$ is
$$\WF(E) = \max \{\WF_J(E) \mid J\subsetneq \tilde\Delta\} \subset \cN_{o, \bar c}.$$
\end{definition}

\begin{lemma}
    \label{lem:upperbound}
    Let $E\in \mathrm{Irr}(W)$ and $\OO^\vee = \OO^{\vee}(E,\CC)$. Then for any $J\subsetneq \tilde\Delta$ we have that 
    \begin{equation}
        \WF_J(E)\le_A d_A(\OO^\vee,1).
    \end{equation}
\end{lemma}

\begin{proof}
If $E'$ is an irreducible constituent of $E|_{W_J}$, then by \cite[Proposition 4.3]{acharaubert} we have that $\OO^\vee\le d_S(\mathbb L(J,\OO(\phi(E'),\CC)))$.
Thus for all $(\OO,\bar C) \in \WF_J(E)$ we have that $d_S(\OO,\bar C)\ge \OO^\vee$ and so by \cite[Theorem 5.5 (3)]{okada2021wavefront} we have that $(\OO,\bar C)\le_A d_A(\OO^\vee,1)$.
\end{proof}

We now introduce the notion of a \emph{faithful nilpotent orbit}, which will play a central role in our calculation of $\WF(E)$. 

\begin{definition}\label{def:faithful}
Let $\OO^{\vee} \subset \cN^{\vee}$ be a nilpotent orbit. We say that $\OO^{\vee}$ \emph{faithful} if there is a pair $(J, \phi) \in \mathcal{F}_{\tilde\Delta}$ such that
\begin{itemize}
    \item[(i)] $\overline{\mathbb L}(J,\OO(\phi)) = d_A(\OO^{\vee},1)$.
    \item[(ii)] If $E$ is an irreducible representation of $W$ with $\OO^{\vee}(E,\CC) = \OO^{\vee}$, there is a representation $F \in \phi \otimes \mathrm{sgn}$ such that
     $\Hom_{W_J}(F,E|_{W_J}) \neq 0$.
\end{itemize}
\end{definition}

\begin{theorem}\label{thm:faithful}
Every nilpotent orbit $\OO^{\vee} \subset \cN^{\vee}$ is faithful.
\end{theorem}

The proof of Theorem \ref{thm:faithful} will be postponed until Section \ref{sec:faithful}. Using this result, and Lemma \ref{lem:upperbound}, we can compute the the wavefront set of any irreducible $W$-representation.

\begin{cor}
\label{cor:wWF}
Let $E\in\mathrm{Irr}(W)$. Then
$$\WF(E) = d_A(\OO^{\vee}(E,\CC),1).$$
\end{cor}

\begin{proof}
    By Lemma \ref{lem:upperbound}, $\WF(E)$ is bounded above by $d_A(\OO^\vee(E,\CC),1)$. It suffices to show that this bound is achieved.
    For this let $\OO^\vee = \OO^\vee(E,\CC)$.
    By Theorem \ref{thm:faithful}, $\OO^\vee$ is faithful. So there is a pair $(J,\phi)\in \mathcal F_{\tilde\Delta}$ such that $\overline{\mathbb L}(J,\OO(\phi,\CC)) = d_A(\OO^\vee,1)$ and there is a representation $F\in \phi\otimes \mathrm{sgn}$ such that $\Hom(F,E|_{W_J})\ne0$.
    But since $F\in\phi\otimes\mathrm{sgn}$ we have that $\phi(F) = \phi$ and so $\overline{\mathbb L}(J,\OO(\phi(F),\CC)) = d_A(\OO^\vee,1)$ as required.
\end{proof}

We now turn to the task of computing wavefront sets of Iwahori-spherical $\mathbf{G}(\sfk)$-representations. Recall the group homomorphisms (\ref{eq:qto1hom}, \ref{eq:restrictiontoW})
$$(\bullet)_{q\to 1}: R_{\mathbf{I}}(\mathbf{G}(\sfk)) \to R(\widetilde{W}), \qquad |_W: R_{\mathbf{I}}(\mathbf{G}(\sfk)) \to R(W)$$
\begin{lemma}\label{l:WF-factors}
Let $X \in \mathcal{C}_{\mathbf{I}}(\mathbf{G}(\sfk))$ and suppose 
\begin{equation}
    \Res_{W_c}^{\widetilde W}X_{q\to 1} = \Res_{\dot W_c}^{W}X|_W
\end{equation}
for all $c\subseteq c_0$. Let 
\begin{equation}
    \mathscr O^\vee = \min\{\OO^\vee(E,\CC) \mid E\in \mathrm{Irr}(W),\Hom(E,X|_{W})\ne0\}.
\end{equation}
Then 
\begin{equation}
    ^K\WF(X,\CC) = \{d_A(\OO^\vee,1) \mid \OO^\vee\in\mathscr O^\vee\}.
\end{equation}
\end{lemma}
\begin{proof}
    Let $J\subsetneq \tilde\Delta$. For $E \in \mathrm{Irr}(W)$,
    let $n_E(X)$ denote the multiplicity of $E$ in $X|_W$, so that 
    \begin{equation}\label{eq:isotypic}X|_W \simeq \bigoplus_{E\in \mathrm{Irr}(W)}n_E(X) E.\end{equation}
    Since $\Res_{W_c}^{\widetilde W}X_{q\to 1} = \Res_{\dot W_c}^{W}X|_W$, by Theorem \ref{thm:locwf} we have that
    \begin{equation}
        ^K\WF_{c(J)}(X,\CC) = \max\{\overline{\mathbb L}(J,\OO(\phi(E'),\CC)) \mid E'\in\mathrm{Irr}(W_c), \ \Hom(E',\Res_{\dot W_c}^{W}(X|_{W}))\ne 0\}.    
    \end{equation}
    But by (\ref{eq:isotypic}), we have $\Hom(E',\Res_{\dot W_c}^{W}(X|_{W}))\ne 0$ if and only if there is a representation $E\in\mathrm{Irr}(W)$ with $n_E(X)\ne0$ such that $\Hom(E',E|_{\dot W_c})\ne 0$.
    We can therefore rewrite $^K\WF_{c(J)}(\pi,\CC)$ as
    \begin{align}
    \label{eq:locWFsimplification}
    \begin{split}  
        ^K\WF_{c(J)}(X,\CC)  &= \max\{\overline{\mathbb L}(J,\OO(\phi(E'),\CC)) \mid E\in \mathrm{Irr}(W), \ E'\in\mathrm{Irr}(W_c), \ n_E(X)\ne 0, \ \Hom(E',E|_{\dot W_c})\ne 0 \} \\
        &= \max \{\WF_J(E)\mid E\in \mathrm{Irr}(W), \ n_E(X)\ne0\}.
    \end{split}
    \end{align}
    Thus
    \begin{align*}
        ^K\WF(X,\CC) &= \max\{\hphantom{ }^K\WF_{c(J)}(X,\CC) \mid J\subsetneq\tilde\Delta\}\\
        &= \max \{\WF_J(E) \mid J\subsetneq\tilde\Delta, \ n_E(X)\ne0\}&& \text{(Equation \ref{eq:locWFsimplification})}\\
        &= \max \{\WF(E) \mid n_E(X)\ne0\} && \text{(Definition \ref{def:wWF})}\\
        &= \max \{d_A(\OO^{\vee}(E,\CC),1) \mid n_E(X)\ne0\} && \text{(Corollary \ref{cor:wWF})}.
    \end{align*}
    Then by Lemma \ref{lem:injectiveachar} we get
    \begin{align*}^K\WF(X) &= \{d_A(\OO^\vee,1)\mid \OO^\vee\in \min\{\OO^{\vee}(E,\CC) \mid E\in\mathrm{Irr}(W), \ n_E(X)\ne0\}\}\\
    &= \{d_A(\OO^{\vee},1) \mid \OO^{\vee} \in \mathscr{O}^{\vee}\}
    \end{align*}
    as required.
\end{proof}

The next result is essentially contained in \cite[\S5]{Re-euler}.
\begin{lemma}\label{l:std-real}
Let $Y$ be an Iwahori-spherical standard module with real infinitesimal character. Then $\Res_{W_c}^{\widetilde W}Y_{q\to 1} = \Res_{\dot W_c}^{W}Y|_W$, in the notation of the previous lemma.
\end{lemma}

\begin{proof}
For every $x\in \CX$ and irreducible tempered representation $X$ with real infinitesimal character, the Bernstein operator $\theta_x$ acts on $m_{\mathbf{I}}(X)$ with eigenvalues of the form $v^m$, $m\in \mathbb Z$, see \cite[\S5.7]{Re-euler}. Then \cite[Lemma 5.8.1]{Re-euler} implies that the claim holds when $Y=X$ is a tempered Iwahori-spherical standard module with real infinitesimal character. It is well known, see for example \cite[Proof of Theorem 5.3]{BM1} that for any (Iwahori-spherical) standard module $Y$ there exists a continuous deformation $Y^\epsilon$, $\epsilon\in [0,1]$, $Y^1=Y$, such that $Y^0$ is a (possibly reducible, Iwahori-spherical) tempered module. Moreover if $Y$ has real infinitesimal character, so does $Y^0$. Then both $\Res_{W_c}^{\widetilde W}Y^\epsilon_{q\to 1}$ and $\Res_{\dot W_c}^{W}Y^\epsilon|_W$ are independent of $\epsilon$, being continuous deformations of finite dimensional representations of finite groups. This completes the proof.

\end{proof}

\begin{theorem}
\label{thm:realwf}
Suppose $X=X(s,n,\rho)$ is  Iwahori-spherical with $s\in T^\vee_\mathbb R$. Write $\OO^{\vee}_X = G^{\vee}n \subset \mathfrak{g}^{\vee}$. Then $^K\WF(\AZ(X))$, $\hphantom{ }^{\bar{\sfk}}\WF(\AZ(X))$ are singletons, and
\begin{align*}
^K\WF(\mathsf{AZ}(X)) &= d_A(\OO^{\vee}_X,1)\\
\hphantom{ }^{\bar{\sfk}}\WF(\AZ(X)) &= d(\OO^{\vee}_X).
\end{align*}
\end{theorem}

\begin{proof}
By Equation (\ref{eq:wfcompatibility}) and the identity
\begin{equation}
    \pr_1(d_A(\OO^\vee_X,1)) = d(\OO^\vee_X)
\end{equation}
the formula for $\hphantom{ }^{\bar{\sfk}}\WF(\AZ(X))$ follows from the formula for $^K\WF(\AZ(X))$. So it suffices to show that $^K\WF(\AZ(X)) = d_A(\OO^{\vee}_X,1)$. 

By (\ref{e:multi}), there is an equality in $R(\mathbf{G}(\sfk))$
\[X = Y(s,n,\rho) + \sum_{n \in \partial (G^{\vee}n')} m_{s,n',\rho'}Y(s,n',\rho')
\]
for standard modules $Y(s,n',\rho')$ and $m_{s,n',\rho'} \in \ZZ$. By Lemma \ref{l:std-real}, 
\[\Res_{W_c}^{\widetilde W}Y(s,n',\rho')_{q\to 1} = \Res_{\dot W_c}^{W}Y(s,n',\rho')|_W.
\]
for each $Y(s,n',\rho')$ and hence
$$\Res_{W_c}^{\widetilde W}X_{q\to 1} = \Res_{\dot W_c}^{W}X|_W $$
so $X$ satisfies the hypotheses of Lemma \ref{l:WF-factors}.
By the argument in \cite[\S10.13]{Lu-gradedII}
\begin{equation}\label{e:Y-coh1}
Y(s,n',\rho')|_W  = H^{\bullet}(\mathcal{B}^{\vee}_{n'})^{\rho'} \otimes \mathrm{sgn}.
\end{equation}
So as (virtual) $W$-representations
\[\mathsf{AZ}(X)|_W =  H^{\bullet}(\mathcal{B}^{\vee}_n)^{\rho} +  \sum_{n \in \partial (G^{\vee}n')} m_{s,n',\rho'}H^{\bullet}(\mathcal{B}_{n'}^{\vee})^{\rho'}
\]
By \cite[Corollaire 2]{BoMac1981}, every cohomology group $H^\bullet(\mathcal{B}_{n'}^{\vee})$ decomposes as a $W$-representation
\begin{equation}\label{e:Y-coh2}
H^\bullet(\mathcal{B}_{n'}^{\vee})=H^{\mathrm{top}}(\mathcal{B}^{\vee}_{n'})+\sum_{\substack{E \in \mathrm{Irr}(W)\\,~n' \in \partial \mathbb O^\vee(E,\mathbb C)}} m_{E} E,
\end{equation}
where $m_E \in \ZZ_{\geq 0}$ for every term in the sum. Combining (\ref{e:Y-coh1}) and (\ref{e:Y-coh2}), we get that
\[\mathsf{AZ}(X)|_W = H^{\mathrm{top}}(\mathcal{B}^{\vee}_n)^{\rho} + \sum_{\substack{E \in \mathrm{Irr}(W)\\~n \in \partial \mathbb O^\vee(E,\mathbb C)}} m_{E}'E,
\]
for some integers $m_E'$. Since $s\in T^\vee_\mathbb R$, it follows that $s_0\in T^\vee_\mathbb R$ as well. Because $Z_{G^\vee}(s_0)$ is a Levi subgroup, the inclusion $Z_{G^\vee}(s_0,n)\to Z_{G^\vee}(n)$ induces an inclusion $A(s_0,n)\hookrightarrow A(n)$. Finally $A(s_0,n)=A(s,n)$, see for example \cite[Lemma 4.3.1]{Re-isogeny}. Hence we obtain
\[\mathsf{AZ}(X)|_W =
\sum_{\psi \in \mathrm{Irr}(A(n))_0} [\psi:\rho]_{A(s,n)} E(n,\psi) + \sum_{\substack{E \in \mathrm{Irr}(W)\\~n \in \partial \mathbb O^\vee(E,\mathbb C)}} m_{E} E
\]
where $E(n,\psi)$ is the irreducible Springer representation in $H^{\mathrm{top}}(\mathcal{B}_n^{\vee})^{\psi}$ and $[\psi:\rho]_{A(s,n)}$ denotes the multiplicity of $\rho$ in $\psi|_{A(s,n)}$. Since there exists at least one $\psi$ such that $[\psi:\rho]_{A(s,n)}\neq 0$, this implies the result by Lemma \ref{l:WF-factors}.
\end{proof}

Now let $\OO^{\vee} \subset \mathfrak{g}^{\vee}$ be a nilpotent orbit. Choose an $\mathfrak{sl}(2)$-triple $(e^{\vee},f^{\vee},h^{\vee})$ with $e^{\vee} \in \OO^{\vee}$. The semisimple operator $\ad(h^{\vee})$ induces a Lie algebra grading
$$\fg^{\vee} = \bigoplus_{n \in \ZZ}\fg^{\vee}[n], \qquad \fg^{\vee}[n] := \{x \in \fg^{\vee} \mid [h^{\vee},x] = nx\}.$$
Write $L^{\vee}$ for the connected (Levi) subgroup corresponding to the centralizer $\fg^{\vee}_0$ of $h^{\vee}$. If we set $s := q^{\frac{1}{2}h^{\vee}}$, then
$$L^{\vee} = G^{\vee}(s), \qquad \fg^{\vee}[2] = \fg^{\vee}_q$$
where $G^{\vee}(s)$ and $\fg^{\vee}_q$ are as defined in Section \ref{s:unip-cusp}. Note that $L^{\vee}$ acts by conjugation on each $\fg^{\vee}[n]$, and in particular on $\fg^{\vee}[2]$. We will need the following well-known facts, see \cite[Section 4]{Kostant1959} or \cite[Prop 4.2]{Lusztigperverse}.

\begin{lemma}\label{lem:orbitclosure}
The following are true:
\begin{itemize}
\item[(i)] $L^{\vee}e^{\vee}$ is an open subset of $\fg^{\vee}[2]$ (and hence the unique open $L^{\vee}$-orbit therein).
\item[(ii)] $L^{\vee}e = \OO^{\vee} \cap \fg^{\vee}[2]$.
\item[(iii)] $G^{\vee}\fg^{\vee}[2] \subseteq \overline{\OO^{\vee}}$.
\end{itemize}
\end{lemma}

\begin{cor}\label{cor:wfbound}
Suppose $X \in \Pi^{\mathsf{Lus}}_{q^{\frac{1}{2}h^{\vee}}}(\mathbf{G}(\sfk))$ is Iwahori-spherical. Then
\begin{align}\label{eq:WFsetbound}
\begin{split}
d_A(\OO^{\vee}, 1) &\leq_A \hphantom{ } ^K\WF(X)\\
d(\OO^{\vee}) &\leq \hphantom{ }^{\bar{\sfk}}\WF(X)
\end{split}
\end{align}
\end{cor}

\begin{proof}
The second inequality follows from the first by applying $\mathrm{pr}_1$ to both sides. So it suffices to show that $d_A(\OO^{\vee}, 1) \leq_A \hphantom{ } ^K\WF(X)$.

Let $X\in \Pi^{\mathsf{Lus}}_{q^{\frac{1}{2}h^{\vee}}}(\mathbf{G}(\sfk))$ be Iwahori-spherical. Since $\AZ$ is an involution on the Iwahori-spherical representations in $\Pi^{\mathsf{Lus}}_{q^{\frac{1}{2}h^{\vee}}}(\mathbf{G}(\sfk))$, we have
$$X=\mathsf{AZ}(X') \text{ for some Iwahori-spherical }  X'=X(q^{\frac12h^\vee},n',\rho')\in \Pi^{\mathsf{Lus}}_{q^{\frac{1}{2}h^{\vee}}}(\mathbf{G}(\sfk))$$
By definition, $n' \in \fg^{\vee}_q = \fg^{\vee}[2]$. Now (iii) of Lemma \ref{lem:orbitclosure} implies that $\OO^\vee_{X'}:=G^\vee n'\le \OO^\vee$. In particular $(\OO^\vee_{X'},1)\le_A(\OO^\vee,1)$ and so $d_A(\OO^\vee,1)\le_A d_A(\OO^\vee_{X'},1)$ since $d_A$ is order-reversing. But by Theorem \ref{thm:realwf}, $d_A(\OO^\vee_{X'},1) = \hphantom{ }^K\WF(X)$. This completes the proof.
\end{proof}

\section{Proof of faithfulness}
\label{sec:faithful}
Fix the notation of Sections \ref{subsec:nilpotent}-\ref{subsec:Wreps}, i.e. $G$, $G^{\vee}$,$\cN$, $\cN^{\vee}$, $\mathcal{F}_{\tilde{\Delta}}$, $\bar{s}$, and so on. Recall that to each irreducible representation $E$ of $W$, we associate nilpotent orbits $\OO(E,\CC) \subset \cN$ and $\OO^{\vee}(E,\CC) \subset \cN^{\vee}$. In this section all of our objects will be defined over $\CC$---so we will simplify the notation by writing $\OO(E) := \OO(E,\CC)$, $\OO^{\vee}(E) :=\OO^{\vee}(E,\CC)$. Similarly, for a family $\phi$ in $\mathrm{Irr}(W)$, we will write $\OO(\phi)$, $\OO^{\vee}(\phi)$ for $\OO(\phi,\CC)$, $\OO^{\vee}(\phi,\CC)$.

Recall the notion of a \emph{faithful} nilpotent orbit, see Definition \ref{def:faithful}. In this section we will prove that all nilpotent orbits are faithful (this is Theorem \ref{thm:faithful}). 

First, note that conditions (i) and (ii) of Definition \ref{def:faithful} depend only on the Lie algebra $\fg = \mathrm{Lie}(G)$. Moreover, if $\OO^{\vee}_1, \OO^{\vee}_2$ are faithful (for $(J_1,\phi_1)$, $(J_2,\phi_2)$ respectively), then $\OO^{\vee}_1 \times \OO^{\vee}_2$ is faithful (for $(J_1 \times J_2, \phi_1 \otimes \phi_2)$). Thus, it suffices to prove Theorem \ref{thm:faithful} for $\mathfrak{g}$ a simple Lie algebra. In many cases, we can appeal to the following proposition.

\begin{lemma}\label{lem:triviallocalsystem}
Suppose $\overline{\mathbb L}(J, \OO(\phi)) = d_A(\OO^{\vee},1)$, $E = E(\OO^{\vee},1)$ (as representations of $W$), and $F = \mathrm{sp}(\phi \otimes \mathrm{sgn})$. Then $j^W_{W_J} F = E$. 
\end{lemma}
\begin{proof}
    Let $(\OO,C) := \mathbb L(J,\OO_{\phi})$.
    Then, by definition of $d_S$ we have that $E(d_S(\OO,C),1) = j_{W_J}^W F$.
    But $d_S(d_A(\OO^\vee,1)) = \OO^\vee$ as required.
\end{proof}

\begin{prop}\label{prop:easycase}
Suppose there is a unique irreducible representation $E$ of $W$ such that $\OO^{\vee}(E) = \OO^{\vee}$. Then $\OO^{\vee}$ is faithful.
\end{prop}

\begin{proof}
By the uniqueness condition, we must have $E = E(\OO^{\vee},1)$. Recall that $\overline{\mathbb L}: \mathcal{K}_{\widetilde{\Delta}} \to \cN_{o,\bar c}$ is surjective. Choose any $(J,\phi)$ such that $\overline{\mathbb L}(J,\OO_{\phi}) = d_A(\OO^{\vee},1)$ and let $F = \mathrm{sp}(\phi \otimes \mathrm{sgn})$. By Lemma \ref{lem:triviallocalsystem}, we have $j^W_{W_J}F = E$, and hence $\Hom_{W_J}(F, E|_{W_J}) \neq 0$ by Frobenius reciprocity.
\end{proof}

Proposition \ref{prop:easycase} implies Theorem \ref{thm:faithful} for $\fg = \mathfrak{sl}(n)$ and for $\fg = \mathfrak{so}(2n)$ and $\OO^{\vee}$ a very even orbit (see Section \ref{subsec:orbits}). The simple exceptional Lie algebras are handled in Section \ref{subsec:exceptional} (where we make repeated use of Proposition \ref{prop:easycase} to simplify the arguments). It remains to prove Theorem \ref{thm:faithful} for $\fg$ a simple Lie algebra of type B/C/D (and $\OO^{\vee}$ not very even in the type D case). The proof will require a fairly lengthy detour into the combinatorics of partitions and symbols.

\subsection{Preliminaries on partitions}\label{subsec:partitions}

Let $\mathcal{P}(n)$ denote the set of partitions of $n$. For $\lambda \in \mathcal{P}(n)$, we will typically write $\lambda = (\lambda_1 \geq ... \geq \lambda_k)$, and we assume $\lambda_k \neq 0$ (unless otherwise stated). We will write $\#\lambda$ for the number of (nonzero) parts of $\lambda$ and $|\lambda|$ for their sum. If $x \in \ZZ_{\geq 0}$, we will write $m_{\lambda}(x)$ for the multiplicity of $x$ in $\lambda$ and $\mathrm{ht}_{\lambda}(x)$ for its `height,' i.e. $\mathrm{ht}_{\lambda}(x) = \sum_{y \geq x} m_{\lambda}(y)$. A partition $\lambda$ is \emph{very even} if all parts are even, occuring with even multiplicity. 

There is a partial order on partitions, defined by the formula
$$\lambda \leq \mu \iff \sum_{i \leq j} \lambda_i \leq \sum_{i \leq j} \mu_i.$$
Given an arbitrary partition $\lambda = (\lambda_1,...,\lambda_k)$, we define two new partitions (of $|\lambda|+1$ and $|\lambda|-1$)
$$\lambda^+ := (\lambda_1+1,\lambda_2,...,\lambda_k), \qquad \lambda^- := (\lambda_1,\lambda_2,...,\lambda_k-1)$$
and we write $\lambda^t$ for the transpose of $\lambda$. Given $\lambda \in \mathcal{P}(n)$ and $\mu \in \mathcal{P}(m)$, we define $\lambda \cup \mu \in \mathcal{P}(m+n)$ by `adding multiplicities'
$$m_{\lambda \cup \mu}(x) = m_{\lambda}(x) + m_{\mu}(x), \qquad \forall x \in \ZZ_{\geq 0}.$$
We write $\alpha = \lambda \cup_{\geq} \mu$ if $\alpha = \lambda \cup \mu$ and the smallest part of $\lambda$ is at least as large as the largest part of $\mu$. We write $\mu \subseteq \lambda$ if $m_{\mu}(x) \leq m_{\lambda}(x)$ for every $x \in \ZZ_{\geq 0}$. In this case, there is a unique subpartition partition $\lambda \setminus \mu \subseteq \lambda$ such that $\lambda = \mu \cup (\lambda \setminus \mu)$. 

A \emph{decorated} partition of $n$ is a partition $\lambda \in \mathcal{P}(n)$ together with a decoration $\kappa \in \{0,1\}$. We write $(\lambda,\kappa)$ as $\lambda^{\kappa}$. If $\lambda$ is \emph{not} very even, we declare $\lambda^0 = \lambda^1$. Otherwise we regard $\lambda^0$ and $\lambda^1$ as distinct. Write $\mathcal{P}^d(n)$ for the set of decorated partitions of $n$, with the equivalence relation just defined.

An \emph{(ordered) bipartition} of $n$ is a pair of partitions $(\lambda,\mu)$ such that $|\lambda|+|\mu|=n$. An \emph{unordered bipartition} of $n$ is an unordered pair $\{\lambda, \mu\}$ with the same property. Write $\overline{\mathcal{P}}(n)$ (resp. $\overline{\mathcal{P}}^u(n)$) for the set of ordered (resp. unordered) bipartitions of $n$. 

A \emph{decorated unordered bipartition} of $n$ is an unordered bipartition $\{\lambda,\nu\}$ of $n$ together with a decoration $\kappa\in\{0,1\}$.
We write $(\{\lambda, \nu\},\kappa)$ as $\{\lambda,\mu\}^\kappa$.
If $\lambda\ne \mu$ we declare $\{\lambda,\mu\}^0 = \{\lambda,\mu\}^1$. Otherwise we regard $\{\lambda,\mu\}^0$ and  $\{\lambda,\mu\}^1$ as distinct.
Write $\overline{\mathcal P}^{d,u}(n)$ for the set of decorated unordered bipartition of $n$, with the equivalence relation just defined.

\subsection{Preliminaries on classical Lie algebras}
\label{sec:classicalliealgebras}
Let $X \in \{B,C,D\}$ and let $\fg_X(n)$ denote the complex simple Lie algebra of type $X$ and rank $n$. Denote the corresponding Weyl group by $W_X(n)$. If $X \in \{B,C\}$, there is a bijection between the set of irreducible representations of $W_X(n)$ and the set $\overline{\mathcal{P}}(n)$ of ordered bipartitions. 
If $X=D$, there is a bijection between the irreducible representations of $W_X(n)$ and the set $\overline{\mathcal{P}}^{d,u}(n)$ of decorated unordered bipartitions.

We label the extended Dynkin diagrams using the standard Bourbaki conventions:

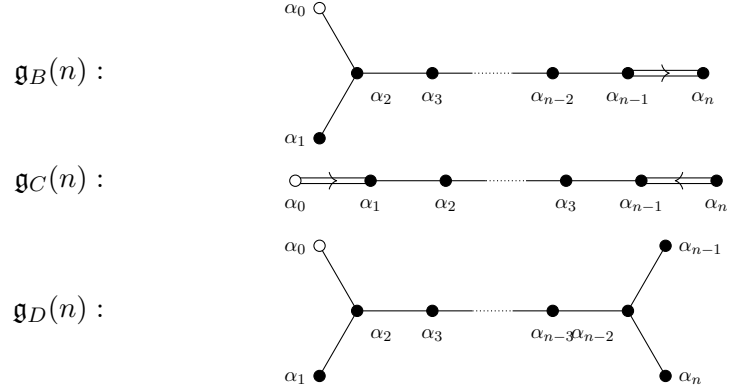
\begin{align*}
    \fg_B(n): \qquad \qquad \qquad &  \dynkin[extended,labels={\alpha_0,\alpha_1,\alpha_2,\alpha_3,\alpha_{n-2},\alpha_{n-1},\alpha_n},edge length = 1 cm, root radius = .075cm] B{} \\
    \fg_C(n): \qquad \qquad \qquad& \dynkin[extended,labels={\alpha_0,\alpha_1,\alpha_2,\alpha_3,\alpha_{n-1},\alpha_n},edge length =  1 cm, root radius = .075cm] C{} \\
    \fg_D(n): \qquad \qquad \qquad& \dynkin[extended,labels={\alpha_0,\alpha_1,\alpha_2,\alpha_3,\alpha_{n-3},\alpha_{n-2},\alpha_{n-1},\alpha_n},edge length = 1 cm, root radius = .075cm] D{}
\end{align*}
\vspace{3mm}
For $0 \leq k \leq n$, define the subset
$$J_X(k,n) := \{\alpha_0,...,\alpha_{k-1},\alpha_{k+1},...,\alpha_n\} \subsetneq \widetilde{\Delta},$$
and write $\mathfrak{l}_X(k,n) \subset \fg$ for the corresponding (standard) pseudolevi subalgebra. Then
\begin{align}\label{eq:pseudolevis}
\begin{split}
    \fl_B(k,n) &\simeq \begin{cases}
        \fg_D(k) \times \fg_B(n-k), & k \notin \{1\}\\
        \fg_B(n), & k \in \{1\}
    \end{cases} \\
    \fl_C(k,n) &\simeq \fg_C(k) \times \fg_C(n-k)\\
    \fl_D(k,n) &\simeq \begin{cases}
        \fg_D(k) \times \fg_D(n-k), & k \notin \{1,n-1\} \\
        \fg_D(n), & k \in \{1,n-1\}
    \end{cases}
\end{split}
\end{align}

\subsection{Preliminaries on nilpotent orbits in classical types}\label{subsec:orbits}

Define the sets
\begin{align*}
    \mathcal{P}_B(n) &= \{\lambda \in \mathcal{P}(2n+1) \mid x \text{ even} \implies m_{\lambda}(x) \text{ even}\}\\
    \mathcal{P}_C(n) &= \{\lambda \in \mathcal{P}(2n) \mid x \text{ odd} \implies m_{\lambda}(x) \text{ odd (or 0)}\}\\
    \mathcal{P}_D(n) &= \{\lambda^{\kappa} \in \mathcal{P}^d(2n) \mid x \text{ even} \implies m_{\lambda}(x) \text{ even}\}
\end{align*}
For $\fg = \fg_X(n)$, $X \in \{B,C,D\}$, there is a well-known bijection
\begin{equation}\label{eq:orbitspartitions}\cN_o^{\fg} \xrightarrow{\sim} \mathcal{P}_X(n)\end{equation}
see \cite[Section 5.1]{CM}. The elements of $\mathcal{P}_X(n)$ are called $X$-partitions of $n$ (they are decorated in the case when $X=D$). If $X=D$ and $\OO \in \cN_0^{\fg}$ corresponds to $\lambda^{\kappa}$ , we say that $\OO$ is very even if $\lambda$ is very even, see Section \ref{subsec:partitions}.
 
For $X \in \{B,C,D\}$, the $X$-collapse of $\lambda$ is the unique largest $X$-partition $\lambda_X$ such that $\lambda_X \leq \lambda$ (in type $B$ (resp. $C$, $D$), this is sensible provided $|\lambda|=2n+1$ (resp. $|\lambda|=2n$)). A formula for $\lambda_X$ is provided in \cite[Lem 6.3.8]{CM}. We will recall the details for the case when $X=C$ (the other cases are analogous). First, write $\lambda$ as a union of even and odd parts
$$\lambda =  (\lambda^e_1,...,\lambda^e_p) \cup (\lambda^o_1,...,\lambda^o_{2q}), \qquad \lambda_1^e \geq ... \geq \lambda_p^e, \ \lambda_1^o \geq ... \geq \lambda_{2q}^o.$$
For $1 \leq i \leq q$, define 
$$(\lambda_{2i-1}^o,\lambda_{2i}^o)' =  \begin{cases} 
(\lambda_{2i-1}^o,\lambda_{2i}^o) &\mbox{if } \lambda^o_{2i-1}=\lambda^o_{2i}\\
(\lambda_{2i-1}^o-1,\lambda_{2i-1}^o+1) &\mbox{otherwise} \end{cases}$$
Then 
\begin{equation}\label{eq:Ccollapse}\lambda_C = (\lambda_1^e,...,\lambda_p^e) \cup \bigcup_{i=1}^q (\lambda_{2i-1}^o,\lambda_{2i}^o)'.\end{equation}
If $\lambda \in \mathcal{P}_X(n)$, define the \emph{dual} partition $d(\lambda)$ of $\lambda$
$$d(\lambda) = \begin{cases} 
((\lambda^t)^-)_C \in \mathcal{P}_C(n)& \mbox{if } X =B\\
((\lambda^t)^+)_B \in \mathcal{P}_B(n) &\mbox{if } X=C\\
(\lambda^t)_D \in \mathcal{P}_D(n) &\mbox{if } X=D.
\end{cases}$$
(If $X=D$ and $\lambda$ is very even, there is a decoration $\kappa$ to keep track of--the behavior of $\kappa$ under $d$ is explained in \cite[Cor 6.3.5]{CM}, but we will not need this result). For $X \in \{B,C,D\}$, we write $\mathcal{P}^*_X(n) \subset \mathcal{P}_X(n)$ for the set of (decorated) partitions corresponding to \emph{special nilpotent orbits}---briefly, for the set of \emph{special X-partitions}. We will use the following characterization found in \cite[Section 6.3]{CM}: a partition $\lambda \in \mathcal{P}_X(n)$ (resp. decorated partition $\lambda^{\kappa} \in \mathcal{P}_D(n)$) is special if and only if it contains:
\begin{align*}
    X=B: &\text{ an even number of odd parts between every pair of consecutive even parts}\\
    &\text{ and an odd number of odd parts greater than the largest even part.}\\
    X=C: &\text{ an even number of even parts between every pair of consecutive odd parts}\\
    &\text{ and an even number of even parts larger than the greatest odd part.}\\
    X=D: &\text{ an even number of odd parts between every pair of consecutive even parts}\\
    &\text{ and an even number of odd parts greater than the largest even part.}
\end{align*}
Note that if $X=D$ and $\lambda \in \mathcal{P}(2n)$ is very even, then $\lambda^0, \lambda^1 \in \mathcal{P}_X^*(n)$.

Define the sets
\begin{align*} 
\overline{\mathcal P}_B(p,n) &= \begin{cases}
    \mathcal P_D(p) \times \mathcal P_B(n-p) & \mbox{if } p\ne 1 \\
    \mathcal \mathcal P_B(n) & \mbox{if } p = 1
\end{cases} \\
\overline{\mathcal P}_C(p,n) &= \mathcal{P}_C(p) \times \mathcal{P}_C(n-p) \\
\overline{\mathcal P}_D(p,n) &= \begin{cases}
    \mathcal P_D(p) \times \mathcal P_D(n-p) & \mbox{if } p\ne 1,n-1 \\
    \mathcal \mathcal P_D(n) & \mbox{if } p = 1,n-1
\end{cases} 
\end{align*}
and define
$$\overline{\mathcal{P}}_X(n) = \coprod_{k=0}^n\overline{\mathcal P}_X(k,n).$$
Analogously define $\overline{\mathcal P}^*_X(k,n)$ and $\overline{P}^*_X(n)$ by replacing all instances of $\mathcal P$ with $\mathcal P^*$ in the definitions above.
Following Achar, we will often write $^{\langle\mu\rangle}(\mu \cup \nu)$ for the bipartition $(\mu,\nu) \in \overline{\mathcal{P}}_X(n)$ (except when $\mu$ or $\nu$ is very even and this notation is ambiguous).
Note that for $\lambda\in \mathcal P_X(n)$ we have $(\emptyset,\lambda)\in \overline{\mathcal P}_X(p,n)$ for 
\begin{align*}
    X = B:&\quad p\in \{0,1\} \\
    X = C:&\quad p\in \{0\} \\
    X = D:&\quad p\in \{0,1\}
\end{align*}
and $(\lambda,\emptyset)\in \overline{\mathcal P}_D(p,n)$ for $X=D$, $p\in\{n-1,n\}$.
Thus there is ambiguity when we write $(\emptyset,\lambda)\in\overline{\mathcal P}_X(n)$.
We will always interpret this statement as $(\emptyset,\lambda)\in \overline{\mathcal P}_X(0,n)$.
Similarly, we interpret $(\lambda,\emptyset)\in\overline{\mathcal P}_D(n)$ to mean $(\lambda,\emptyset)\in\overline{\mathcal P}_D(n,n)$.

For $J = J_X(p,n)$, we have natural bijections
\begin{align}\label{eq:pseudolevisorbits}
    \cN_o^{L_J} &\xrightarrow{\sim} \overline{\mathcal{P}}_X(p,n) 
\end{align}
defined via (\ref{eq:orbitspartitions}) and (\ref{eq:pseudolevis}). 

Let 
\begin{equation}
    \mathcal K_{\tilde\Delta}^{max} := \{(J,\OO)\in\mathcal K_{\tilde\Delta} \mid |\tilde\Delta|-|J|=1\}.
\end{equation}
There is a natural bijection 
\begin{equation}\label{eq:KPbijection}\mathcal K_{\tilde\Delta}^{max}\xrightarrow{\sim} \overline{\mathcal P}_X(n)\end{equation}
defined via the bijections (\ref{eq:pseudolevisorbits}).

Consider the composition
$$s: \overline{\mathcal{P}}_X(n) \xrightarrow{\sim} \mathcal{K}_{\tilde{\Delta}}^{max} \hookrightarrow \mathcal{K}_{\tilde{\Delta}} \overset{\mathbb{L}}{\twoheadrightarrow} \cN_{o,c},$$
and the further composition
$$\bar s: \overline{\mathcal{P}}_X(n) \overset{s}{\twoheadrightarrow} \cN_{o,\bar c} \overset{\mathfrak{Q}}{\twoheadrightarrow} \cN_{o,\bar c}.$$
The fibers of $\bar s$ are described in \cite[Section 3.4]{Acharduality}. His description requires some additional terminology. If $\lambda \in \mathcal{P}_X(n)$ and $x \in \ZZ_{\geq 0}$, we say that $x$ is \emph{markable in} $\lambda$ (or simply \emph{markable}, if $\lambda$ is understood) if $m_{\lambda}(x) \geq 1$ and
\begin{align*}
    X=B: & \text{ $x$ is odd and $\mathrm{ht}_{\lambda}(x)$ is odd}.\\
    X=C: & \text{ $x$ is even and $\mathrm{ht}_{\lambda}(x)$ is even}.\\
    X=D: & \text{ $x$ is odd and $\mathrm{ht}_{\lambda}(x)$ is even}.
\end{align*}
Write $x_k > x_{k-1} > ... > x_1$ for the markable parts of $\lambda$. Then for $\mu$ a partition, the \emph{reduction} of $\mu$ is a subpartition $r_{\lambda}(\mu) \subseteq \lambda$ defined by
$$m_{r_{\lambda}(\mu)}(x_i) = \begin{cases} 
1   &\mbox{if } \mathrm{ht}_{\mu}(x_i) - \mathrm{ht}_{\mu}(x_{i+1}) \text{ is odd}.\\
0 &\mbox{if } \mathrm{ht}_{\mu}(x_i) - \mathrm{ht}_{\mu}(x_{i+1}) \text{ is even}.
\end{cases}, \qquad m_{r_{\lambda}(\mu)}(y) = 0 \text{ if } y \text{ is not markable in } \lambda,$$
where we set $\mathrm{ht}_{\mu}(x_{k+1}) = 0$. 

\begin{prop}[Section 3.4, \cite{Acharduality}]\label{prop:Acharequivalence}
Let $\lambda \in \mathcal{P}_X(n)$ (if $X=D$, we assume $\lambda$ is not very even and so ignore the decoration). Choose $^{\langle\mu\rangle}\lambda, ^{\langle\mu'\rangle}\lambda \in \overline{\mathcal{P}}_X(n)$. Then $\bar{s}(^{\langle\mu\rangle}\lambda) = \bar{s}(^{\langle\mu'\rangle}\lambda)$ if and only if $r_{\lambda}(\mu) = r_{\lambda}(\mu')$.
\end{prop}

\subsection{Preliminaries on s-symbols}

Write $a\ll b$ to mean $a+2\le b$.
For a sequence $a = (a_1,a_2,\dots,a_k)$ define $\#a:=k$.
For an integer $x$ we write $x+a$ for the sequence $(x+a_1,\dots,x+a_k)$.
Write $a(i)$ for the $i$th entry from the right of $a$.
Define $z^{0,k}=(0,1,\dots,k-1)$.
An s-symbol $\Lambda:=(a;b)$ is a pair of sequences $a=(a_1\ll a_2 \ll \cdots \ll a_k)$, $b=(b_1 \ll b_2 \ll \dots \ll b_l)$ with $a_1,b_1\ge 0$.
Define $\#\Lambda = \#a+\#b$.
The defect of $\Lambda$ is defined to be $\#a-\#b$.
We say $\Lambda$ is an s-symbol of type X if 
\begin{align*}
    X=B: &\text{ $\Lambda$ has defect 1}\\
    X=C: &\text{ $\Lambda$ has defect 1 and $b_1>0$}\\
    X=D: &\text{ $\Lambda$ has defect 0.}
\end{align*}
We will introduce two different equivalence relations on s-symbols, denoted $\sim$ and $\approx$. For $\Lambda=((a_1,\dots,a_k);(b_1,\dots,b_l))$ and $\Lambda'=((a_1',\dots,a_{k'}');(b_1',\dots,b_{l'}'))$ we write $\Lambda\sim\Lambda'$ if $$\{a_1,\dots,a_k,b_1,\dots,b_{l}\} = \{a_1',\dots,a_{k'}',b_1',\dots,b_{l'}'\}$$ as multisets. 
On the other hand, $\approx$ is defined to be the equivalence relation generated by
\begin{align*}
    X=B,D&:((a_1,\dots,a_s);(b_1,\dots,b_t))\approx ((0,a_1+2,\dots,a_s+2);(0,b_1+2,\dots,b_t+2))\\
    X=C&:((a_1,\dots,a_s);(b_1,\dots,b_t))\approx ((0,a_1+2,\dots,a_s+2);(1,b_1+2,\dots,b_t+2)).
\end{align*}
For a set of s-symbols $S$ of type $X$ and an $\Lambda\in S$ write $[\Lambda;S]$ (resp. $\phi(\Lambda;S)$) for the set of $\Lambda'\in S$ such that $\Lambda\approx\Lambda'$ (resp. $\Lambda\sim\Lambda'$).
For a sequence $a=(a_1\ll a_2 \ll \cdots \ll a_k)$ define
\begin{align*}
    \rho_{s,0}(a) &= \sum_ia_i-2(i-1) \\
    \rho_{s,1}(a) &= \sum_ia_i-2(i-1)-1 \text{ provided $a_1> 0$}.
\end{align*}
Write $\Lambda_X(n;k)$ for the set of s-symbols $(a;b)$ of type $X$ with $\#b=k$ and 
\begin{align*}
    X=B: & \ \rho_{s,0}(a)+\rho_{s,0}(b) = n\\
    X=C: & \ \rho_{s,0}(a)+\rho_{s,1}(b) = n\\
    X=D: & \ \rho_{s,0}(a)+\rho_{s,0}(b) = n.
\end{align*}
For an s-symbol $\Lambda$, define $\bar \Lambda$ be the sequence
$$\bar \Lambda := \begin{cases}(a_1,b_1,a_2,b_2,\dots) & \mbox{ if $X \in \{B,C\}$}\\
 (b_1,a_1,b_2,a_2,\dots) & \mbox{ if $X=D$}.
    \end{cases} $$
We say $\Lambda$ is monotonic if $\bar \Lambda$ is a non-decreasing sequence.
Note every $\sim$ equivalence class in $\Lambda_X(n;k)$ contains a monotonic s-symbol.
Define $$\Lambda_X(n) = \coprod_{k\ge0} \Lambda_X(n;k)/\approx.$$
For $\Lambda\in \coprod_{k\ge0} \Lambda_X(n;k)$ we will simply write $[\Lambda]$ for $\left[\Lambda;\coprod_{k\ge0} \Lambda_X(n;k)\right]$.
Define $[\Lambda_1]\sim[\Lambda_2]$ if there exists $\Lambda_1'\in[\Lambda_1],\Lambda_2'\in[\Lambda_2]$ such that $\Lambda_1'\sim\Lambda_2'$.
Write $\phi([\Lambda])$ for the set of $[\Lambda']\in \Lambda_X(n)$ such that $[\Lambda]\sim[\Lambda']$.
For $k\ge 0$ and $\Lambda\in\Lambda_X(n;k)$ the map
\begin{equation}
    \label{eq:familybijection}
    \phi(\Lambda;\Lambda_X(n;k))\xrightarrow{\sim}\phi([\Lambda]), \qquad \Lambda'\mapsto [\Lambda']
\end{equation}
is a bijection.

For $X\in\{B,C\}$, recall from section \ref{sec:classicalliealgebras} that $\mathrm{Irr}(W_X(n))$ is parameterised by $\overline{\mathcal P}(n)$.
Define a map
\begin{equation}\label{eq:LambdaE}\Lambda:\mathrm{Irr}(W_X(n))\to \coprod_{j\ge0}\Lambda_X(n;k),\qquad E\mapsto \Lambda(E)\end{equation}
as follows. 
Suppose $E\in\mathrm{Irr}(W_X(n))$ corresponds to $(\lambda,\mu)\in\overline{\mathcal P}(n)$ with $\lambda=(\lambda_1\le\lambda_2\le\cdots\\leq \lambda_k),\mu=(\mu_1\le\mu_2\le\cdots\mu_l)$.
If $k > l$ let $\lambda'=\lambda$, $\mu'=(0^{k-l-1})\cup_\le \mu$ and define 
\begin{equation}
    \Lambda(E) := \begin{cases}
        (\lambda'+2z^{0,k};\mu'+2z^{0,k-1}) & \mbox{ if $X=B$} \\
        (\lambda'+2z^{0,k};\mu'+2z^{0,k-1}+1) & \mbox{ if $X=C$}.
    \end{cases}
\end{equation}
If $k\le l$ let $\lambda=(0^{l-k+1})\cup_\le\lambda,\mu'=\mu$ and define $\Lambda(E)$ as above but with $k$ replaced by $l$. We get a bijection
$$\mathrm{Irr}(W_X(n))\xrightarrow{\sim} \Lambda_X(n), \qquad E \mapsto [\Lambda(E)].$$
This bijection has the following properties \cite[Section 13.3]{Carter1993}
\begin{enumerate}
    \item $\OO(E_1)=\OO(E_2)$ if and only if $[\Lambda(E_1)]\sim [\Lambda(E_2)]$,
    \item $\Lambda(E)$ is monotonic if and only if for all $\Lambda\in [\Lambda(E)]$, $\Lambda$ is monotonic if and only if $E = E(\OO,1)$ for some $\OO\in\cN_o$.
\end{enumerate}

The situation in type $D$ is a bit more delicate.
Firstly we must consider unordered s-symbols. 
These are s-symbols, except they consist of unordered pairs $\Lambda:=\{a;b\}$ instead of ordered pairs.
We say $\Lambda$ has defect $|\#a-\#b|$.
For an ordered pair $\Omega=(a;b)$ write $\{\Omega\}$ for the unordered pair $\{a;b\}$.
We say $\Lambda$ is of type $D$ if it has defect $0$.
We will also require a decoration for $\Lambda=\{a;b\}$.
We call $\Lambda^\kappa$ a decorated s-symbol where $\kappa\in \{0,1\}$.
If $a\ne b$ then we declare $\Lambda^0=\Lambda^1$, otherwise they are different.
We define $\sim$ for unordered s-symbols the same as for s-symbols.
We define $\sim$ for decorated unordered s-symbols by 
\begin{equation}
    \{a;b\}^{\kappa} \sim \{a';b'\}^{\kappa'} \text{ if } \begin{cases}
        \{a;b\} = \{a';b'\}, \kappa=\kappa' & \mbox{if } a=b \\
        \{a;b\} \sim \{a';b'\} & \mbox{if } a\ne b.
    \end{cases}
\end{equation}
We define $\approx$ for unordered s-symbols to be the equivalence relation generated by 
\begin{equation}
    \{(a_1,\dots,a_s);(b_1,\dots,b_t)\}\approx \{(0,a_1+2,\dots,a_s+2);(0,b_1+2,\dots,b_t+2)\}.
\end{equation}
We define $\approx$ for decorated unordered s-symbols by
\begin{equation}
    \{a;b\}^{\kappa} \approx\{a';b'\}^{\kappa'} \text{ if } \begin{cases}
        \{a;b\} \approx \{a';b'\}, \kappa=\kappa' & \mbox{if } a=b \\
        \{a;b\} \approx \{a';b'\} & \mbox{if } a\ne b.
    \end{cases}
\end{equation}
For a set $S$ of unordered s-symbols (resp. decorated unordered s-symbols) and $\Lambda\in S$ (resp. $\Lambda^\kappa\in S$)we define $[\Lambda;S]$ and $\phi(\Lambda;S)$ analagously to the ordered case.
Write $\tilde\Lambda_D(n;k)$ (resp. $\tilde\Lambda_D^\bullet(n;k)$) for the set of unordered s-symbols (resp. decorated unordered s-symbols) of type $D$ with $\#b=k$ and $\rho_{s,0}(a)+\rho_{s,0}(b) = n$.
Define 
\begin{equation}
    \tilde\Lambda_D(n) = \coprod_{k\ge0} \tilde\Lambda_D(n;k)/\approx, \qquad \tilde\Lambda_D^\bullet(n) = \coprod_{k\ge0} \tilde\Lambda_D^\bullet(n;k)/\approx.
\end{equation}
For $\Lambda \in \coprod_{k\ge0} \tilde\Lambda_D(n;k)$ (resp. $\Lambda^\kappa\in\coprod_{k\ge0} \tilde\Lambda_D^\bullet(n;k)$) we will simply write $[\Lambda]$ (resp. $[\Lambda^\kappa]$) for $[\Lambda;\coprod_{k\ge0} \tilde\Lambda_D(n;k)]$ (resp. $[\Lambda^\kappa;\coprod_{k\ge0} \tilde\Lambda^\bullet_D(n;k)]$).
Write $\phi([\Lambda])$ (resp. $\phi([\Lambda^\kappa])$) for the set of $[\Lambda']\in \tilde\Lambda_D(n)$ (resp. $[\Lambda'^{\kappa'}]\in\tilde\Lambda_D^\bullet(n)$) such that $[\Lambda]\sim[\Lambda']$ (resp. $[\Lambda^\kappa]\sim[\Lambda'^{\kappa'}]$).
For $k\ge0$ and $\Lambda\in\tilde\Lambda_D(n;k)$ (resp. $\Lambda^\kappa\in\tilde\Lambda_D^\bullet(n;k)$) the map
\begin{align}
    \phi(\Lambda;\tilde\Lambda_D(n;k)) &\to \phi([\Lambda]), \qquad \Lambda'\mapsto[\Lambda'] \\
    \text{resp. }\phi(\Lambda^\kappa;\tilde\Lambda_D^\bullet(n;k)) &\to \phi([\Lambda^\kappa]), \qquad \Lambda'^{\kappa'}\mapsto[\Lambda'^{\kappa'}]
\end{align}
is a bijection.

Recall from section \ref{sec:classicalliealgebras} that $\mathrm{Irr}(W_D(n))$ is parameterised by $\overline{\mathcal P}^{d,u}(n)$.
Define 
$$\tilde\Lambda^\bullet:\mathrm{Irr}(W_D(n))\to \coprod_{j\ge0}\tilde\Lambda_D^\bullet(n;k),\qquad E\mapsto \tilde\Lambda^\bullet(E)$$ 
as follows.
Suppose $E\in\mathrm{Irr}(W_D(n))$ corresponds to $\{\lambda,\mu\}^\kappa\in\overline{\mathcal P}^{d,u}(n)$.
Let $\lambda',\mu'$ be $\lambda,\mu$ padded with $0$'s in a similar manner to types $B$ and $C$ but to ensure they have the same number entries.
Define
\begin{equation}
    \tilde\Lambda^\bullet(E) = \{\lambda'+2z^{0,k};\mu'+2z^{0,k}\}^\kappa
\end{equation}
where $k=\#\lambda'=\#\mu'$.
Define the map
$$\tilde\Lambda:\mathrm{Irr}(W_D(n))\to \coprod_{j\ge0}\tilde\Lambda_D(n;k),\qquad E\mapsto \tilde\Lambda(E),$$
by setting $\tilde\Lambda(E)$ to be the underlying unordered symbol for $\tilde{\Lambda}^{\bullet}(E)$. 
We get a bijection
\begin{equation}
    \label{eq:ssymbD}
    \mathrm{Irr}(W_D(n)) \xrightarrow{\sim} \tilde\Lambda_D^\bullet(n), \qquad E \mapsto [\tilde{\Lambda}^{\bullet}(E)].    
\end{equation}

%


For certain identities we will require an ordering on the rows of $\Lambda$.
Define
\begin{equation}
    \underline\Lambda = \begin{cases}
        (a;b) & \mbox{if } \sum_ia_i\ge \sum_ib_i \\ 
        (b;a) & \mbox{if } \sum_ia_i< \sum_ib_i.
    \end{cases}
\end{equation}
and
\begin{equation}
    \underline{\overline\Lambda} = \begin{cases}
        (b_1,a_1,b_2,a_2,\dots) & \mbox{if } \sum_ia_i\ge \sum_ib_i \\ 
        (a_1,b_1,a_2,b_2,\dots) & \mbox{if } \sum_ia_i< \sum_ib_i.
    \end{cases}
\end{equation}
For $E\in\mathrm{Irr}(W_D(n))$ write $\Lambda(E)$ for $\underline{\tilde\Lambda(E)}$.
We say $\Lambda^\kappa=\{a;b\}^\kappa\in \tilde\Lambda_D^\bullet(n)$ is monotonic if $\underline{\overline\Lambda}$ is a non-decreasing sequence.

The bijection in Equation (\ref{eq:ssymbD}) has the following properties \cite[Section 13.3]{Carter1993}
\begin{enumerate}
    \item $\OO(E_1)=\OO(E_2)$ if and only if $[\tilde\Lambda^\bullet(E_1)]\sim [\tilde\Lambda^\bullet(E_2)]$,
    \item $\tilde\Lambda^\bullet(E)$ is monotonic if and only if for all $\Lambda^\kappa\in [\tilde\Lambda^\bullet(E)]$, $\Lambda^\kappa$ is monotonic if and only if $E = E(\OO,1)$ for some $\OO\in\cN_o$.
\end{enumerate}

For $\Lambda\in\Lambda_X(n;k)$ monotonic we have the following convenient description for $[\Lambda;\Lambda_X(n;k)]$.
Decompose $\bar\Lambda$ into subsequences $(I_1,I_2,\dots,I_n)$ by first extracting all pairs, and then taking intervals of consecutive integers (`intervals' for short).
We call this decomposition the \textit{refinement} of $\bar\Lambda$.
Now decompose $a,b$ as $a=(A_1,A_2,\dots,A_n),b=(B_1,B_2,\dots,B_n)$ where $A_i$ (resp. $B_i$) consists (in increasing order) of the elements of $a$ (resp. $b$) in $I_i$.
We call this the \textit{refinements} of $a$ and $b$ respectively.
\begin{example}
    Suppose
    \begin{equation}
        \Lambda=\begin{pmatrix}
            0 & & 2 & & 3 & & 7 & & 10 & & 13 \\
            & 1 & & 3 & & 6 & & 8 & & 11 &
        \end{pmatrix}
    \end{equation}
    then $\bar \Lambda = (0,1,2,3,3,6,7,8,10,11,13)$ and $I_1=(0,1,2),I_2=(3,3),I_3=(6,7,8),I_4=(10,11),I_5=(13)$.
    The refinements of $a$ and $b$ are then $a=(A_1,\dots,A_4),b=(B_1,\dots,B_4)$ where
    \begin{align*}
        A_1=(0,2),\qquad A_2=(3),\qquad A_3=(7),\qquad A_4=(10),\qquad A_5=(13) \\
        B_1=(1),\qquad B_2=(3),\qquad B_3=(6,8),\qquad B_4=(11),\qquad B_5=\emptyset.
    \end{align*}
\end{example}

The following lemma is trivial and we omit the proof.

\begin{lemma}
    \label{lem:familyflips}
    Let $\Lambda=(a;b),\Lambda'\in\Lambda_X(n;k)$, $X\in\{B,C,D\}$ and suppose $\Lambda$ is monotonic.
    Form the refinements $a=(A_1,\dots,A_n),B=(B_1,\dots,B_n)$ of $a$ and $b$. 
    Then $\Lambda'\sim\Lambda$ if and only if $$\Lambda' = ((X_1,X_2,\dots,X_n);(Y_1,Y_2,\dots,Y_n))$$ where $\{X_i,Y_i\} = \{A_i,B_i\}$.
\end{lemma}

In the statement of the lemma, we say that the index $i$ for $\Lambda$ has been `flipped' if $X_i=B_i,Y_i=A_i$.
So in other words one can obtain all $\Lambda'\sim\Lambda$ by 
\begin{align*}
    X=B: &\text{ flipping indices in a manner that yields a symbol of defect 1.}\\
    X=C: &\text{ flipping indices in a manner that yields a symbol of defect 1,}\\
    &\text{except that the index $1$ cannot be flipped if $A_1$ starts with a $0$} \\
    X=D: &\text{ flipping some indices in a manner that yields a symbol of defect 0.}
\end{align*}
For decorated unordered s-symbols in type $D$ we have a similar description.
Let $\Lambda^\kappa=(a;b)^\kappa$ be monotonic.
The refinements of $\overline{\underline \Lambda}, a, b$ are defined the same as in the ordered case.
\begin{lemma}
    Let $\Lambda^\kappa=(a;b)^\kappa,\Lambda'^{\kappa'}\in\Lambda_D^\bullet(n;k)$.
    Form the refinements $a=(A_1,A_2,\dots,A_n)$ and $b=(B_1,B_2,\dots,B_n)$.
    We have that $\Lambda'^{\kappa'}\sim\Lambda^\kappa$ iff $$\Lambda' = \{(X_1,X_2,\dots,X_n);(Y_1,Y_2,\dots,Y_n)\}$$ where $\{X_i,Y_i\} = \{A_i,B_i\}$ and 
    \begin{align*}
       \begin{cases}
           \kappa=\kappa' & \mbox{ if } a=b \\
           \text{no restriction on } \kappa,\kappa' & \mbox{ otherwise}.
       \end{cases} 
    \end{align*}
\end{lemma}

\subsection{Preliminaries on a-symbols}

An a-symbol is a pair $\alpha:=(a;b)$ of sequences $a=(a_1< a_2 <\cdots <a_k)$, $b=(b_1 <b_2 <\dots <b_l)$ with $a_1,b_1\ge 0$.
The defect of $\alpha$ is $\#a-\#b$.
We say $\alpha$ is an a-symbol of type $X$ if 
\begin{align*}
    X=B: &\text{ $\alpha$ has defect 1}\\
    X=C: &\text{ $\alpha$ has defect 1}\\
    X=D: &\text{ $\alpha$ has defect 0}.
\end{align*}
Define $\sim$ the same as for s-symbols.
On the other hand, $\approx$ is defined for all types to be the equivalence relation generated by
\begin{align*}
    ((a_1,\dots,a_s);(b_1,\dots,b_t))\approx ((0,a_1+1,\dots,a_s+1);(0,b_1+1,\dots,b_t+1)).
\end{align*}
Define $[\bullet;\bullet]$ and $\phi(\bullet;\bullet)$, for a-symbols to be the same as for s-symbols.
For a sequence $a=(a_1< a_2< \cdots < a_k)$ define
\begin{align*}
    \rho_{a,0}(a) &= \sum_ia_i-(i-1) \\
    \rho_{a,1}(a) &= \sum_ia_i-(i-1)-1 \text{ provided $a_1>0$}.
\end{align*}
We define $\bar \alpha$ the same as for s-symbols.
Define $$\alpha_X(n) = \coprod_{k\ge0} \alpha_X(n;k)/\approx.$$
For $\alpha\in \coprod_{k\ge0} \alpha_X(n;k)$ we will simply write $[\alpha]$ for $\left[\alpha;\coprod_{k\ge0} \alpha_X(n;k)\right]$.
Define $[\alpha_1]\sim[\alpha_2]$ if there exists $\alpha_1'\in[\alpha_1],\alpha_2'\in[\alpha_2]$ such that $\alpha_1'\sim\alpha_2'$.
Write $\phi([\alpha])$ for the set of $[\alpha']\in \alpha_X(n)$ such that $[\alpha]\sim[\alpha']$.
For $k\ge 0$ and $\alpha\in\alpha_X(n;k)$ the map
\begin{equation}
    \phi(\alpha;\alpha_X(n;k))\to \phi([\alpha]), \qquad \alpha'\mapsto [\alpha']
\end{equation}
is a bijection.

For $X\in\{B,C\}$ define the map
$$\alpha:\mathrm{Irr}(W_X(n))\to \coprod_{j\ge0}\alpha_X(n;k),\qquad E\mapsto \alpha(E)$$ 
as follows. Suppose $E\in\mathrm{Irr}(W_X(n))$ corresponds to the bipartition $(\lambda,\mu)$ with $\lambda=(\lambda_1\le\lambda_2\le\cdots\lambda_k),\mu=(\mu_1\le\mu_2\le\cdots\mu_l)$.
If $k > l$ let $\lambda'=\lambda$, $\mu'=(0^{k-l-1})\cup_\le \mu$ and define 
\begin{equation}\label{eq:alphaE}
    \alpha(E) := (\lambda'+z^{0,k};\mu'+z^{0,k-1})
\end{equation}
If $k\le l$ let $\lambda=(0^{l-k+1})\cup_\le\lambda,\mu'=\mu$ and define $\alpha(E)$ as above but with $k$ replaced by $l$. We get a bijection
$$\mathrm{Irr}(W_X(n))\xrightarrow{\sim} \alpha_X(n), \qquad E \mapsto [\alpha(E)].$$
This bijection has the following properties \cite[Section 13.2]{Carter1993}
\begin{enumerate}
    \item $E_1\sim E_2$ if and only if $[\alpha(E_1)]\sim [\alpha(E_2)]$,
    \item $\alpha(E)$ is monotonic if and only if for all $\alpha\in [\alpha(E)]$, $\alpha$ is monotonic if and only if $E$ is special.
\end{enumerate}

The situation in type $D$ is much like situation in type $D$ for s-symbols.
We define unordered a-symbols, the defect, type $D$ unordered a-symbols, decorated unordered a-symbols, $\sim$ (resp. $\approx$, resp. $[\bullet;\bullet]$, resp. $\phi(\bullet;\bullet)$) for unordered a-symbols, $\sim$ (resp. $\approx$, resp. $[\bullet;\bullet]$, resp. $\phi(\bullet;\bullet)$) for decorated unordered a-symbols, monotonicity for decorated unordered a-symbols, and $\{\bullet\}$ analogously to the same notions for s-symbols, but with s replaced with a.

Write $\tilde\alpha_D(n;k)$ (resp. $\tilde\alpha_D^\bullet(n;k)$) for the set of unordered a-symbols $\{a;b\}$ (resp. decorated unordered a-symbols $\{a;b\}^\kappa$) of type $D$ with $\#b=k$ and $\rho_{a,0}(a)+\rho_{a,0}(b) = n$.
Define 
\begin{equation}
    \tilde\alpha_D(n) = \coprod_{k\ge0} \tilde\alpha_D(n;k)/\approx, \qquad \tilde\alpha_D^\bullet(n) = \coprod_{k\ge0} \tilde\alpha_D^\bullet(n;k)/\approx.
\end{equation}
For $\alpha \in \coprod_{k\ge0} \tilde\alpha_D(n;k)$ (resp. $\alpha^\kappa\in\coprod_{k\ge0} \tilde\alpha_D^\bullet(n;k)$) we will simply write $[\alpha]$ (resp. $[\alpha^\kappa]$) for $[\alpha;\coprod_{k\ge0} \tilde\alpha_D(n;k)]$ (resp. $[\alpha^\kappa;\coprod_{k\ge0} \tilde\alpha^\bullet_D(n;k)]$).
Write $\phi([\alpha])$ (resp. $\phi([\alpha^\kappa])$) for the set of $[\alpha']\in \tilde\alpha_D(n)$ (resp. $[\alpha'^{\kappa'}]\in\tilde\alpha_D^\bullet(n)$) such that $[\alpha]\sim[\alpha']$ (resp. $[\alpha^\kappa]\sim[\alpha'^{\kappa'}]$).
For $k\ge0$ and $\alpha\in\tilde\alpha_D(n;k)$ (resp. $\alpha^\kappa\in\tilde\alpha_D^\bullet(n;k)$) the map
\begin{align}
\begin{split}
    \phi(\alpha;\tilde\alpha_D(n;k)) &\to \phi([\alpha]), \qquad \alpha'\mapsto[\alpha'] \\
    \left(\text{resp. }\phi(\alpha^\kappa;\tilde\alpha_D^\bullet(n;k))\right. &\left.\to \phi([\alpha^\kappa]), \qquad \alpha'^{\kappa'}\mapsto[\alpha'^{\kappa'}]\right)
\end{split}
\end{align}
is a bijection.

Define 
$$\tilde\alpha^\bullet:\mathrm{Irr}(W_D(n))\to \coprod_{j\ge0}\tilde\alpha_D^\bullet(n;k),\qquad E\mapsto \tilde\alpha^\bullet(E)$$ 
as follows.
Suppose $E\in\mathrm{Irr}(W_D(n))$ corresponds to $\{\lambda,\mu\}^\kappa$.
Let $\lambda',\mu'$ be $\lambda,\mu$ padded with $0$'s in a similar manner to types $B$ and $C$ but to ensure they have the same number entries.
Define
\begin{equation}
    \tilde\alpha^\bullet(E) = \{\lambda'+z^{0,k};\mu'+z^{0,k}\}^\kappa
\end{equation}
where $k$ is the number of entries in $\lambda'$ (and hence also $\mu'$).
Define a map
$$\tilde\alpha:\mathrm{Irr}(W_D(n))\to \coprod_{j\ge0}\tilde\alpha_D(n;k),\qquad E\mapsto \tilde\alpha(E),$$
where $\alpha(E)$ is the underlying unordered symbol for $\tilde{\alpha}^{\bullet}(E)$. We get a bijection
$$\mathrm{Irr}(W_D(n)) \xrightarrow{\sim} \tilde\alpha_D^\bullet(n), \qquad E \mapsto [\tilde{\alpha}^{\bullet}(E)]$$
This bijection has the following properties \cite[Section 13.2]{Carter1993}
\begin{enumerate}
    \item $E_1\sim E_2$ if and only if $[\tilde\alpha^\bullet(E_1)]\sim [\tilde\alpha^\bullet(E_2)]$,
    \item $\tilde\alpha^\bullet(E)$ is monotonic if and only if for all $\alpha^\kappa\in [\tilde\alpha^\bullet(E)]$, $\alpha^\kappa$ is monotonic if and only if $E$ is special.
\end{enumerate}

For $E\in \mathrm{Irr}(W_D(n))$ write $\alpha(E)$ for $\underline{\tilde\alpha(E)}$.
For certain identities we will need to modify unordered defect $0$ a-symbols into defect $1$ a-symbols.
Write $\alpha_D^1(n;k)$ for the set of defect 1 symbols $(a;b)$ with $\#b=k$, $a_1=0, b_1>0$, and $\rho_{a,0}(a)+\rho_{a,1}(b) = n$.
Given a $\alpha\in \alpha_D(n;k)$ define $\alpha^!$ by 
\begin{equation}
    \alpha^! := ((0,a_1+1,a_2+1,\dots,a_k+1);(b_1+1,b_2+1,\dots,b_k+1)).
\end{equation}
This map induces a bijection $\alpha_D(n;k)\to\alpha_D^1(n;k)$.
Given a $\alpha\in \tilde\alpha_D(n;k)$ define $\alpha^!$ by $\underline\alpha^!$.
Note that for $\alpha,\beta\in\alpha_D(n;k)$ (resp. $\tilde\alpha_D(n;k)$) we have $\alpha\sim\beta$ iff $\alpha^!\sim\beta^!$.

\begin{lemma}
    \label{lem:flips}
    Let $\alpha=(a;b)$ be an a-symbol (of any defect) and suppose $\bar \alpha$ is decomposed into subsequences $(I_1,I_2,\dots,I_n)$.
    Let $a=(A_1,A_2,\dots,A_n)$ and $b=(B_1,B_2,\dots,B_n)$ be the corresponding decompositions of $a$ and $b$.
    We use the convention that $I_i=A_i=B_i=\emptyset$ for $i<1$ and $i>n$.
    Suppose for all $1\le j\le n$ and all $i<j<k$ and all $x\in I_i,y\in I_j,z\in I_k$ that $x<y<z$.
    Then flipping any number of indices yields an a-symbol (of possibly different defect).
\end{lemma}
\begin{proof}
    The condition guarantees that for $1\le j\le n$ and all $i<j<k$ and all $x\in A_i\cup B_i=I_i,y\in A_j\cup B_j=I_j,z\in A_k\cup B_k=I_k$ we have $x<y<z$ .
    Since each $A_i,B_i$ for $1\le i\le n$ is also increasing, flipping any number of indices yields an a-symbol.
\end{proof}
\begin{lemma}
    \label{lem:jump1}
    Let $\Lambda=(x;y)$ be a monotonic s-symbol of type $X^\vee$ and $\alpha_1=(a_1;b_1),\alpha_2=(a_2;b_2)$ be monotonic a-symbols of the same size and defect as $\Lambda$ such that $\Lambda = \alpha_1+\alpha_2$ (by which we mean $x=a_1+a_2$ and $y=b_1+b_2$).
    Let $\bar\Lambda=(I_1,\dots,I_n)$ be the refinement of $\bar\Lambda$, and let $T_i=|I_i|$.
    For $s \in \{1,2\}$, let $\bar\alpha_s=(J_1^s,\dots,J_n^s)$ be the decomposition of $\alpha_s$ corresponding to the refinement of $\bar\Lambda$ (i.e. $|J_i^s| = T_i$ for all $i$).
    Suppose for all $i$ such that $I_i,I_{i-1}$ are intervals we have that
    $$J_{i-1}^1(1) + 1 = J_{i}^1(T_{i}).$$
    Then the decomposition $\bar\alpha_s=(J_1^s,\dots,J_n^s)$ satisfies the conditions of Lemma \ref{lem:flips} for $s\in\{1,2\}$.
\end{lemma}
\begin{proof}
    If $n=1$ there is nothing to prove so suppose $n>1$.
    Since $\alpha_1,\alpha_2$ are both monotonic it suffices to show that $J_{i-1}^s(1)<J_i^s(T_i)$ for all $1<i\le n,s\in\{1,2\}$.
    Thus let $1<i\le n$.
    We consider three cases: $I_{i-1},I_i$ are both interval; $I_{i-1}$ is a pair; $I_i$ is a pair.
    Suppose $I_{i-1},I_i$ are both intervals.
    Then $I_{i-1}(1)\ll I_{i}(T_{i})$ since otherwise $I_{i-1}(1) + 1 = I_{i}(T_{i})$ and so $(I_{i-1},I_{i})$ would itself be an interval which would contradict the the definition of the $I_j$'s.
    Since $I_{j}(k)=J_{j}^1(k)+J_{j}^2(k)$ for all $1\le j\le n,1\le k\le T_i$ we must have that $J_{i-1}^2(1) < J_{i}^2(T_{i})$.
    Thus we have that $J_{i-1}^s(1) < J_{i}^s(T_{i})$ for $s\in\{1,2\}$.

    Now suppose $I_{i-1}$ is a pair say $I_{i-1}=(x,x)$.
    Since $\alpha_1,\alpha_2$ are monotonic we must have $J_{i-1}^s=(u^s,u^s)$ for some $u^s$ where $s\in\{1,2\}$ and $u^1+u^2 = x$.
    Write $J_i^s=(v_1^s,\dots,v_{T_i}^s)$.
    Since $\alpha_s$ is an a-symbol, $u^s=J_{i-1}^s(2)<J_i^s(T_i)=v_1^s$.
    In particular $J_{i-1}^s(1)<J_i^s(T_i)$ for $s\in\{1,2\}$ as required.

    The case when $I_i$ is a pair is analagous.
\end{proof}

\subsection{Springer correspondence via symbols}
In this section we describe how to compute the map $\lambda \mapsto [\alpha(E(\lambda,1))]$ in types $B/C/D$.
\subsubsection{Type $B$}
\label{sec:springerB}
Let $\lambda\in \mathcal P_B(n)$. 
Let $\lambda'$ be equal to $\lambda$ padded with $0$'s so that $\#\lambda'$ is odd.
Note that $\#\lambda$ is already necessarily odd, but the recipe still works if we pad $\lambda$ with $0$'s.
Write $\#\lambda'$ as $2k+1$ and $\lambda' = (\lambda_1'\le\lambda_2'\le\cdots\le\lambda_{2k+1}')$ in non-decreasing order.
Let $2\xi_1+1<2\xi_2+1<\cdots<2\xi_{k+1}+1$ and $2\eta_1<2\eta_2<\cdots<2\eta_k$ be an enumeration of the odd and even parts of $\lambda' + z^{0,2k+1}$.
Let $\xi = (\xi_1,\dots,\xi_{k+1})$ and $\eta = (\eta_1,\dots,\eta_k)$.
Then $[\alpha(E(\lambda,1))] = [(\xi;\eta)]$.
\subsubsection{Type $C$}
\label{sec:springerC}
Let $\lambda\in \mathcal P_C(n)$. 
Let $\lambda'$ be equal to $\lambda$ padded with $0$'s so that $\#\lambda'$ is odd.
Note that in type $C$, $\#\lambda$ may be either odd or even.
Write $\#\lambda'$ as $2k+1$ and $\lambda' = (\lambda_1'\le\lambda_2'\le\cdots\le\lambda_{2k+1}')$ in non-decreasing order.
Let $2\xi_1+1<2\xi_2+1<\cdots<2\xi_{k}+1$ and $2\eta_1<2\eta_2<\cdots<2\eta_{k+1}$ be an enumeration of the odd and even parts of $\lambda' + z^{0,2k+1}$.
Let $\xi = (\xi_1,\dots,\xi_{k+1})$ and $\eta = (\eta_1,\dots,\eta_k)$.
Then $[\alpha(E(\lambda,1))] = [(\eta;\xi)]$.
\subsubsection{Type $D$}
\label{sec:springerD}
Let $\lambda^\kappa\in \mathcal P_D(n)$. 
Let $\lambda'$ be equal to $\lambda$ padded with $0$'s so that $\#\lambda'$ is even.
Note that $\#\lambda$ is already necessarily even, but the recipe still works if we pad $\lambda$ with $0$'s.
Write $\#\lambda'$ as $2k$ and $\lambda' = (\lambda_1'\le\lambda_2'\le\cdots\le\lambda_{2k}')$ in non-decreasing order.
Let $2\xi_1+1<2\xi_2+1<\cdots<2\xi_{k}+1$ and $2\eta_1<2\eta_2<\cdots<2\eta_k$ be an enumeration of the odd and even parts of $\lambda' + z^{0,2k}$.
Let $\xi = (\xi_1,\dots,\xi_{k+1})$ and $\eta = (\eta_1,\dots,\eta_k)$.
Then $[\tilde\alpha^\bullet(E(\lambda^\kappa,1))] = [\{\xi;\eta\}^\kappa]$.
\begin{rmk}
    By the above constructions we have that $\Lambda(E(\lambda,1))\cap\Lambda_X(n;k)\ne\emptyset$ (resp. $\Lambda(E(\lambda^\kappa,1))\cap\Lambda_X(n;k)\ne\emptyset$) for all $k\ge \lfloor \#\lambda/2 \rfloor$ and $X\in\{B,C\}$ (resp. $X=D$).
\end{rmk}

\begin{lemma}
\label{lem:asymbolfordcollapse}
Suppose $\kappa\in\{0,1\}$, $\lambda \in \mathcal{P}_C(n)$ and $\lambda^t \in \mathcal{P}_D(n)$. 
Write $E=E((\lambda_D)^\kappa,1)$. Pad $\lambda$ with $0$'s so that $\#\lambda$ is even.
Let $2\xi_1+1<2\xi_2+1<\cdots<2\xi_{k}+1$ and $2\eta_1<2\eta_2<\cdots<2\eta_k$ be an enumeration of the odd and even parts of $\lambda + z^{0,\#\lambda}$.
Then 
\begin{equation}
    \label{eq:aSymbolComputation}
    [\alpha(E)] = [(\xi,\eta)].
\end{equation}
\end{lemma}

\begin{proof}
Since $\lambda$ is a $C$-partition, it admits a decomposition
$$\lambda = \lambda^1 \cup_{\leq} \lambda^2 \cup_{\leq} ... \cup_{\leq} \lambda^k$$
where $\lambda^1,...,\lambda^k$ are `elementary' partitions of the following two types:
\begin{itemize}
    \item[(i)] $\lambda^i = (e_1,o_1,...,o_{2m},e_2)$, where $e_1$, $e_2$ are even and $o_1,...,o_{2m}$ are odd (note: it might be the case that $m=0$)
    \item[(ii)] $\lambda^i = (o_1,...,o_{2m})$, where $o_1,...,o_{2m}$ are odd.
\end{itemize}

We will show that, in fact, all $\lambda^i$ are type (i). Since $\lambda^t \in \mathcal{P}_D(n)$, the largest part of $\lambda$ must be odd. So $\lambda^k$ is type (i). Suppose, for contradiction, that $\lambda^{i-1}$ is type (ii) for $2 \leq i \leq k$. Concatenating if necessary, we can assume that $\lambda^i$ is type (i). Write
$$\lambda^{i-1} = (o_1,...,o_{2p}), \qquad \lambda^i = (e_1,o'_1,...,o'_{2q},e_2), \qquad o_{2p} < e_1.$$
Then
$$m_{\lambda^t}(\mathrm{ht}_{\lambda}(e_1)) = e_1-o_{2p}.$$
Since $\mathrm{ht}_{\lambda}(e_1)$ is even and $\lambda^t \in \mathcal{P}_D(n)$, it follows that $e_1-o_{2p}$ is even. Contradiction. So all $\lambda^i$ are type (i), as asserted.

For each $\lambda^i = (e_1,o_1,...,o_{2p},e_2)$, write
$$\overline{\lambda}^i := (e_1+x_i,o_1+x_i,...,o_{2p}+x_i,e_2+x_i) \qquad x_i = \#\lambda - \mathrm{ht}_{\lambda}(e_1).$$
Since $x_i$ is even, $\overline{\lambda}^i$ is type (i). By (the type D anologue of ) equation (\ref{eq:Ccollapse}), we have
$$\lambda_D = (\overline{\lambda}^1)_D \cup_{\leq} ... \cup_{\leq} (\overline{\lambda}^k)_D.$$
So it suffices to prove the lemma for each $\overline{\lambda}^i$. Since each $\overline{\lambda}^i$ is a type (i) partition, we can reduce to the case in which $\lambda = (e_1,o_1,...,o_{2m},e_2)$. 

In this case, we have
$$\lambda_D = (e_1+1,o_1,...,o_{2p},e_2-1).$$
The right hand side of Equation (\ref{eq:aSymbolComputation}) can then readily be computed to be
$$\begin{pmatrix} \lfloor \frac{o_1+1}{2}\rfloor & \lfloor \frac{o_3+3}{2} \rfloor & \ldots & \lfloor \frac{o_{2p-1}+2p-1}{2} \rfloor & \lfloor \frac{e_2+2p}{2} \rfloor \\
\lfloor \frac{e+1}{2}\rfloor & \lfloor \frac{o_2+2}{2} \rfloor & \ldots & \lfloor \frac{o_{2p-2}+2p-2}{2} \rfloor & \lfloor \frac{o_{2p}+2p}{2} \rfloor 
\end{pmatrix}. $$
But by the algorithm in Section \ref{sec:springerD}
$$\tilde\alpha(E) \approx \begin{Bmatrix}
\lfloor \frac{e+1}{2}\rfloor & \lfloor \frac{o_2+2}{2} \rfloor & \ldots & \lfloor \frac{o_{2p-2}+2p-2}{2} \rfloor & \lfloor \frac{o_{2p}+2p}{2} \rfloor \\
 \lfloor \frac{o_1+1}{2}\rfloor & \lfloor \frac{o_3+3}{2} \rfloor & \ldots & \lfloor \frac{o_{2p-1}+2p-1}{2} \rfloor & \lfloor \frac{e_2+2p}{2} \rfloor
\end{Bmatrix} $$
and therefore
$$\alpha(E) \approx \begin{pmatrix}
\lfloor \frac{o_1+1}{2}\rfloor & \lfloor \frac{o_3+3}{2} \rfloor & \ldots & \lfloor \frac{o_{2p-1}+2p-1}{2} \rfloor & \lfloor \frac{e_2+2p}{2} \rfloor\\
\lfloor \frac{e+1}{2}\rfloor & \lfloor \frac{o_2+2}{2} \rfloor & \ldots & \lfloor \frac{o_{2p-2}+2p-2}{2} \rfloor & \lfloor \frac{o_{2p}+2p}{2} \rfloor
\end{pmatrix} $$
as required.
\end{proof}

\subsection{Restriction and induction via symbols}

Let $J\subsetneq\tilde\Delta$ be maximal.
Then $W_J$ is of the form 
\begin{align*}
    X=B: &\qquad W_m(D)\times W_{n-m}(B)\\
    X=C: &\qquad W_m(C)\times W_{n-m}(C)\\
    X=D: &\qquad W_m(D)\times W_{n-m}(D)
\end{align*}
for some $0\le m\le n$.
Define
\begin{equation}
    Y(X) = \begin{cases}
        D & \mbox{ if } X=B,D \\
        C & \mbox{ if } X=C.
    \end{cases}
\end{equation}
\begin{prop}
    \cite[Sections 4.5, 5.3, 6.3]{lusztig09}
    \label{prop:jinduction}
    Let $E_1\otimes E_2$ be an (irreducible) special representation of $W_J$ and $E = j_{W_J}^{W}E_1\otimes E_2$.
    Then for $k$ sufficiently large so that we can find $\Lambda \in [\Lambda(E)]\cap\Lambda_{X^\vee}(n;k), \alpha_1\in [\alpha(E_1)]\cap\alpha_{Y(X)}(n;k)$ and $\alpha_2\in [\alpha(E_2)]\cap\alpha_X(n;k)$ we have that
\begin{align*}
    X=B: &\qquad \Lambda = \alpha_1^!+\alpha_2\\
    X=C: &\qquad \Lambda = \alpha_1  +\alpha_2\\
    X=D: &\qquad \Lambda = \alpha_1  +\alpha_2.
\end{align*}
\end{prop}
\begin{prop}
    \label{prop:sumtohom}
    Let $E_1\otimes E_2\in\mathrm{Irr}(W_J)$ and $E\in\mathrm{Irr}(W_{X^\vee}(n))$. If $X=D$, assume that $\OO^{\vee}(E)$ is not very even.
    Let $\Lambda\in\Lambda_{X^\vee}(n;k)$ and
    \begin{itemize}
        \item [(B)] let $\alpha\in\alpha_D(m;k),\beta\in\alpha_B(n-m;k)$, and suppose $\Lambda = \alpha^!+\beta$. If $[\tilde\alpha(E_1)] = [\{\alpha\}]$, $[\alpha(E_2)]=[\beta]$, $[\Lambda(E)]=[\Lambda]$ then $\Hom(E_1\otimes E_2,E^\vee|_{W_J})\ne0$.
        \item [(C)] let $\alpha\in\alpha_C(m;k),\beta\in\alpha_C(n-m;k)$, and suppose $\Lambda = \alpha+\beta$. If $[\alpha(E_1)] = [\alpha]$, $[\alpha(E_2)=\beta]$, $[\Lambda(E)]=[\Lambda]$ then $\Hom(E_1\otimes E_2,E^\vee|_{W_J})\ne0$.
        \item [(D)] let $\alpha\in\alpha_D(m;k),\beta\in\alpha_D(n-m;k)$, and suppose $\Lambda = \alpha+\beta$. If $[\tilde\alpha(E_1)] = [\{\alpha\}]$, $[\tilde\alpha(E_2)]=[\{\beta\}]$, $[\tilde\Lambda(E)]=[\{\Lambda\}]$ then $\Hom(E_1\otimes E_2,E^\vee|_{W_J})\ne0$.
    \end{itemize}
\end{prop}
\begin{proof}
We will prove the type B case (the other cases are analogous). Write $(\mu,\nu)$,$\{\mu_1,\nu_1\}^{\kappa}$,$(\mu_2,\nu_2)$ for the bipartitions corresponding to $E$,$E_1$,$E_2$. In view of the formulas (\ref{eq:LambdaE}) and (\ref{eq:alphaE}) for the a- and s-symbols associated to a given bipartition, the assertion in (B) is equivalent to the following
$$(\mu,\nu) = (\mu_1,\nu_1) + (\mu_2,\nu_2) \implies \Hom(E_1 \otimes E_2,E^{\vee}|_{W_J}) \neq 0.$$
For a partition $\lambda$ of $k$, write $V(\lambda)$ for irreducible representation of $S_k$ corresponding to $\lambda$. Let $g(\mu_1,\mu_2,\mu)$ denote the multiplicity of $V(\mu_1) \otimes V(\mu_2)$ in the restriction of $V(\mu)$ to $S_{|\mu_1|} \times S_{|\mu_2|}$. By \cite[Theorem 3.2]{GeissengerKinch}, the multiplicity of $E_1 \otimes E_2$ in $E^{\vee}|_{W_J}$ is the product $g(\mu_1,\mu_2,\mu)g(\nu_1,\nu_2,\nu)$. Thus it suffices to show that $g(\mu_1,\mu_2,\mu),g(\nu_1,\nu_2,\nu) \neq 0$. This is an immediate consequence of the Littlewood-Richardson rule. We leave the details to the reader.
\end{proof}

\subsection{Construction of $(J,\phi)$ in classical types}\label{subsection:Jphi}

Now let $\mathfrak{g} = \mathfrak{g}_X(n)$ for $X \in \{B,C,D\}$ and fix a nilpotent orbit $\OO^{\vee} \subset \cN^{\vee}$ (which is not very even in type D). Let $\lambda \in \mathcal{P}_{X^\vee}(n)$ denote the partition corresponding to $\OO^{\vee}$. To show that $\OO^{\vee}$ is faithful, we must exhibit a pair $(J,\phi) \in \mathcal{F}_{\tilde\Delta}$ satisfying conditions (i) and (ii) of Definition \ref{def:faithful}. In this subsection, we will construct such a pair and verify condition (i) (condition (ii) will be verified in Section \ref{sec:prooffaithful}). We will construct $(J,\phi)$ by defining an explicit bipartition $^{\langle\mu\rangle}d(\lambda)$. The definition is as follows (we also define an auxilary bipartition $^{\langle\pi\rangle}d(\lambda)$ which is needed for the proofs).

\begin{definition}\label{def:pimu}
Define subpartitions $\pi(\lambda) \subseteq \lambda^t$ and $\mu(\lambda) \subseteq \lambda^t$ as follows
\begin{itemize}
    \item[(i)] If $X=C$, then 
    $$m_{\pi(\lambda)}(x) = \begin{cases} 
1 &\mbox{if } x \text{ is even and } m_{\lambda^t}(x) \text{ is}\\
& \mbox{odd}\\
0 &\mbox{otherwise} \end{cases}, \qquad m_{\mu(\lambda)}(x) = \begin{cases} 
2 &\mbox{if } x \text{ is even and } m_{\lambda^t}(x) \text{ is}\\
&\mbox{even and nonzero}\\
1 & \mbox{if } x \text{ is even and } m_{\lambda^t}(x) \text{ is}\\ 
& \mbox{odd}\\
0 &\mbox{otherwise} \end{cases} $$
    \item[(ii)] If $X=B$ or $X=D$, then 
    $$m_{\pi(\lambda)}(x) = \begin{cases} 
1 &\mbox{if } x \text{ is odd and } m_{\lambda^t}(x) \text{ is}\\
&\mbox{odd}\\
0 &\mbox{otherwise} \end{cases}, \qquad m_{\mu(\lambda)}(x) = \begin{cases} 
2 &\mbox{if } x \text{ is odd and } m_{\lambda^t}(x) \text{ is}\\
&\mbox{even and nonzero}\\
1 & \mbox{if } x \text{ is odd and } m_{\lambda^t}(x) \text{ is}\\
&\mbox{odd}\\
0 &\mbox{otherwise} \end{cases} $$
\end{itemize}
\end{definition}

We will shortly show that $\mu(\lambda)\subseteq d(\lambda)$, but note that care must be taken in the $X=B$ and $X=D$ cases since if $\mu(\lambda)\in\mathcal P(2)$ for $X\in\{B,D\}$ or $\mu(\lambda)\in\mathcal P(2n-2)$ for $X=D$, the bipartition $^{\langle \mu(\lambda) \rangle}d(\lambda)$ cannot be an element of $\overline{\mathcal P}(n)$.
These are inconvenient edge cases which we will treat separately now and ignore for the rest of the section.
\begin{prop}
    \label{prop:edgecases}
    Let $X\in\{B,D\}$ and $\lambda\in\mathcal P_{X^\vee}(n)$.
    Then
    \begin{enumerate}
        \item $\mu(\lambda) \in \mathcal P(2)$ if and only if $\lambda$ is of the form $\lambda=(\lambda_1,\lambda_2^{o_2},\lambda_3^{e_3},\dots,\lambda_k^{e_k})$ where $\lambda_1>\lambda_2\ge \lambda_3\ge \cdots\lambda_k\ge 0$, $\lambda_1-\lambda_2$ is even, $o_2>0$ is odd, and $e_3,\dots,e_k>0$ are even.
        \item if $X=D$, then $\mu(\lambda)\in\mathcal P(2n-2)$ if and only if $\lambda \in \{(1,1),(3,1)\}$.
    \end{enumerate}
\end{prop}
\begin{proof}
    \begin{enumerate}
        \item $(\Rightarrow)$ If $\mu(\lambda) \in \mathcal P(2)$ then since $\mu(\lambda)$ only has odd parts, $\mu(\lambda) = (1,1)$.
        By definition this means that $\lambda^t$ has only even parts except a non-zero even number of $1$'s.
        Thus $\lambda$ must take the form given.
        $(\Leftarrow)$ The only odd part of $\lambda^t$ is $1$ and it occurs with even multiplicity and so $\mu(\lambda) = (1,1)$.
        \item $(\Rightarrow)$ Suppose that $\mu(\lambda) \in\mathcal P(2n-2)$.
        Let $\nu = \lambda^t\setminus \mu(\lambda)$.
        Then $\nu \in \mathcal P(2)$ and so $\nu = (2)$ or $(1,1)$.
        Since $\mu$ only has odd parts, if $\nu = (1,1)$ then $\lambda^t$ only has odd parts.
        So $\lambda$ is of the form $(\lambda_1^{a_1},\lambda_2^{a_2},\dots,\lambda_k^{a_k})$ where $a_1$ is odd (in fact $=1$), and $a_i$ is even for all $i\ge 2$.
        Since $\lambda$ is of type $D$ this implies that $\lambda_1$ is odd.
        But $\lambda$ is a partition of $2n$ which is impossible since $\sum_i a_i\lambda_i \equiv 1 \pmod 2$.
        If $\nu = (2)$ then there are two cases to consider: $\lambda^t$ only has entries $\ge 2$ or $1$ is a part of $\lambda^t$.
        In the former case $\lambda$ must be of the form $(\lambda_1^{2},\lambda_2^{o_2},\lambda_3^{e_3},\dots,\lambda_k^{e_k})$ where $\lambda_1>\lambda_2\ge\lambda_3\ge\cdots\ge \lambda_k\ge 0$, $o_2>0$ is odd, $e_3,\dots,e_k>0$ are even.
        By parity considerations, this can only be a partition of $2n$ in the case $\lambda = (\lambda_1^2)$.
        But the multiplicity of $2$ in $\lambda^t$ is $1$ and so $\lambda_1 = 1$.
        Finally, suppose $1$ is a part of $\lambda^t$.
        Then $\lambda$ is of the form $(\lambda_1,\lambda_2,\lambda_3^{o_3},\lambda_4^{e_4},\dots,\lambda_k^{e_k})$ where $\lambda_1>\lambda_2>\lambda_3\ge\lambda_4\ge\cdots \ge\lambda_k\ge 0$, $o_3>0$ is odd, and $e_4,\dots,e_k>0$ are even.
        Again by parity considerations this is a partition of $2n$ only in the case when $\lambda = (\lambda_1,\lambda_2)$.
        Since $2$ has multiplicity $1$ in $\lambda^2$ we must have $\lambda_2 = 1$.
        Since all the parts of $\mu$ have multiplicity at most $2$, $\lambda=(3,1)$ or $(2,1)$.
        But we can rule out the $(2,1)$ case since $2+1$ is odd.
        $(\Leftarrow)$ This is trivial to check.
    \end{enumerate}
\end{proof}

Note that since $D_1\simeq A_1$ and $D_2 \simeq A_1\times A_1$, the orbits appearing in Proposition \ref{prop:edgecases} (2) are faithful since we have already shown that nilpotent orbits of simple lie algebras of type $A$ are faithful.
We now prove the remaining case.
\begin{prop}
    Let $X\in\{B,D\}$ and $\lambda$ be of the form $\lambda=(\lambda_1,\lambda_2^{o_2},\lambda_3^{e_3},\dots,\lambda_k^{e_k})$ where $\lambda_1>\lambda_2\ge \lambda_3\ge \cdots\lambda_k\ge 0$, $\lambda_1-\lambda_2$ is even, $o_2>0$ is odd, and $e_3,\dots,e_k>0$ are even.
    Then $\OO^\vee$ is faithful.
\end{prop}
\begin{proof}
    Let $(J,\phi) = (\Delta,\phi(E(\OO^\vee,1)))$.
    Since $\lambda$ is special we have that $$\overline{\mathbb L}(J,\OO(\phi)) = \overline{\mathbb L}(\Delta,d(\OO^\vee)) = (d(\OO^\vee),1) = d_A(\OO^\vee,1).$$
    This verifies condition (i) for faithfulness.
    Since $W_J = W$ for $J=\Delta$, proving condition (ii) for our choice of $(J,\phi)$ is equivalent to showing that all representations $E\in\mathrm{Irr}(W)$ with $\OO^\vee(E) = \OO^\vee$ lie in the same family as $E(\OO^\vee,1)$.
    This is an easy exercise in the combinatorics of s-symbols and a-symbols and is omitted.
\end{proof}

For the remainder of this section we assume that $\lambda$ is not equal to one of the edge cases mentioned.

\begin{lemma}\label{lem:propsofmu}
The following are true:
\begin{itemize}
    \item[(i)] $\pi \subseteq d(\lambda)$.
    \item[(ii)] $\mu \subseteq d(\lambda)$.
    \item[(iii)] $^{\langle\pi\rangle}d(\lambda) \in \overline{\mathcal{P}}^*(n)$.
    \item[(iv)] $^{\langle\mu\rangle}d(\lambda) \in \overline{\mathcal{P}}^*(n)$.
\end{itemize}
\end{lemma}

\begin{proof}
We will prove (i)-(iv) for $X=C$ (the proofs for $X=B,D$ are analogous, and are omitted). 

Since $\pi \subseteq \mu$, (ii) implies (i). And since all parts of $\pi \setminus \mu$ have even multiplicity, (iv) implies (iii). So it is enough to prove (ii) and (iv).

Choose $\nu$ so that $\mu \cup \nu = \lambda^t$. Note that $|\nu|$ is odd. We will first show that
\begin{equation}\label{eq:dlambdamu}
d(\lambda) = \mu \cup (\nu^-)_C
\end{equation}
There are two cases to consider. 

\vspace{3mm}

(a) $\mu \subseteq (\lambda^t)^-$. In this case we have $(\lambda^t)^- = \mu \cup \nu^-$. Since $\mu$ consists only of even parts, equation (\ref{eq:Ccollapse}) implies
$$d(\lambda) = ((\lambda^t)^-)_C = \mu \cup (\nu^-)_C,$$
as desired.

\vspace{3mm}
(b) $\mu \nsubseteq (\lambda^t)^-$. Write 
$$\nu = \nu' \cup_{\geq} (x,y_1,...,y_r),$$
where $x$ is the smallest odd part of $\nu$ (note that $r$ might be $0$). Then by (\ref{eq:Ccollapse}) we have
\begin{equation}\label{eq:nu'}(\nu^-)_C = (\nu')_C \cup_{\geq} (x-1,y_1,...,y_r)\end{equation}
Now let $z$ be the smallest part of $\mu$ and choose $\mu'$ such that $\mu = \mu' \cup (z)$. Then 
$$(\lambda^t)^- = \mu' \cup [\nu' \cup_{\geq} (x,y_1,...,y_r,z-1)].$$
So by (\ref{eq:Ccollapse}) we have
\begin{align}\label{eq:dlambda}
\begin{split}
d(\lambda) &= ((\lambda^t)^-)_C\\
&= \mu' \cup [(\nu')_C \cup_{\geq} (x-1,y_1,...,y_r,z)]\\
&= \mu \cup [(\nu')_C \cup_{\geq} (x-1,y_1,...,y_r)]
\end{split}
\end{align}
Substituting (\ref{eq:nu'}) into the right hand side of (\ref{eq:dlambda}), we obtain
$$d(\lambda) = \mu \cup (\nu^-)_C,$$
as desired.

\vspace{3mm}

Thus, we have proved (\ref{eq:dlambdamu}). (ii) follows immediately. In view of (\ref{eq:dlambdamu}), (iv) amounts to the assertion that $(\nu^-)_C$ is special. This will require some notation. 

For $\alpha \in \mathcal{P}(2n)$, define
$$e_{\alpha}(i,j) := \#\{k \mid i \leq k \leq j \text{ and } \alpha_k \text{ is even}\}$$
Let $\alpha_{t_1},...,\alpha_{t_p}$ be the odd parts of $\alpha$ (for $t_1 < t_2 < ... < t_p$). We say that
\begin{itemize}
    \item $\alpha$ is \emph{special} if $e_{\alpha}(1,t_1)$ is even and  $e_{\alpha}(t_i,t_{i+1})$ is even for $1 \leq i \leq p-1$.
    \item $\alpha$ is \emph{quasi-special} if $e_{\alpha}(1,t_1)$ is even,  $e_{\alpha}(t_i,t_{i+1})$ is even for $1 \leq i \leq p-2$, and the smallest part of $\alpha$ is odd, with multiplicity $1$.
\end{itemize}
    
By (\ref{eq:Ccollapse}), the odd parts of $\alpha_C$ are precisely those of the form $(\alpha_C)_{t_{2i-1}}$ or $(\alpha_C)_{t_{2i}}$ for $i$ satisfying $\alpha_{t_{2i-1}}=\alpha_{t_{2i}}$. Choose $i < j$ such that 
$$\alpha_{t_{2i-1}} = \alpha_{t_{2i}} \qquad \text{and} \qquad \alpha_{t_{2j-1}} = \alpha_{t_{2j}}.$$
Then by (\ref{eq:Ccollapse})
$$e_{\alpha_C}(t_{2i},t_{2j-1}) = e_{\alpha}(t_{2i},t_{2j-1}) + 2(j-i-1).$$
In particular,
$$e_{\alpha_C}(t_{2i},t_{2j-1}) \equiv e_{\alpha}(t_{2i},t_{2j-1}) \mod{2}$$
The following implications are immediate
\begin{itemize}
    \item $\alpha \text{ special} \implies \alpha_C \text{ special}$.
    \item $\alpha \text{ quasi-special} \implies \alpha_C \text{ special}.$
\end{itemize}

Returning to the situation at hand, there are two cases to consider. First, suppose the smallest part of $\nu$ is even. In this case, $\nu^-$ is quasi-special. So $(\nu^-)_C$ is special. Next, suppose the smallest part of $\nu$ is odd. In this case, $\nu^-$ is special. So $(\nu^-)_C$ is special. In either case, $(\nu^-)_C$ is special, as desired.
\end{proof}

\begin{rmk}
If $X=D$, then $\mu$ consists of only odd parts. In particular, it is never very even.
\end{rmk}

Write $(J,\OO) \in \mathcal{K}_{\tilde{\Delta}}^{max}$ (resp. $(J',\OO') \in \mathcal{K}_{\tilde{\Delta}}^{max}$) for the pairs corresponding to the bipartition $^{\langle\mu\rangle}d(\lambda) \in \overline{\mathcal{P}}_X(n)$ (resp. $^{\langle\pi\rangle}d(\lambda) \in \overline{\mathcal{P}}_X(n)$) under the bijection (\ref{eq:KPbijection}). And write
\begin{equation}\label{eq:Jphi}(J,\phi) := (J,\phi(E(\OO,1) \otimes \mathrm{sgn})) \in \mathscr{F}_{\tilde{\Delta}}, \qquad (J',\phi') := (J', \phi(E(\OO',1) \otimes \mathrm{sgn})) \in \mathscr{F}_{\tilde{\Delta}}.\end{equation}
%
%
It is easy to show that both $(J,\phi)$ and $(J',\phi')$ satisfy condition (i) of Definition \ref{def:faithful}.

\begin{lemma}\label{lem:conditioniclassical}
$\overline{\mathbb L}(J,\OO(\phi)) = \overline{\mathbb L}(J',\OO'(\phi')) = d_A(\OO^{\vee},1)$.
\end{lemma}

\begin{proof}
Note that $\phi(E(\OO,1)\otimes \mathrm{sgn})$ is the family containing $E(\OO,1)$ and so $\OO(\phi) = \OO$ since $\OO$ is special (and similarly for $\OO'$). So it suffices to show
$$\overline{\mathbb L}(J,\OO) = \overline{\mathbb L}(J',\OO') = d_A(\OO^{\vee},1).$$
These are equivalent to the equalities
$$\bar{s}(^{\langle\mu\rangle}d(\lambda)) = \bar{s}(^{\langle\pi\rangle}d(\lambda)) = d_A(\OO^{\vee},1).$$
The second equality follows imediately from Achar's definition of $d_A$ (see Equation (9) in \cite[Section 4]{Acharduality}). The first equality, by Proposition \ref{prop:Acharequivalence}, is equivalent to the following
$$r_{d(\lambda)}(\mu) = r_{d(\lambda)}(\pi).$$
But this is clear from definitions: $\mu$ is obtained from $\pi$ by adding parts of even multiplicity. So for each $x \in \ZZ_{\geq 0}$, we have
$$\mathrm{ht}_{\mu}(x) \equiv \mathrm{ht}_{\pi}(x) \mod{2},$$
and therefore
$$m_{r_{d(\lambda)}(\mu)}(x) = m_{r_{d(\lambda)}(\pi)}(x),$$
as required.
\end{proof}

\subsection{Proof of faithfulness in classical types}\label{sec:prooffaithful}

Continue with the notations of Subsection \ref{subsection:Jphi}, i.e. $\OO^{\vee}$, $^{\langle\mu\rangle}d(\lambda)$, $(J,\phi)$, and so on. In this subsection, we will show that $(J,\phi)$ satisfies condition (ii) of Definition \ref{def:faithful}. This will allow us to complete the proof of Theorem \ref{thm:faithful} for types B/C/D. 

The proof of condition (ii) will require a sequence of technical lemmas.

Define 
\begin{equation}
    \omega_X = \begin{cases}
        1 & \mbox{ if $X=B$} \\
        0 & \mbox{ if $X=C$} \\
        1 & \mbox{ if $X=D$}
    \end{cases}    
\end{equation}
and
\begin{equation}
    \chi(x) = \begin{cases}
        0 & \mbox{ if $x$ is even} \\
        1 & \mbox{ if $x$ is odd}.
    \end{cases}
\end{equation}
Write $\lambda=(\lambda_1^{p_1},\lambda_2^{p_2},\dots,\lambda_l^{p_l})$ where $p_i>0$ and $0<\lambda_1<\lambda_2<\cdots<\lambda_l$.
Note that if $X=C$ then $\#\lambda$ is odd, if $X=D$ then $\#\lambda$ is even, and if $X=B$ then $\#\lambda$ can be either even or odd.
However, in type $B$, when computing a/s-symbols we will want to pad $\lambda$ with 0's to ensure it has an odd number of entries.
Thus we define
\begin{equation}
    \lambda' = (\lambda_0^{p_0},\lambda_1^{p_1},\dots,\lambda_l^{p_l})
\end{equation}
where $\lambda_0 = 0$ and $p_0=0$ if $X\in\{C,D\}$ and $p_0$ is some integer $>0$ of the opposite parity to $\#\lambda$ if $X=B$.
Let $P_i = \sum_{j=0}^ip_j$ and $Q_i = \sum_{j=i}^lp_j$.
We have that $\#\lambda' = P_{i-1}+Q_{i}$ for all $0\le i\le l+1$.
\begin{lemma}\label{lem:technicallemma1}
    Let $\Lambda \in [\Lambda(E(\lambda,1))]\cap \Lambda_{X^\vee}(n;k)$ where $k>\lfloor \#\lambda/2\rfloor$.
    Let $\overline\Lambda = (I_1,I_{2},\dots,I_n)$ be the refinement for $\bar\Lambda$.
    Let $1\le a_1 < \cdots < a_u \le n$ be the indices such that $I_a$ is an interval.
    Let $0\le b_1 < \cdots < b_v \le l$ be the indices such that $\lambda_b \equiv \omega_{X^\vee}\pmod 2$.
    Then $u = v$ and $I_{a_i}=(\overline\Lambda(Q_{b_i}),\dots,\overline\Lambda(Q_{b_i+1}+1))$ for all $i$.
    When $X=B$ we additionally have $b_1 = 0$ and $I_1(\#I_1) = 0$.
\end{lemma}
\begin{proof}
First consider the case when $X=B$.
The condition $k>\lfloor \#\lambda/2\rfloor$ translates to the $p_0>0$ condition from the paragraph above. 
Thus we need only compute $\Lambda(E(\lambda,1))$ using the algorithm in Paragraph \ref{sec:springerB} with $p_0>0$ as above.
Write $\#\lambda = 2k+1$.
We can write $\lambda + z^{0,\#\lambda}$ as $(K_0,\dots,K_l)$ where $K_i = \lambda_i + P_{i-1} + (0,1,\dots,p_i-1)$.
Then the even entries of $K_i$ are $2x_i$ and the odd entries of $K_i$ are $2y_i+1$ where
\begin{align}
    x_i &= \lfloor (\lambda_i+P_{i-1})/2\rfloor + (0,1,\dots) + \chi(\lambda_i+P_{i-1}) \\
    y_i &= \lfloor (\lambda_i+P_{i-1})/2\rfloor + (0,1,\dots)
\end{align}
and
\begin{align}
    \#x_i = \begin{cases}
        \lceil p_i/2\rceil & \mbox{ if $\lambda_i + P_{i-1}$ is even} \\
        \lfloor p_i/2\rfloor & \mbox{ if $\lambda_i + P_{i-1}$ is odd}
    \end{cases},
    \qquad
    \#y_i = \begin{cases}
        \lfloor p_i/2\rfloor & \mbox{ if $\lambda_i + P_{i-1}$ is even} \\
        \lceil p_i/2\rceil & \mbox{ if $\lambda_i + P_{i-1}$ is odd}.
    \end{cases}
\end{align}
Note that $\#x_i+\#y_i = p_i$ for all $i$.
Let $x = (x_1,\dots,x_l)$ and $y = (y_1,\dots,y_l)$.
Then $\Lambda := (x+z^{0,k+1};y+z^{0,k}+1)$ is an s-symbol for $E(\lambda,1)$.
Decompose $\Lambda$ as $((x_0',\dots,x_l');(y_0',\dots,y_l'))$ where $\# x_i' = \# x_i$ and $\# y_i' = \# y_i$.
We claim that $\bar \Lambda = (B_1,\dots,B_l)$ where 
\begin{equation}
    B_i = \begin{cases}
        \overline{(x_i';y_i')} & \mbox{ if $P_{i-1}$ is even} \\
        \overline{(y_i';x_i')} & \mbox{ if $P_{i-1}$ is odd}.
    \end{cases}
\end{equation}
We prove this by inducting on the content of the first $P_i$ entries of $\bar \Lambda$. 
For $i=0$, we have $P_{i-1} = P_{-1} = 0$.
If $p_0$ is odd then $\#x_1 = \#y_1+1$ and $B_1 = \overline{(x_1';y_1')}$.
If $p_1$ is even then $\#x_1 = \#y_1$ and so again $B_1 = \overline{(x_1';y_1')}$.
Now suppose $i>0$.
If $P_{i-1}$ is even then by the inductive hypothesis the first entry of $B_i$ must be the first entry of $x_i$ and so since $p_i = \#x_i'+\#y_i'$ we only need to show that $x_i'$ has $0$ or $1$ more entries than $y_i'$.
If $p_i$ is odd then $\lambda_i$ is even and so $\#x_i = \#y_i + 1$ as required.
If $p_i$ is even then $\#x_i = \#y_i$ as required.
The case when $P_{i-1}$ is odd is analogous.

It follows from this that 
\begin{align}
    x_i' &= \lfloor P_{i-1}/2 \rfloor + (0,1,\dots) + x_i + \chi(P_{i-1}) \nonumber \\
         &= \lfloor \lambda_i/2\rfloor + P_{i-1} + (0,2,\dots) + \chi(P_{i-1})\vee \chi(\lambda_i) \\
    y_i' &= \lfloor P_{i-1}/2 \rfloor + (0,1,\dots) + y_i + 1 \nonumber \\
         &= \lfloor \lambda_i/2\rfloor + P_{i-1} + (1,3,\dots) - \chi(P_{i-1})\wedge \chi(\lambda_i+P_{i-1}).
\end{align}
and so
\begin{equation}
    \label{eq:blockform}
    B_i = \lfloor \lambda_i/2 \rfloor + P_{i-1} + \begin{cases}
        (0,1,2,3,\dots) & \mbox{ $\lambda_i$ is even} \\
        (1,1,3,3,\dots) & \mbox{ $\lambda_i$ is odd}
    \end{cases}
\end{equation}
where $\#B_i$ is even when $\lambda_i$ is odd since then $p_i = \#B_i$ is even.
Finally, if both $\lambda_{i-1},\lambda_i$ are even then $B_{i-1}(1) = \lfloor \lambda_{i-1}/2\rfloor + P_{i-1} - 1$ and $B_i(p_i) = \lfloor \lambda_i/2\rfloor + P_{i-1}$.
Since $\lambda_{i-1}\ll \lambda_i$ it follows in particular that $(B_{i-1},B_i)$ is not an interval.
Thus the intervals of $\bar\Lambda$ are exactly the $B_i$ where $\lambda_i$ is even.
Since $B_i = (\bar\Lambda(Q_{i}),\dots,\bar\Lambda(Q_{i+1}+1))$, the result follows.
For the observation that $b_1 = 0$ simply note that $\lambda_0 = 0 \equiv \omega_{C} = 0 \pmod 2$.
Then $I_1(\#I_1)=0$ follows from Equation (\ref{eq:blockform}).

The cases $X=C,D$ are similar (in fact simpler) and left to the reader.
We remark only that in these cases $\approx$ only adds/removes pairs to the start of $\bar\Lambda$ which does not affect the position (indexed from the right) of the intervals in $\bar\Lambda$.
Thus it suffices to prove the claim for a single s-symbol in the $\approx$ equivalence class.
\end{proof}

Let 
\begin{equation}
    \alpha_1 := \begin{cases}
        \alpha(E(d_{LS}(\mu),1))^! & \mbox{ if $X=B$}\\
        \alpha(E(d_{LS}(\mu),1)) & \mbox{ if $X\in\{C,D\}$}.
    \end{cases}
\end{equation}

We will now provide a more explicit description of $\alpha_1$.

We have that $\lambda^t=(Q_l^{q_l},Q_{l-1}^{q_{l-1}},\dots,Q_1^{q_1})$ where $q_i = \lambda_i-\lambda_{i-1}$.
Let $1\le c_1<c_2<\cdots<c_r\le l$ be all the indices $c$ such that $Q_c\equiv \omega_X\pmod 2$.
Let 
\begin{equation}
    \eta_i = \begin{cases}
        1 & \mbox{ if } q_i \text{ is odd} \\
        2 & \mbox{ if } q_i \text{ is even}.
    \end{cases}
\end{equation}
Then by definition $\mu = (Q_{c_r}^{\eta_{c_r}},Q_{c_{r-1}}^{\eta_{c_{r-1}}},\dots,Q_{c_1}^{\eta_{c_1}})$.
Define $H_i = \sum_{j\le i}\eta_{c_j}$ and $t_i = Q_{c_i}-Q_{c_{i+1}}$ with the convention that $c_{r+1} = l+1$.
Then $\mu^t = (H_1^{t_1},H_2^{t_2},\dots,H_r^{t_r})$.
Let $d_i = c_i-1$.
Then $t_i = \#\lambda-P_{d_i}-\#\lambda+P_{d_{i+1}} = P_{d_{i+1}}-P_{d_i}$.
Thus $p(\mu^t) = \#\lambda-P_{d_1}$.
It will be useful to pad $\mu^t$ with 0's so that it has the same number of parts as $\lambda'$ so we write it as $\mu^t = (H_0^{t_0},H_1^{t_1},\dots,H_r^{t_r})$ with the convention that $c_0=0$.
Write $t_i'=\lfloor t_i/2\rfloor$, $H_i'=\lfloor H_i/2\rfloor$, $P_i'=\lfloor P_i\rfloor$.
Note that we always have $P_{d_{i-1}}'+t_{i-1}'=P_{d_i}'$ for all $1\le i\le r$ since if $i=1$, $P_{d_{i-1}}=0$ and if $i>1$, $t_{i-1}$ is even.
We now record $\alpha_1$ for the various types.
We only show the calculation in type $B$ explicitly as this is the most difficult case. 
The calculations in types $C$ and $D$ are similar, but more straightforward (there is no need to define $\tilde t_i$, $\tilde P_i$ etc.).

\vspace{5mm}
\paragraph{Type $B$}
\label{par:typeB}
By Lemma \ref{lem:propsofmu} we have that $^{\langle\mu\rangle}d(\lambda)\in \overline{\mathcal P}^*(n)$ and so $\mu$ is a special type $D$ partition.
By \cite[Proposition 6.3.7]{CM} this means that $\mu^t$ is a partition of type $C$ and so we can compute $\alpha(E(d_{LS}(\mu),1))$ using Lemma \ref{lem:asymbolfordcollapse}.

Write $\mu^t$ as $\mu^t=(H_0^{\tilde t_0},H_1^{\tilde t_1},\dots,H_r^{\tilde t_r})$ where $\tilde t_0= t_0-1\ge0$ and $\tilde t_i= t_i$ for $1\le i \le r$ so that $\mu^t$ has an even number of entries.
Define $\tilde p_0 = p_0-1$ and $\tilde p_i = p_i$ for $1\le i\le l$ and set $\tilde P_{i} = \sum_{j=0}^i\tilde p_j$ and $\tilde Q_{i} = \sum_{j=i}^l\tilde p_j$.
We have $\tilde t_i = \tilde P_{d_{i+1}}-\tilde P_{d_i}$.
Write $\tilde t_i'=\lfloor \tilde t_i/2\rfloor$ and $\tilde P_i'=\lfloor \tilde P_i\rfloor$.
Note that similar to before we always have $\tilde P_{d_{i-1}}'+\tilde t_{i-1}'=\tilde P_{d_i}'$ for all $1\le i\le r$.
Then $\mu^t + z^{0,\#\mu^t}$ can be decomposed into subsequences $(K_0,K_1,\dots,K_r)$ where $K_i = H_i + \tilde P_{d_i} + (0,1,\dots,\tilde t_i-1)$.
Since $\tilde P_{d_i}+ \tilde Q_{c_i} = \#\lambda'-1$ we have that $\tilde P_{d_i}\equiv \tilde Q_{c_i} \pmod 2$. 
But $\tilde Q_{c_i} = Q_{c_i}$ and $Q_{c_i}\equiv \omega_B \equiv 1 \pmod 2$ for $1\le i \le r$ and so $\tilde P_{d_i}$ is odd for $1 \le i \le r$, $P_{d_0}=0$ is even, $\tilde t_i$ is even for $1 \le i < r$ and $\tilde t_0,\tilde t_r$ are odd.
Thus the odd entries of $K_i$ are $2x_i+1$ where 
\begin{equation}
    x_0 = (0,1,\dots), \qquad x_i = \tilde P_{d_i}' + H_i' + (0,1,\dots) + \chi(H_i), \text{ for } 1\le i\le r,
\end{equation}
the even entries are $2y_i$ where 
\begin{equation}
    y_0 = (0,1,\dots), \qquad y_i = \tilde P_{d_i}'+H_i'+(0,1,\dots)+1, \text{ for } 1\le i\le r, 
\end{equation}
and
\begin{equation}
    \# x_i = \begin{cases}
        \tilde t_i' & \mbox{ if } i<r \\
        \tilde t_r'+1 & \mbox{ if } i=r
    \end{cases}, \qquad
    \# y_i = \begin{cases}
        \tilde t_0'+1 & \mbox{ if } i=0 \\
        \tilde t_i' & \mbox{ if } i>0.
    \end{cases}
\end{equation}
Therefore
\begin{equation}
    \alpha_1 = ((a_0,a_1,\dots,a_r);(b_0,b_1,\dots,b_r))
\end{equation}
where
\begin{align*}
    a_0&=(0,1,\dots,\tilde t_0')\\
    b_0&=(0,1,\dots,\tilde t_0') + 1 \\
    a_i&=\tilde P_{d_i}'+H_i'+(1,2,\dots,\tilde t_i') + \chi(H_i)\\
    b_i&=\tilde P_{d_i}'+H_i'+(1,2,\dots,\tilde t_i') + 1 \\
    a_r&=\tilde P_{d_r}'+H_r'+(1,2,\dots,\tilde t_r',\tilde t_r'+1) + \chi(H_r)\\
    b_r&=\tilde P_{d_r}'+H_r'+(1,2,\dots,\tilde t_r') + 1
\end{align*}
for $0<i<r$.
Using the fact that $t_0' = \tilde t_0'+1$, $t_i' = \tilde t_i'$, and $P_{d_i}' = \tilde P_{d_i}'+1$ for $1\le i \le r$ we get that
\begin{align*}
    a_0&=(0,1,\dots,t_0'-1)\\
    b_0&=(0,1,\dots,t_0'-1) + 1 \\
    a_i&=P_{d_i}'+H_i'+(0,1,\dots,t_i'-1) + \chi(H_i)\\
    b_i&=P_{d_i}'+H_i'+(0,1,\dots,t_i'-1) + 1 \\
    a_r&=P_{d_r}'+H_r'+(0,1,\dots,t_r'-1,t_r') + \chi(H_r)\\
    b_r&=P_{d_r}'+H_r'+(0,1,\dots,t_r'-1) + 1.
\end{align*}
We also have
\begin{equation}
    \overline{\alpha_1} = (B_0,B_1,\dots,B_r)
\end{equation}
where $B_i = \overline{(a_i;b_i)}$ for $0\le i\le r$.

\vspace{5mm}
\paragraph{Type $C$}
\begin{equation}
    \alpha_1 = ((a_0,a_1,\dots,a_r);(b_0,b_1,\dots,b_r))
\end{equation}
where
\begin{align*}
    a_0&=(0,1,\dots,t_0')\\
    b_0&=(0,1,\dots,t_0'-1)\\
    a_i&=P_{d_i}'+H_i'+(0,1,\dots,t_i'-1) + 1\\
    b_i&=P_{d_i}'+H_i'+(0,1,\dots,t_i'-1) + \chi(H_i).
\end{align*}
for $0<i\le r$.
We also have
\begin{equation}
    \overline{\alpha_1} = (B_0,B_1,\dots,B_r)
\end{equation}
where $B_0=\overline{(a_0;b_0)}$ and $B_i=\overline{(b_i;a_i)}$ for $0<i\le r$.

\vspace{5mm}
\paragraph{Type $D$}
\begin{equation}
    \alpha_1 = ((a_0,a_1,\dots,a_r);(b_0,b_1,\dots,b_r))
\end{equation}
\begin{align*}
    a_0&=(0,1,\dots,t_0'-1) \\
    b_0&=(0,1,\dots,t_0'-1,t_0')\\
    a_i&=P_{d_i}'+H_i'+(0,1,\dots,t_i'-1) + \chi(H_i)\\
    b_i&=P_{d_i}'+H_i'+(0,1,\dots,t_i'-1) + 1\\
    a_r&=P_{d_r}'+H_r'+(0,1,\dots,t_r'-1,t_r') + \chi(H_r)\\
    b_r&=P_{d_r}'+H_r'+(0,1,\dots,t_r'-1) + 1
\end{align*}
for $0<i<r$.
We also have
\begin{equation}
    \overline{\alpha_1} = (B_0,B_1,\dots,B_r)
\end{equation}
where $B_0=\overline{(b_0;a_0)}$ and $B_i=\overline{(a_i;b_i)}$ for $0<i\le r$.

\begin{lemma}
    \label{lem:endparity}
    For all $1\le i< r$ we have $\lambda_{c_i}\equiv\lambda_{c_{i+1}-1} \pmod 2$.
\end{lemma}
\begin{proof}
    Suppose first $c_{i+1}-c_i=1$.
    Then $\lambda_{c_i}=\lambda_{c_{i+1}-1}$ so the lemma holds trivially.
    Now consider the case $c_{i+1}-c_i>1$.
    Then by definition $Q_c\equiv\omega_X\pmod 2$ for $c=c_i,c_{i+1}$ and $Q_c\nequiv \omega_X\pmod 2$ for all $c_{i}<c<c_{i+1}$.
    For this to be the case we must have that $p_{c_i},p_{c_{i+1}-1}$ are both odd.
    But then $\lambda_{c_i},\lambda_{c_{i+1}-1}\equiv \omega_{X^\vee}\pmod 2$.
    In all cases we have $\lambda_{c_i}\equiv\lambda_{c_{i+1}-1}\pmod 2$ as required.
\end{proof}
For $0\le i<r$ define 
\begin{equation}
    \Delta_i = \begin{cases}
        0 & \mbox{ if $\lambda_{c_{i+1}-1}$ is even} \\
        1 & \mbox{ if $\lambda_{c_{i+1}-1}$ is odd},
    \end{cases}
\end{equation}
and
\begin{equation}
    \Delta_r = \begin{cases}
        0 & \mbox{ if $\lambda_{c_{r}}$ is even} \\
        1 & \mbox{ if $\lambda_{c_{r}}$ is odd}.
    \end{cases}
\end{equation}
Then for all $0\le i\le r$, we have $\Delta_i\equiv \lambda_{c_i} \pmod 2$ (where we use the lemma for $1\le i<r$) and so $\eta_{c_i} \equiv \lambda_{c_{i}}-\lambda_{c_{i}-1} \equiv \Delta_i-\Delta_{i-1}\pmod 2$ for all $1\le i\le r$.
It follows that $H_i = \sum_{j\le i}\eta_{c_i}\equiv \Delta_i-\Delta_0$.
\begin{lemma}
    We have that $\Delta_0=\omega_{X^\vee}$.
\end{lemma}
\begin{proof}
    First note that $c_1=1$ if and only if $\#\lambda = Q_1 \equiv \omega_X \pmod 2$.
    But if $X=C$ then $\#\lambda$ is odd which is $\nequiv \omega_C = 0 \pmod 2$, and if $X=D$ then $\#\lambda$ is even which is $\nequiv \omega_D = 1 \pmod 2$.
    Thus $c_1=1$ if and only if $X=B$ and $\#\lambda$ is odd.
    
    Now suppose $c_1>1$. 
    Then $Q_{c_1}\equiv\omega_X \pmod 2$ and $Q_c\nequiv\omega_X\pmod 2$ for all $1\le c<c_1$.
    It follows that $p_{c_1-1}$ must be odd and so $\lambda_{c_1-1} \equiv \omega_{X^\vee}\pmod 2$ as required.
    
    If $c_1=1$ then as argued above, we must have $X=B$.
    Thus $\lambda_{c_1-1} = \lambda_0 = 0 \equiv \omega_{C}\pmod 2$ as required.
    
\end{proof}

\begin{lemma}\label{lem:technicallemma2}
    $\alpha(Q_i+1)+1=\alpha(Q_i)$ for all $1\le i\le l$ such that $\lambda_i,\lambda_{i-1}\equiv\omega_{X^\vee}\pmod 2$.
\end{lemma}

\begin{proof}


Let $0\le e_1 < e_2 < \cdots < e_s \le l$ be the indices $e$ such that $\lambda_e\equiv \omega_{X^\vee}\pmod 2$ (note that $e_1=0$ only in type $B$).
Let $i>0$ be such that $e_{i} = e_{i-1}+1$.
Write $e$ for $e_i$.
We must show that $\overline\alpha_1(Q_{e}+1) + 1 = \overline\alpha_1(Q_{e})$.
Let $j$ be such that $c_{j}\le e< c_{j+1}$.
We have that
\begin{align}
    \label{eq:intervalpos1}
    \overline\alpha_1(Q_e) &= B_j(Q_{e}-Q_{c_{j+1}})\\
    \overline\alpha_1(Q_e+1) &= \begin{cases}
        B_j(Q_{e}-Q_{c_{j+1}}+1) & \mbox{ if } c_j<e \\
        B_{j-1}(1) & \mbox{ if } c_j=e.
    \end{cases}
    \label{eq:intervalpos2}
\end{align}
Moreover, since $0\le e_{i-1}$ we must have $1\le e$ and so $c_0<e$.
Thus $e=c_j$ implies that $j>0$.

We consider two cases: $e=c_j$ and $e>c_j$.
Suppose first that $e = c_{j}$ (and hence $j>0$).
Then (in all cases)
\begin{equation}
    \overline\alpha_1(Q_{e}+1) = B_{j-1}(1) = P_{d_{j-1}}'+H_{j-1}'+t_{j-1}'
\end{equation}
and
\begin{equation}
    \overline\alpha_1(Q_{e}) = B_j(Q_e-Q_{c_{j+1}}) = P_{d_{j}}'+H_{j}'+\chi(H_j).
\end{equation}
Since $j>0$, we have $P_{d_{j-1}}'+t_{j-1}' = P_{d_{j}}'$.
Also $\eta_{c_{j}} \equiv \lambda_{c_{j}}-\lambda_{c_{j}-1} = \lambda_{e}-\lambda_{e-1}\equiv 0 \pmod 2$ and so $\eta_{c_j} = 2$ and $H_j'=H_{j-1}'+1$.
Moreover as $\lambda_{c_j-1} = \lambda_{e_{i-1}}\equiv \omega_{X^\vee}\pmod 2$ we have that $\Delta_{j-1}=\omega_{X^\vee}$ and so $H_{j-1} \equiv \Delta_{j-1}-\Delta_0 = 0$.
Thus $\chi(H_j) = 0$.
So we have that $\overline\alpha(E;p')(Q_e+1) = \overline\alpha(E;p')(Q_e)-1$ as required.

For the second case we suppose $e > c_j$.
Then $c_j\ll c_{j+1}$.
Thus as in Lemma \ref{lem:endparity} we have that $\Delta_j = \omega_{X^\vee}$.
Note this holds for $j=0,r$ too.
It follows that $H_j$ is even and so $\chi(H_j) = 0$.
Therefore, by inspecting the expressions for the a-symbol in the various cases, we have that $B_j(x) = B_j(x+1)+1$ whenever $x\equiv\sigma_j$ where $\sigma_i = 1$ for $0\le i<r$ and $\sigma_r = 1-\omega_{X}$.
It is thus sufficient to check that $Q_{e}-Q_{c_{j+1}}\equiv \sigma_j\pmod 2$ for all $0\le j\le r$.
But since $c_j<e<c_{j+1}$, we must have that $Q_e\nequiv\omega_X\pmod 2$.
Since $Q_{c_{j+1}}$ has the opposite parity of $Q_e$ if $j<r$ and $=0$ if $j=r$ we have that $Q_e-Q_{c_{j+1}}\equiv 1\pmod 2$ for $j<r$ and $\equiv \sigma_r\pmod 2$ if $j=r$ as required.
\end{proof}

Recall the pair $(J,\phi) \in \mathscr{F}_{\tilde{\Delta}}$ constructed in (\ref{eq:Jphi}). By Lemma \ref{lem:conditioniclassical}, $(J,\phi)$ satisfies condition (i) of Definition \ref{def:faithful}. We will now show that is satisfies condition (ii) as well. This will complete the proof of Theorem \ref{thm:faithful} for classical types.

\begin{lemma}
    If $E$ is any irreducible representation of $W$ with $\OO^\vee(E) = \OO^{\vee}$, there is a representation $F \in \phi \otimes \mathrm{sgn}$ such that
     $\Hom(F, E|_{W_J})\ne0$.
\end{lemma}

\begin{proof}
    Let $E_0 = E(\OO^\vee,1)$ and $F_0 = \mathrm{sp}(\phi)$.
    By Lemma \ref{lem:conditioniclassical} we have that $\overline{\mathbb L}(J,\OO(\phi)) = d_A(\OO^\vee,1)$.
    Thus by Lemma \ref{lem:technicallemma1} and Frobenius reciprocity we have $\Hom(F_0,E|_{W_J})\ne0$.
    But $F_0 = G_1\otimes G_2$ where $G_1=E(d_{LS}(\mu),1)$ and $G_2 = E(d_{LS}(\nu),1)$.
    Let $\Lambda_0' := \Lambda(E_0)$, $\alpha_1' = \alpha(G_1)$ and $\alpha_2 := \alpha(G_2)$.
    If $X=D$ let $\kappa$ be the decoration for $\tilde\alpha_D^\bullet(G_2)$.
    Let $k>\lfloor \#\lambda/2\rfloor$ be sufficiently large so that there exists $\Lambda_0=(a;b) \in [\Lambda_0']\cap \Lambda_{X^\vee}(n;k)$, $\alpha_1''\in [\alpha_1']\cap \alpha_{Y(X)}(n;k)$ and $\alpha_2\in [\alpha_2']\cap \alpha_X(n;k)$.
    Let $\alpha_1 = \alpha''^!$ if $X=B$ and $\alpha_1 = \alpha_1''$ otherwise.
    Then by Proposition \ref{prop:jinduction}, $\Lambda_0 = \alpha_1 + \alpha_2$.
    Let $\bar\Lambda = (I_1,\dots,I_k)$ be the refinement of $\bar\Lambda_0$ and $a=(A_1,\dots,A_k),b=(B_1,\dots,B_k)$ be the refinements of $a$ and $b$.
    Let $\alpha_i = (a^i;b^i)$ and $a^i = (A_1^i,\dots,A_k^i),b^i = (B_1^i,\dots,B_k^i)$ be the corresponding decomposition of $\alpha_i$ for $i\in\{1,2\}$.
    
    Now let $E\in\mathrm{Irr}(W)$ be any representation with $\OO^\vee(E) = \OO^\vee$ and let $\Lambda \in [\Lambda(E)]\cap \Lambda_X(n;k)$ ($\ne \emptyset$ by Equation \ref{eq:familybijection}).
    Then by Lemma \ref{lem:familyflips}, $\Lambda = ((X_1,\dots,X_k);(Y_1,\dots,Y_k))$ where $\{A_i,B_i\} = \{X_i,Y_i\}$ for all $1\le i\le k$.
    For $i\in\{1,2\}$, let $\beta_i = ((X_1^i,\dots,X_k^i);(Y_1^i,\dots,Y_k^i))$ where 
    \begin{equation}
        (X_j^i,Y_j^i) = \begin{cases}
            (A_j^i,B_j^i) & \mbox{ if } (X_j,Y_j)=(A_j,B_j) \\
            (B_j^i,A_j^i) & \mbox{ if } (X_j,Y_j)=(B_j,A_j).
        \end{cases}
    \end{equation}
    By construction, $\beta_1$ and $\beta_2$ have the same defect and number of entries as $\Lambda$ and $\Lambda = \beta_1+\beta_2$.
    By Lemma \ref{lem:technicallemma2} we can apply Lemma \ref{lem:jump1} and so both $\beta_1,\beta_2$ are a-symbols of type $X$.
    Note that when $X=B$, since $k>\lfloor \#\lambda/2\rfloor$, by Lemma \ref{lem:technicallemma1}, $A_1(\#A_1) = 0$ which implies that $(X_1,Y_1)=(A_1,B_1)$ (since $\Lambda$ is an s-symbol of type $C$).
    By the expression for $\alpha_1$ given in Paragraph \ref{par:typeB}, we also have $A_1^1(\#A_1^1)=0,B_1^1(\#B_1^1) = 1$ and so $\beta_1 \in \alpha_D^1(n;k)$ and so there is a $\gamma_1\in \alpha_D(n;l)$ such that $\gamma_1^! = \beta_1$.
    Since $\mu$ is not very even for $X\in\{B,D\}$, there is a unique representation $F_1\in\mathrm{Irr}(W_{Y(X)}(|\mu|))$ such that 
    \begin{align*}
        X=B: &[\tilde\alpha(F_1)] = [\{\gamma_1\}] \\
        X=C: &[\alpha(F_1)] = [\beta_1] \\
        X=D: &[\tilde\alpha(F_1)] = [\{\beta_1\}].
    \end{align*}
    Let $F_2\in \mathrm{Irr}(W_X(|\nu|))$ be the unique representation such that
    \begin{align*}
        X&=B,C: [\alpha(F_2)] = [\beta_2] \\
        X&=D: [\tilde\alpha^\bullet(F_2)] = [\{\beta_2\}^\kappa].
    \end{align*}
    Let $F = F_1\otimes F_2$.
    Then, since $\beta_i\sim\alpha_i$ (and $\gamma_1\sim\alpha_1$ when $X=B$) we have that $F_i\sim G_i$, and so $F\sim F_0$.
    This implies $F\in \phi\otimes\mathrm{sgn}$.
    Finally, since $\Lambda = \beta_1+\beta_2$, we have by Proposition \ref{prop:sumtohom} that $\Hom(F,E|_{W_J})\ne0$ as required.
\end{proof}

\subsection{Proof of faithfulness in exceptional types}\label{subsec:exceptional}
Let $\fg$ be a simple exceptional Lie algebra and let $\OO^{\vee} \subset \fg^{\vee}$ be a nilpotent orbit. If $A(\OO^\vee)$ is trivial, then $\OO^{\vee}$ is faithful by Proposition \ref{prop:easycase}. If $A(\OO^{\vee})$ is nontrivial, we use GAP to exhibit an explicit pair $(J,\phi) \in \mathscr{F}_{\tilde{\Delta}}$ satisfying conditions (i) and (ii) of Definition \ref{def:faithful}. In many cases, we may take $(J,\phi) = (\Delta,\phi(E(\OO^\vee,1)))$ (indeed this can be done precisely when $\OO^\vee$ is special and all the representations $E\in\mathrm{Irr}(W)$ with $\OO^\vee(E) = \OO^\vee$ lie in the same family as $E(\OO^\vee,1)$). However, in some cases, a less obvious choice is required. These (less obvious) cases are listed in Tables \ref{table:1}, \ref{table:2}, and \ref{table:3}.
Note that in these tables we list $(J,\OO(\phi))$ instead of $(J,\phi)$, but the map $\phi\mapsto \OO(\phi)$ induces a bijection between families and special nilpotent orbits, so $\phi$ can be recovered from the information provided.
Note also that there are no tables for $G_2$ or $E_6$---in these cases, there are no exceptions to consider.

\begin{table}[H]
    \begin{tabular}{ |c||c|c|c|  }
    \hline
    $\OO^\vee$ & $J$ & Type & $\OO(\phi,\CC)$\\
    \hline
    $A_2$ & \dynkin[extended,labels={\times,\times,\times,\times,}] F4 & $B_4$ & $(711)$ \\
    $B_2$ & \dynkin[extended,labels={\times,\times,\times,\times,}] F4 & $B_4$ & $(531)$ \\
    $C_3(a_1)$ & \dynkin[extended,labels={\times,,\times,\times,\times}] F4 & $A_1+C_3$ & $(2)\times (42)$ \\
    $F_4(a_2)$ & \dynkin[extended,labels={\times,\times,\times,\times,}] F4 & $B_4$ & $(32211)$ \\
    \hline
    \end{tabular} 
    \caption{$(J,\phi)$ for $\fg = F_4$}\label{table:1}
\end{table}

\begin{table}[H]
    \begin{tabular}{ |c||c|c|c|  }
    \hline
    $\OO^\vee$ & $J$ & Type & $\OO(\phi,\CC)$\\
    \hline
    $A_3+A_2$ & \dynkin[extended,labels={\times,\times,\times,\times,\times,\times,,\times}] E7 & $D_6+A_1$ & $(7311) \times (2)$ \\
    $E_7(a_4)$ & \dynkin[extended,labels={\times,\times,\times,\times,\times,\times,,\times}] E7 & $D_6+A_1$ & $(332211) \times (2)$ \\
    \hline
    \end{tabular}    
    \caption{$(J,\phi)$ for $\fg =E_7$}\label{table:2}
\end{table}

 \begin{longtable} { |c||c|c|c|  }
    \hline
    $\OO^\vee$ & $J$ & Type & $\OO(\phi,\CC)$\\
    \hline
    $A_3+A_2$ & \dynkin[extended,labels={\times,,\times,\times,\times,\times,\times,\times,\times}] E8 & $D_8$ & $(11,3,1,1)$ \\
    $D_4+A_2$ & \dynkin[extended,labels={\times,,\times,\times,\times,\times,\times,\times,\times}] E8 & $D_8$ & $(7711)$ \\
    $D_6(a_2)$ & \dynkin[extended,labels={\times,,\times,\times,\times,\times,\times,\times,\times}] E8 & $D_8$ & $(7531)$ \\
    $E_6(a_3)+A_1$ & \dynkin[extended,labels={\times,\times,\times,\times,\times,\times,\times,,\times}] E8 & $E_6+A_2$ & $E_6(a_3) \times (3)$ \\
    $E_7(a_5)$ & \dynkin[extended,labels={\times,\times,\times,\times,\times,\times,\times,\times,}] E8 & $E_7+A_1$ & $E_7(a_5) \times (2)$ \\
    $E_7(a_4)$ & \dynkin[extended,labels={\times,,\times,\times,\times,\times,\times,\times,\times}] E8 & $D_8$ & $(732211)$ \\
    $D_5+A_2$ & \dynkin[extended,labels={\times,,\times,\times,\times,\times,\times,\times,\times}] E8 & $D_8$ & $(5533)$ \\
    $E_8(b_6)$ & \dynkin[extended,labels={\times,\times,\times,\times,\times,\times,\times,,\times}] E8 & $E_6+A_2$ & $D_4(a_1) \times (3)$ \\
    $D_7(a_1)$ & \dynkin[extended,labels={\times,,\times,\times,\times,\times,\times,\times,\times}] E8 & $D_8$ & $(443311)$ \\
    $E_8(b_4)$ & \dynkin[extended,labels={\times,,\times,\times,\times,\times,\times,\times,\times}] E8 & $D_8$ & $(33222211)$ \\
    \hline
    \caption{$(J,\phi)$ for $\fg = E_8$} \label{table:3}
    \end{longtable}

\begin{sloppypar} \printbibliography[title={References}] \end{sloppypar}

\end{document}